\documentclass[11pt,a4paper]{amsart}
\usepackage{amsmath,amsthm,amsfonts,latexsym,amssymb,enumerate,amscd,epsfig}
\usepackage{musicography}
\usepackage{hyperref}
\input{xypic}

\usepackage{tikz}
\usetikzlibrary{backgrounds,calc}
\usetikzlibrary{calc,arrows}

\usepackage{fullpage}
\usepackage{graphicx}
\usepackage[all]{xy}
\usepackage{verbatim}
\usepackage{mathtools}
 \usepackage{color}
 
\usepackage{lineno} 
\usepackage{amsopn}
\usepackage{enumitem}

\DeclareMathOperator{\trace}{trace}
\DeclareMathOperator{\grad}{grad}
\DeclareMathOperator{\Exp}{Exp}
\DeclareMathOperator{\diam}{diam}
\DeclareMathOperator{\dist}{dist}
\DeclareMathOperator{\supp}{supp}
\DeclareMathOperator{\vol}{vol}
\DeclareMathOperator{\inj}{inj}

\DeclareMathOperator{\g}{grad}
\DeclareMathOperator{\ad}{ad}
\DeclareMathOperator{\sym}{sym}
\newcommand{\Hper}{{H^{1}_{per}}}
\newcommand{\mathcalH}{\mathcal{H}}
\newcommand{\mathcalV}{\mathcal{V}}
\newcommand{\calV}{\mathcal{V}}
\newcommand{\vectors}{\mathfrak{X}}
\newcommand{\Vvectors}{ \mathfrak{X}(\mathcal{V} )}
\newcommand{\V}{\mathcal V}

\def \p{\partial}

\newcommand{\ip}[2]{\left\langle #1 ,#2\right\rangle}
\newcommand{\depa}[2]{\frac{\p #1}{\p #2}}

\def\Xint#1{\mathchoice
   {\XXint\displaystyle\textstyle{#1}}%
   {\XXint\textstyle\scriptstyle{#1}}%
   {\XXint\scriptstyle\scriptscriptstyle{#1}}%
   {\XXint\scriptscriptstyle\scriptscriptstyle{#1}}%
   \!\int}
\def\XXint#1#2#3{{\setbox0=\hbox{$#1{#2#3}{\int}$}
     \vcenter{\hbox{$#2#3$}}\kern-.5\wd0}}
\def\ddashint{\Xint=}
\def\dashint{\Xint-}

 \newcommand{\avint}{\dashint}
 
\begin{document}

\title{Homogenization on parallelizable Riemannian manifolds}

\date{\today}

%\thanks{
%Mathematics Subject Classification. Primary 35P25, 45Q05; Secondary  42B37,     35J10}
\thanks{D.Faraco, L.Guijarro and A.Ruiz acknowledge financial support from the Spanish Ministry of Science and Innovation through the Severo Ochoa Program for Centers of Excellence in R\&D (CEX2019-000904-S) and by the PID2021-124195NB-C32, from the Spanish Ministry of Economy and Competitiveness through research program MTM2017-85934-C3-2-P2, from  CAM through the Line of Excellence for University Teaching Staff between CAM and UAM,  and from  ERC Advanced Grant 834728. Y.Kurylev was partially funded by EP/R002207/1 (Nonlinear geometric inverse problems) funded by EPSRC}

\author{Daniel Faraco}
\address{Departamento de Matem\'aticas - Universidad Aut\'onoma de Madrid
 and Instituto de Ciencias Matem\'aticas
CSIC-UAM-UC3M-UCM, 28049 Madrid, Spain} \email{daniel.faraco@uam.es}
\author{Luis Guijarro}
\address{Departamento de Matem\'aticas - Universidad Aut\'onoma de Madrid
 and Instituto de Ciencias Matem\'aticas
CSIC-UAM-UC3M-UCM, 28049 Madrid, Spain} \email{luis.guijarro@uam.es}
\author{Yaroslav Kurylev }
\address{Department of Mathematics, University College London, Gower
Street, London UK, WC1E 6BT
} 
\author{Alberto Ruiz}
\address{Departamento de Matem\'aticas - Universidad Aut\'onoma de Madrid
 and Instituto de Ciencias Matem\'aticas
CSIC-UAM-UC3M-UCM, 28049 Madrid, Spain} \email{alberto.ruiz@uam.es}
\dedicatory{Dedicated to the memory of  our colleague and friend Slava Kurylev}

{\allowdisplaybreaks \sloppy \theoremstyle{plain}}
\newtheorem{Theorem}{Theorem}[section]
\newtheorem{Lemma}[Theorem]{Lemma}
\newtheorem{lem}[Theorem]{Lemma}
\newtheorem{Cor}[Theorem]{Corollary}
\newtheorem{question}[Theorem]{Question}
\newtheorem{Prop}[Theorem]{Proposition}

\theoremstyle{remark}
\newtheorem*{Rem}{Remark}

\theoremstyle{definition}
\newtheorem{defn}[Theorem]{Definition}
\newtheorem{Def}[Theorem]{Definition}
\newtheorem{example}[Theorem]{Example}
%\newtheorem{Rem}[Theorem]{}
%\noindent \emph{Mathematics Subject Classification (2000):} 35R30,
%35J15, 30C62.

%\newtheorem{Cor}{Corollary}
%\newtheorem{Def}{Definition}[section]
\newtheorem{Prob}{Problem}
\numberwithin{equation}{section}
\def\halmos{{\ \vbox{\hrule\hbox{\vrule height1.3ex\hskip0.8ex\vrule}\hrule}}\par \medskip}
\def\Xint#1{\mathchoice
{\XXint\displaystyle\textstyle{#1}}%
{\XXint\textstyle\scriptstyle{#1}}%
{\XXint\scriptstyle\scriptscriptstyle{#1}}%
{\XXint\scriptscriptstyle\scriptscriptstyle{#1}}%
\!\int}
\def\XXint#1#2#3{{\setbox0=\hbox{$#1{#2#3}{\int}$}
\vcenter{\hbox{$#2#3$}}\kern-.5\wd0}}
\def\ddashint{\Xint=}
\def\dashint{\Xint-}
%\dashint gives a single-dashed integral sign, \ddashint a double-dashed one.
%%%%%%%%%%%%%%%%%%%%%%%%%%%%%%%%%%%%%%%%%%%%%%%%%%%%%%%%%%%%%%%%%%555
%Abreviations
\def \t{\tilde}
\def \r{\tilde{R}}
%%%%%%%%%%%%%%%%%%%%%%%%%%%%%%%%%%%%%%%%%%%%%%%%%%%%%%%%%%%%%%%%555
%Domains
\def \D{\mathbb{D}}
\def \p{\partial}
\def \div{\operatorname{div}}
\def \Re{\operatorname{Re}}
\def \Im{\operatorname{Im}}
\def \supp{\operatorname{supp}}
\def \diam{\operatorname{diam}}
\def \C{\mathbb{C}}
\def \N{\mathbb N}
\def \Z{\mathbb Z}
\def\Lip{\operatorname{Lip}}
\newcommand{\R}{\mathbb{R}}
\newcommand{\T}{\mathbb{T}}
\def\TM{\mathbb{T}M}
\newcommand{\VM}{\mathfrak{X}(\mathbb{T}M)^{\text{ver}}}
\def  \TVkm{C(M,L^2T^{k,m}(\mathcal{V}_p))}
\def  \TV11{C(M,L^2T^{1,1}(\mathcal{V}_p))}
\def  \TVX{C(M,L^2\mathfrak{X}(\mathcal{V}_p))}
\def\up{\uparrow}
\def\ver{\textrm{ver}}

\def\widechi{\widetilde{\psi}}
\def\widef{\widetilde{f}}
\def\tildeX{\widetilde{X}}

\newcommand{\mathcalT}{\mathcal{T}}

\newcommand{\mathcalN}{\mathcal{N}}

\renewcommand{\a}{\alpha}
\newcommand{\e}{\varepsilon}

\newcommand{\RR}{\mathbb{R}}
\newcommand{\TT}{\mathbb{T}}
\newcommand{\ZZ}{\mathbb{Z}}

\newcommand{\tildeA}{\widetilde{A}}
\newcommand{\tildef}{\widetilde{f}}

\newcommand{\hh}{\bar{h}}
%%%%%%%%%%%%%%%%%%%%%%%%%%%%%%%%%%%%%%%%%%%%%%%%%%%%%%%%%%%%%%%%%%%%%%%%5555
%Operators
\def\2L{\Lambda_{\tilde{\gamma}}}
\def\1L{\Lambda_{\gamma}}
\def \beq {\begin {eqnarray*} }
\def \eeq {\end {eqnarray*} }
\def \e{\epsilon}

\newcommand{\inn}[2]{\ensuremath{\langle #1,#2\rangle}}

%\DeclareMathOperator{\Iso}{Iso}

%%%%%%%%%%%%%%%%%%%%%%%%%%%%%%%%%%%%%%%%%%%%%%%%%%%%%%%%%%%%%%%%%%%%%%555
%Derivatives
\def \c{\overline}
\def \d{\partial_z}
\def \dc{\partial_{\overline z}}
\def \cd{\overline{\partial_z}}
\def \dk{\partial_{\overline k}}
\def \kd{\overline{\partial_k}}
\def \dx{\partial_x}
%%%%%%%%%%%%%%%%%%%%%%%%%%%%%%%%%%%%%%%%%%%%%55555
%Spaces
\def \H{H^{1/2}}
%%%%%Musical
\def \s{{\musSharp}}
\def\2s{\xrightharpoonup{2s}}
\def\S2s{\stackrel{2s}{\longrightarrow}}

\begin{abstract} 
We  consider the problem of finding the homogenization limit of oscillating linear elliptic equations  in an arbitrary
parallelizable manifold $(M,g,\Gamma)$. We replicate  the concept of two-scale convergence by pulling back tensors 
defined on the torus bundle $\mathbb{T}M$ to the manifold $M$. The process consists of two steps: localization in the slow variable through Voronoi domains, and inducing local 
periodicity in the fast variable from the local exponential map in combination with the geometry of the torus bundle. The procedure  yields explicit cell formulae
for the homogenization limit, and as a byproduct a theory of two-scale convergence of tensors of arbitrary order.
\end{abstract}

\maketitle


%\tableofcontents

\section{Introduction}

The mathematical model known as the theory of periodic homogenization is used to describe the overall behavior of microscopic structures. Its roots can be traced back to the concept of $G$-convergence introduced by Spagnolo, generalized later by the $H$-convergence of Murat and Tartar, and it has since become a fundamental concept in the field of partial differential equations over the past four decades.
The archetypal example comes from considering a unit cube $Y$ and a bounded domain $\Omega$  in $\mathbb{R}^n$, a $Y$-periodic function $\sigma:\Omega\times Y\to\R$, and the corresponding differential operators $A_\epsilon:=\textrm{div}(\sigma(x, \frac{x}{\epsilon})\nabla)$. As $\epsilon\to 0$,  the solutions $u_\epsilon$ to the corresponding elliptic problem 
$A_\epsilon u_\epsilon= f$
 converge to some function $u$ that solves the limit equation $\textrm{div}(\sigma^*\nabla u)=f$,
where $\sigma^*$ is given by the  classical cell formula

\[ 
\sigma^*_{ij}(x)=\int_{Y} \sigma (x,y) [\nabla_y w_i(x,y)+e_i,\nabla_y w_j(x,y)+e_j]\, dy.
\]
Here $w_i(x,y)$ is the $Y$--periodic solution to the cell problem 
\[ -\div_y (\sigma(x,y)[\nabla_y w_i(p,v)+e_i ])=0\]
See for example the monographs \cite{Allairebook, BLPbook, Braidesbook, BraidesDefranceschibook, DalMasobook, JKObook, Miltonbook, Tartarbook} for basic reference, although the number of variants and applications is enormous. 

The aim of our research is to investigate from an intrinsic viewpoint such oscillating PDE's in a general Riemannian manifold $(M,g)$, and more importantly, to obtain explicit cell formulae which are amenable to computations. Notice that  a frontal obstacle for this is the lack of the concept of periodicity in a general Riemannian manifold. Partial advances in this direction appear in \cite{ConItuSic}, where the authors  put conditions on the fundamental group of the manifold that allows them to homogenize on certain Abelian covers of the original manifold. Interestingly, the probabilistic approach to homogenization (see e.g \cite{BLP79}) has been adapted to  manifolds with ergodic geodesic flow by Pak \cite{Pak05}. This is done using the tangent space at one fixed point and mapping it to the manifold with the exponential map. 

Our starting point is the observation that tangent spaces can be always endowed with the symmetries of Euclidean geometry. Interestingly, this naturally leads to consider models where at each point in the manifold the symmetries might differ.  The notion of parallelizable manifold seems tailored to describe this situation. An $n$--dimensional manifold $(M,g)$ is \emph{parallelizable} if  there is a smooth non-degenerate frame $\Gamma(p)=\{e_i(p)\}^n_{i=1}$ of vector fields  defined in the whole manifold; in more geometrical terms, this can also be formulated as the tangent bundle of the manifold being trivial.  This condition can be topologically restrictive (for instance, it forces the vanishing of every characteristic class of $M$) , but it allows nonetheless for numerous examples that were not considered in previous approaches. For instance,  
all closed 3--dimensional  manifolds and all Lie groups are parallelizable. Moreover, even if $M$ is not parallelizable, we could always work with subdomains where a frame does not degenerate.  We also think that parallelizability is the first step that needs to be taken in order to handle homogenization for Riemannian manifolds \emph{without any restrictions}. We expect to complete this project in future work.

All in all, the notion of parallelizable manifold allows us to formalize our previous intuition.  Even if $M$ does not admit symmetries, the tangent plane $T_pM$ admits those given by the lattice $\Gamma(p)$. Thus, it is natural to consider functions defined on the tangent bundle of $M$ which are $\Gamma(p)$ periodic on the second variable.  The best way to codify periodicity is to work on the so-called torus bundle $\TM$ (see Section 2 for precise definitions). For variables, $[p,v] \in \TM$, the manifold variable $p$ will play the role  of a slow variable, which we handle by discretization.
 The variable $v$ in the torus fiber will play the role of the periodic fast variable and will be handled by a local exponential map.
Such idea leads  to  a manifold version of   the by now  classical  two--scale convergence \cite{Allaire92,Nguetseng89}. In fact, we recover the classical two scale convergence theory  in the Euclidean space as a particular case. Let us remark  that furthermore, our result is new even in the Euclidean case, as we are allowing the periodicity to vary smoothly from point
to point.
(See~Remark \ref{rem:old2scale}).

Indeed, we will obtain a satisfactory two-scale convergence theory on manifolds, but we believe  that 
our main contribution is to obtain explicit formulae for homogenization limits of rapidly oscillating equations on the manifold.  We will use the standard terminology on the analysis of the torus bundle $\TM$ which will be carefully revisited in Section 2. For example $\V$ stands for the  vertical  part of the tangent to the  torus bundle and therefore $\mathfrak{X}(\V)$ for vertical vector fields (tangent to the fibers of $\TM$) and  correspondingly $T^{m,n}(\V)$ for vertical $(m,n)$-tensors.  
  
   In order to define  oscillating tensors in the spirit of the Euclidean case  $\sigma(x,\frac{x}{\epsilon})$, we will bring the periodic frame structure of the tangent plane to the manifold via the exponential.  For a fine enough net of points in $M$, $p_1,\dots, p_N$, the diffeomorphisms  $H_{\epsilon,j}: M \to T_{p_j}M$,
 \[ H_{\epsilon,j}(q)= \frac{\exp_{p_j}^{-1}(q)}{\epsilon} \]
 will play a crucial role. \footnote{The domain of these maps is not the whole of $M$, but we will write it as so in order to facilitate the reading.}  In order to bring to $M$ a function $\tilde{f}$ defined on $\TM$ we declare
 \[ f^\epsilon(q)=\sum_j \psi_j \tilde{f}\left(p_j,{\exp_{p_j}^{-1}(q)}/{\epsilon}\right)=\sum_k \psi_j H_{\epsilon,j}^* \tilde{f}
 \]
 where $\psi_j$ is a suitably constructed partition of unity.

It turns out that one can pull back vertical tensors $A \in T^{m,n}(\V)$ of the torus bundle similarly. In particular for vertical $(1,1)$-tensors $A \in T^{1,1}(\V)$ with coordinates $A_{i}^{j}[p,v]$, 
we define a sequence of oscillating tensor$\bar{A}^\epsilon \in T^{1,1}(M)$ by pulling back its coordinates, that is

\begin{equation}\label{eq:defAbarepsilon}
(\bar{A}^\epsilon)_{i}^{j}= (A_i^j)^\e.
\end{equation}

Alternatively, as in the case of functions, one can use $H_{\epsilon,j}$ and
define
 \[ 
 (A^\epsilon(q)=\sum \psi_j(q) H_{\epsilon,j}^*(A), 
 \]
(see  section~\ref{sec:2stensors} for precise definitions of tensors in the torus bundle and their 
pullbacks).

Notice
that, even in $\mathbb{R}^n$, our definition circumvents the measurability issues which are nagging with the classical $\sigma(x,\frac{x}{\epsilon})$. In any case,
both oscillating tensors $A^\epsilon, \bar{A}^\epsilon$ (more precisely their symmetric parts) are homogenized to an explicit  tensor $A^*$. The tensor $A^* \in  T^{1,1}(M)$ is obtained from $A \in T^{1,1}(\V)$
by explicit cell formulae. In the case that the metric $g$ and the frame $\Gamma$ are linked, in the sense that $\Gamma$ is an orthonormal frame respect to $g$, the cell formula
parallels the Euclidean situation discussed earlier. We call such metric the frame metric.
\begin{Def}[The homogenized problem - frame metric]
\label{def:hom}
Given $A \in T^{1,1}(\V)$, we define  $A^* \in T^{1,1}(M)$ by stating that for any $e_i,e_j \in T_pM$
\[ A^*(p)[e_i,e_j]=\int_{\TM_p} A[p,v] [\grad_v w_i(p,v)+e_i^\up, \grad_v w_j(p,v)+e_j^\up]\, dv\]
Here $w_i(p,v) \in L^2(M,\Hper(\TM_p))$ are the $\Gamma(p)$-periodic solutions to the cell problems
\[ -\div_v (A(p,v)[\grad_v w_i(p,v)+e_i ^\up])=0\]

\end{Def}
The  vector  $e_j^\up[p,v] $ is a vertical vector, obtained by lifting $e_j$ to the fiber (See definition \ref{def:vertical}). In  the case that the metric and the parallelization are arbitrary, 
the homogenized problem  incorporates the metric and its formula is given in equation \eqref{homatrix}

All in all, we obtain the following clean  homogenization theorem on manifolds. 
\begin{Theorem} 
\label{thm:hom}
Let $\Omega$ be a bounded smooth domain in $M$. 
Given  $f \in H^{-1}(M)$, the unique solution $u_\epsilon \in H_0^{1}(\Omega)$ of the problem
\[  \div (A_{\sym}^\epsilon[\grad u_\epsilon])=f \]
converges weakly to $u^* \in H_0^1(\Omega)$, defined as  the unique solution to 
\[  \div (A^*[\grad u^*])=f \]
\end{Theorem}
The same homogenization result holds if we replace $A^\epsilon$ by $\bar{A}^\epsilon$ as defined in \eqref{eq:defAbarepsilon}.
%consider the equations given by  $\bar{A}^\epsilon$ as defined in \eqref{eq:defAbarepsilon}  instead of $A^\epsilon$. 
As a matter of fact,  being technical for a moment, as soon as we have an
arbitrary sequence $A_\epsilon \in T^{1,1}(M)$ which two scale converges in the strong sense to $A \in T^{1,1}(\V)$ (c.f section~\ref{sec:2stensors} for definition of two strong convergence of tensor fields) the same homogenization result holds.  In fact, the homogenization  theorem is a corollary of the corresponding  two scale homogenization Theorem \ref{2stheorem}. Notice that, loosely speaking, strong two scale convergence  amounts to $\e$ being the scale of oscillation. An analogous theory of multiscale convergence  that deals with various scales of oscillations, and leading to reiterated homogenization is also possible, but we defer its development to future work. Similarly, our approach should work for nonlinear equations, but for brevity, we include here only the linear case.

The structure of the paper is as follows. Section 2 gives some background on parallelizable manifolds, and shows how to construct its associated torus bundle; the main idea is that its fibers will detect in the limit the periodicity phenomena, thus resulting in the right setting where two-scale limits abide. Notice that pulling back functions from the tangent bundle is relatively easy, but pulling back vector fields and forms is rather complicated, and it requires a careful understanding of the torus bundle and
its vertical fibers. In this section we introduce the vertical framing $\{e_1^\up, \ldots, e_n^\up\}$ which will be fundamental to understand two scale convergence of vector fields. 
The other key  idea to study oscillatory phenomena at a very small scale is to approximate the manifold by discrete versions. This is attained through metric $\e$-nets, 
$\{p_j \}$ and more precisely, through   
their Voronoi decompositions.  Associated to them, it is important to construct partitions of unity $\{\psi_j\}$ with a good control on the size of the gradients of the functions appearing in them. We do this in Section \ref{sec:voronoi}, with Theorem \ref{thm:usable_partition_of_unity} as the final statement that is used at latter points in the paper. We think that this construction can be of independent interest  and certainly useful in other contexts. For example, we need to control the angles between geodesics emanating from the points in the metric net (See lemma~\ref{lem:triangle}) and to describe the boundary of the Voronoi domains rather precisely (see lemma~\ref{lem:Voronoiboundary}). Voronoi domains also appeared  in the reference \cite{GirouardLagace21}, although their authors did not need such an accurate description of their geometry. 

Part 3 of the paper contains Sections 4 to 8, dedicated to introducing two-scale convergence of functions, vector fields, and differential forms. As explained before, in the spirit of \cite{Nguetseng89} and \cite{Allaire92}, 
given $\tilde{f}:\TM\to\R$, 
%we 
%given  $f \in \mathbb{T}M$  and $\tilde{f}$ the corresponding lift, 
we declare test functions
\[ f^\epsilon(q)=\sum_j \psi_j \tilde{f}\left(p_j,{\exp_{p_j}^{-1}(q)}/{\epsilon}\right)=\sum_k \psi_j H_{\epsilon,j}^* \tilde{f}\]
With such functions at hand, a suitable notion of two--scale convergence can be defined. Indeed, we can replicate some of the similar results  for functions in the Euclidean space. This takes most of Sections 4 and 5.  The careful choice
of the partition of unity allows us to replicate all results in the Euclidean setting, without problems of measurability. It also allows us to calibrate the Euclidean setting,  shifting the origin as desired. 
To make a  theory of two scale convergence of vector fields is more subtle. Morally, if in order to define test functions in $M$ we need to pull back functions defined in the torus bundle $\TM$ back to $M$ through the use of the maps $H_{\epsilon,j}$, for vector fields, we will push forward vertical vector fields tangent to the torus fibers of $\TM$ towards $M$.  This is convenient for integrating by parts, and at the end of the day equivalent to the two scale convergence of the coordinate functions. It takes some effort to prove  that all these alternative definitions of two scale convergence for vector fields are indeed equivalent, and 
this constitutes the bulk of Section \ref{sec:2stensors}. All in all, a solid and flexible  theory of two-scale convergence of vector fields is established there. As two scale convergence of two tensors is also needed to deal with elliptic equations we develop the theory  for arbitrary $(k,m)$-tensors, which could be of interest in other contexts, e.g. in linear elasticity.  We warn  the reader just interested in the application to elliptic equations that, for that purpose, most of the section  can be skipped. It suffices  with definition~\ref{def:2stensors} and sections \ref{sec:by_parts_main} and \ref{sec:2sdifferentials}.

In the Euclidean case, it holds that $\nabla f(x,\frac{x}{\epsilon})=(\nabla f_x+\frac{1}{\epsilon} \nabla f_y)(x,\frac {x}{\epsilon})$, namely, for the obvious framing, gradients of test functions and test vector fields in $M$ coming from the vertical gradient of a function agree. 
A similar relation in the manifold case holds only approximately as $\e\to 0$; see Theorem \ref{vector_limit}. This is, indeed, a rather difficult fact to establish, and its proof takes all of Sections \ref{sec:by_parts_technical} and \ref{sec:by_parts_main}.
Section~\ref{sec:2sdifferentials} examines the relation between convergence of functions and of their gradients. In addition to the integration by parts formula, we needed to establish a vertical Hodge decomposition type of argument.

The last part of the paper, part~\ref{part:equations}, deals with second order elliptic equations in manifolds. 
Section~\ref{sec:2selliptic} introduces carefully the classes of tensors that we are able to homogenize and gives the proof of the theorem~\ref{2stheorem}.
Finally, Section~\ref{sec:homogenization}  uses all the previous results in order to prove the Main Theorem \ref{thm:hom} (in the most general case of unrelated  frames and metrics) relating two scale convergence with homogenization. This section is parallel to the Euclidean one, but one needs to be careful in the Riemannian setting.  

To help the reader, the appendixes review some general facts of tangent bundles of manifolds, as well as discussing
well-posedness of the two scale homogenized problem. It also  reviews the weak and strong formulation of both homogenized problems. 

%\section{Notations and preliminaries}
We conclude this introduction by pointing out ideas for future research. From the manifold viewpoint, it would be interesting to investigate whether allowing a degeneracy set for the frame, we could deal with any Riemannian manifold (thus getting rid of the parallelizability restriction).  Similarly, our theory possibly extends to a reiterated homogenization. 
%and particular aspects should be incorporated. 
From the point of view of applications,  the ideas expressed here should yield various types of Darcy Law valid on manifolds with a heterogeneous distribution of pores, resulting in a point dependent permeability tensor.  For problems related to microscopic crystalline structures  in elasticity (e.g. \cite{FranckfortGiacomini14}), it might be  also important that the periodicity frame is allowed to  vary from point to point. From the viewpoint of partial differential equations, it will be desirable to make our homogenization results  quantitative alike in the Euclidean situation (see e.g \cite{KLS12, Shen22, Shenbook})

Finally, the original motivation of our research was to study the combination of collapse and homogenization in general geometries. Collapsing happens naturally when considering Riemannian manifolds with geometric conditions that guarantee Gromov-Hausdorff convergence to a lower dimension metric space. The more natural geometric conditions are bounds on the sectional  curvature, as was considered originally in the seminal work of Cheeger, Gromov and Fukaya in \cite{CG1}, \cite{CG2}, \cite{CFG}, or \cite{Fuk}. An excellent introduction to collapsing under these conditions is the survey paper of X.Rong \cite{Rong}.  
We are interested in this aspect from a Riemannian manifold view point, but this approach should also be interesting in dimension reduction theories in nonlinear elasticity, particularly the non-Euclidean case (see the recent available rigidity estimate \cite{CDMpreprint} in the line of Friesecke-James-Muller theory (\cite{BCHKpreprint, FJM02, FJMM03, FJM06,Lewickabook, LMP09, MoraMuller04,Mullerbook}). Consider, for example, the theory of non-Euclidean plates and the works combining homogenization with dimension reduction in elasticity (e.g \cite{HNV14, HV15}) e.g. where oscillating variational problems are considered.
These are all exciting lines of research that we leave for future work.

\textbf{Acknowledgments:}
The authors would like to thank Juan Luis V\'{a}zquez for general discussion about partial differential equations on manifolds and useful suggestions to improve the readability of the paper. They also thank Marcos Solera for interesting discussions.  We also thank the hospitality of University College of London for our various visits along the years. D.F acknowledges as well the hospitality and financial support of the Institute for Advanced Study in Princeton, where part of this research took place.

\part{Parallelizable manifolds}\label{part:parallelizable}

\section{Parallelizable manifolds}

\begin{defn}
We say that a differentiable manifold $M$ is \emph{parallelizable} if there exists a vector bundle  isomorphism $\Theta:M\times \mathbb{R}^n \to TM$. Such an isomorphism is called a parallelization. In other words, there are vector fields $X_1,\dots, X_n$ in $M$ (where $n=\dim M$) such that at each point $p\in M$, the vectors $X_{1,p},\dots,X_{n,p}$ form a basis of $T_pM$.
\end{defn}

 A parallelization provides a set $\Gamma=\Theta(M \times \mathbb{Z}^n)$ that becomes a linear lattice of $T_pM$ for each $p\in M$. Let $e_i \in {\mathcal X}(M)$ be the vector field defined as 
 \[
 e_i(p)=\Theta(p, {\bf e}_i), 
\qquad 
 {\bf e}_i=(0, \dots, 1,\dots, 0).
\]
  Notice that the collection of  possible lattices ${\Gamma} $ can be parametrized by  
smooth sections of the trivial bundle $M \times GL(n, \R)$. 
Since frames are non-degenerate and $M$ is compact, volumes of cells associated to two different lattices are comparable. 

Being parallel is, in general, a rather restrictive topological condition to set on a manifold. For instance, it forces every Pontryaguin number of the manifold to vanish (see \cite{MilSta}). Nonetheless, it allows for a rather large set of examples. For instance, every orientable three-dimensional manifold is parallelizable, every torus of any dimension,  as well as Lie groups and the 7-dimensional sphere \cite{Lee}.

\subsection{The torus bundle associated to a lattice}
\label{sec:torus_bundle}
Since periodicity with respect to a lattice $\Gamma$ is a key notion in our work, we will  consider the torus bundle
$\TM$ associated to a frame.
Let $(M,g)$ a parallelizable Riemannian manifold with lattice $\Gamma$ defined as above. There is a free, properly discontinuous action 
$\theta :\mathbb{Z}^n\times TM \to TM$
defined as 
\[
 \theta((k_1,k_2,\dots k_n),(p,v))=(p,v+\sum k_i e_i).
\] 
\begin{Def} \emph{The torus bundle associated to $\Gamma$} is the quotient manifold $\mathbb{T}M=TM/\mathbb{Z}^n$. 
\end{Def}

For a tangent vector $(p,v)\in TM$, we will denote by $[p,v]$ the corresponding point in $\TM$, and by $\Phi:TM\to\TM$ the quotient map sending $(p,v)$ to $[p,v]$. 
It is clear that $\Phi$ is a
 local diffeomorphism  that restricts to the covering map  $T_pM \to \TM_p\simeq\mathbb{T}^n$ for each $p\in M$. 
This can be summarized in the following diagram.

\[
\xymatrix{
M\times\R^n \ar[d] \ar[r]^{\Theta} &  TM \ar[d]^\Phi\\
M\times\T^n \ar[r] & \TM 
}
\qquad
\xymatrix{
(p,(x_1,\dots,x_n)) \ar[d]\ar[r] &  \left(p,\sum_{i=1}^n\,x_i\,e_i(p)\right) \ar[d]\\
(p,[x_1,\dots,x_n]) \ar[r] & [p, \sum_{i=1}^n\,x_i]
}
\]

Let $Q$ be an open cube in $\R^n$ with unit length sides.
If $(U,\phi)$ are coordinates  in $M$ and $\tilde{\phi}$ is the chart in $TM$ associated to the frame defined in \eqref{eq:chart_in_TM}, then in the open set
$
W:=\tilde{\phi}^{-1}(\mathbb{R}^n\times Q),
$ 
the quotient map $\Phi$ is injective, and we can define coordinates in $\TM$ as
$
\tilde{\phi}\circ (\Phi|_W)^{-1}.
$
By changing $Q$ we can obtain a full coordinate atlas of $\TM$. 

The projection $\pi:TM\to M$,  $(p,v)\to p$, induces a map $\pi:\TM\to M$, $[p,v]\to p$.
Observe that  with this map, $\mathbb{T}M$ is a fiber bundle over $M$ with fiber the torus $\mathbb{T}^n=\R^n/\Z^n$.

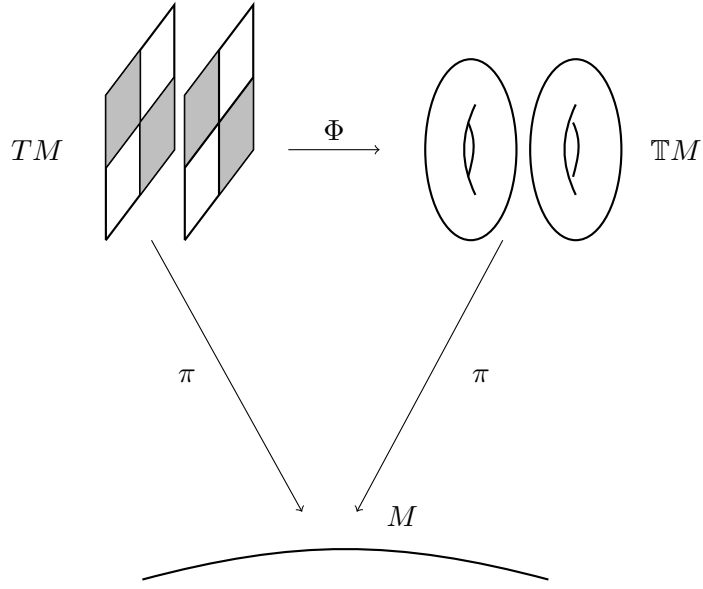
\begin{figure}
    \centering
        \begin{tikzpicture}[scale=0.6]
    \draw[thick, yscale=0.8] (-8,-2.5) -- (-8,1.5) -- (-6.5,4) -- (-6.5,0) -- (-8,-2.5);
    \draw[thick, yscale=0.8] (-8,-0.5) -- (-6.5,2);
    \draw[thick, yscale=0.8] (-7.25,-1.25) -- (-7.25,2.75);
    \draw[fill=lightgray, yscale=0.8] (-8,-0.5) -- (-8,1.5) -- (-7.25,2.75) -- (-7.25,0.75) -- (-8,-0.5);
    \draw[thick, yscale=0.8,xshift=1.73cm, yshift=0cm] (-8,-2.5) -- (-8,1.5) -- (-6.5,4) -- (-6.5,0) -- (-8,-2.5);
       \draw[fill=lightgray, yscale=0.8,xshift=1.73cm, yshift=0cm] (-8,-0.5) -- (-8,1.5) -- (-7.25,2.75) -- (-7.25,0.75) -- (-8,-0.5);
       \draw[fill=lightgray, yscale=0.8] (-7.25,-1.25) -- (-7.25,0.75) -- (-6.5,2) -- (-6.5,0) -- (-7.25,-1.25);
       \draw[fill=lightgray, yscale=0.8,xshift=1.73cm, yshift=0cm] (-7.25,-1.25) -- (-7.25,0.75) -- (-6.5,2) -- (-6.5,0) -- (-7.25,-1.25);    
    
    \draw[thick, yscale=0.8,xshift=1.73cm, yshift=0cm] (-8,-0.5) -- (-6.5,2);
    \draw[thick, yscale=0.8,xshift=1.73cm, yshift=0cm] (-7.25,-1.25) -- (-7.25,2.75);

    \draw[->] (-7,-2) -- (-3.7,-8);
    \node[left] at (-5.8,-5) {$\pi$};
     \draw[->] (-4,0) -- (-2,0);
     \node[above] at (-3,0) {$\Phi$};
    
    \draw[thick] (0,0) ellipse (1cm and 2cm);
    \draw[thick] (0.1,-1)  to [out=115, in=245] (0.1,1);
    \draw[thick] (-0.06,-0.6)  to [out=75, in=295] (-0.06,0.6);
        \draw[thick] (2.3,0) ellipse (1cm and 2cm);
    \draw[thick] (2.3,-1)  to [out=115, in=245] (2.3,1);
    \draw[thick] (2.24,-0.6)  to [out=75, in=295] (2.24,0.6);
%    \end{scope}
    \draw[->] (0.7,-2) -- (-2.5,-8);
    \node[right] at (-0.2,-5) {$\pi$};
    \draw[thick] (-7.2,-9.5)  to [out=15, in=165] (1.7,-9.5);
    \node at (4.5,0) {$\mathbb{T}M$};
        \node at (-9.5,0) {$TM$};
     \node at (-1.5,-8.1) {$M$};
    \end{tikzpicture}      
    \caption{Relation between the tangent and the torus bundles}
    \label{fig:tangent and torus bundle}
\end{figure}

Since the action $\theta$ leaves the fibers $T_pM$ invariant, the  decomposition of the tangent bundle to $TM$ into horizontal and vertical parts induces
a similar decomposition in the tangent spaces at points in $\TM$; we will use the same notation for them.  Namely 

\begin{equation}
\label{eq:horiz_and_vertical_in_torus_bundle}
T_{[p,v]}\TM = \mathcalH_{[p,v]}\oplus \mathcalV_{[p,v]}, 
\end{equation}
where 
\[
\mathcalH_{[p,v]}=\Phi_*\mathcalH_{(p,v)}, 
\qquad
\mathcalV_{[p,v]}=\Phi_*\mathcalV_{(p,v)}.
\]
It is also clear that $\mathcalV_{[p,v]}$ is tangent to the fibers of $\TM$ and $\mathcalH_{[p,v]}$ is transverse to them. 
The decomposition 
\begin{equation}
%\label{eq:horiz_and_vertical_in_torus_bundle2}
T_{[p,v]}\TM = \mathcalH_{[p,v]}\oplus \mathcalV_{[p,v]}, 
\end{equation}
gives vector bundles
\[
\mathcalH\to \TM, \qquad \mathcalV\to\TM.
\]
Observe that $\mathcalH$ and $\mathcalV$ are subbundles of the tangent bundle to $\TM$. 

\begin{defn}
\label{def:Sasaki_metric}
    The \emph{Sasaki metric} in $\TM$ is the Riemannian metric such that for  
    vectors $\xi$, $\eta$ in $T_  {[p,v]}TM$, we define
\[
g_{[p,v]}(\xi,\eta):= g_p(\xi^h,\eta^h)+g_p(\xi^v, \eta^v), 
\]
where $\xi=\xi^h+\xi^v$, $\eta=\eta^h+\eta^v$ are the horizontal-vertical splittings.
\end{defn}

The metric is induced by the usual Sasaki metric in the tangent bundle $TM$ (see Appendix \ref{appendix:tangent} or  the excellent introduction in \cite{Paternain}), and with them, the quotient map $\Phi:TM\to\TM$ becomes a local isometry. 

\subsection{Vertical tensors in \texorpdfstring{$\TM$}{}}
When working with certain objects  in $\TM$, it is usually more convenient to think of them as  objects in $TM$ satisfying certain periodicity conditions. We make this explicit in the following lemma.

\begin{lem}
\label{lem:tensors_torus_bundle}
Let $M$ be a parallelizable manifold and $\TM$ the torus bundle associated to a given lattice. Then there is a one-to-one correspondence between
functions (vector fields, $k$-forms) in $\TM$ and functions (vector fields, $k$-forms) in $TM$ that are invariant under the $\mathbb{Z}^n$ action.
\end{lem}
\begin{proof}
For any $\Z^n$-invariant function $f:\TM\to\R$, define 
$\tilde{f}:TM\to\R$ as $\tilde{f}(p,v):=f([p,v])$.
For any integer $n$-tuple $\bar{k}\in\Z^n$, $[p,v]=[p,v+\sum k_i e_i]$, thus $\tilde{f}$ is $\Z^n$-periodic. The reciprocal is obvious.
Similar proofs apply to the other cases.
\end{proof}

We will be specially interested on tensors related to the vertical bundle $\mathcalV$, since two-scale limits are tensors of this type.  
Thus, we will denote by 
$\vectors(\mathcalV)$ the vertical vectors, and by $\Lambda^k(\mathcalV)$ the vertical alternating $k$-multilinear maps; that is, for $X\in\vectors(\mathcalV)$ and $\omega\in\Lambda^k(\mathcalV)$, we have that 
\[
X_{[p,v]}\in\mathcalV_{[p,v]}, \qquad
\omega_{[p,v]}:\mathcalV_{[p,v]}\times\dots\times\mathcalV_{[p,v]}\to\R
\text{ multilinear and alternating},
\]  
at each point $[p,v]\in\TM$. 
The sections of these bundles will be the vertical vector fields and the vertical $k$-forms. The set of $0$-forms can be identified with the functions $f:\TM\to\R$.

We describe now how to use the frame associated to a parallelization to construct framings of $\vectors(\mathcalV)$ and $\Lambda^k(\mathcalV)$. 
\begin{Def}\label{def:vertical}
Let $e\in T_pM$. For any $[p,v]\in \TM$, we define the vertical vector field $e^{\uparrow}\in\vectors(\mathcalV)$ as 
\begin{equation}
\label{def:vector_lift}
e^{\uparrow}[p,v]= \frac{d}{dt}\bigg|_{t=0}[p,v+te ].
\end{equation}
\end{Def}

The above definition is correct, since the corresponding vertical lift 
$e^{\uparrow}$ to $TM$ is invariant by the $\Z^n$-action, and thus, by  
Lemma \ref{lem:tensors_torus_bundle} defines a vector field in $\TM$. 
More precisely, if $[p,v]=[p,w]$, i.e, if $v-w\in\Z^n$ in coordinates with respect the  frame, then $[p,v+te]=[p,w+te]$, and $e^\up[p,v]=e^\up[p,w]$ as claimed. 

\subsubsection{Vertical vector field framing}
Let  $e_i$ be the vector fields in the frame of $M$. 
\begin{Def}
The vertical framing of $\TM$ is the collection of vertical vector fields 
$e_1^\up, \dots, e_n^\up$.
\end{Def}
It is clear that $e_1^\up, \dots, e_n^\up$ form a basis of $\mathcalV_{[p,v]}$ at each $[p,v]$, since the fields $e_i$ form a basis of $T_pM$ at each $p$. 
Notice that if we have the coordinates
\begin{equation}\label{framing coordinates}
\tilde{\varphi}(p,v)=(x^1(p), \ldots, x^n(p), v^1(p,v), \ldots, v^n(p,v))
\end{equation}
in $TM$, where $(x_1,\ldots,x_n)$ are coordinates in $M$, and for $v\in T_pM$, the $v^i$ are given by 
\[ 
v=\sum_{j=1}^n v^i (p,v) e_i, 
\]
then for the corresponding coordinates in $\TM$, we have that 
\begin{equation}
\label{eq:up_and_coordinates}
e^{\up}_{i,[p,v]}= \left.\frac{\p}{\p v^i}\right|_{[p,v]}
\end{equation} 
Therefore, if $X \in \mathfrak{X}(M)$ has coordinates  $X_p= \sum_i X^i(p) e_i$, then

\begin{equation}
X^\uparrow_{[p,v]}= \sum_i X^i(p)\frac{\p}{\p v^i}
\end{equation}
More in general, any vector field in $\vectors(\mathcalV)$ can be written as a linear combination of the vertical frame $e_i^\up$, although in this case, the coefficients of the $e_i^\up$ may vary in $[p,v]$.

\subsubsection{Framing for vertical \texorpdfstring{$k$}{}-forms}
Given a 1-form $\omega$ in $M$, we can lift $\omega_p$ to each point $(p,v)\in TM$ as 
\begin{equation}
\label{eq:one_form_lift_defn}
\omega^\up_{(p,v)}(e^\up)=\omega_p(e), \quad e\in T_pM.
\end{equation}
Observe that we have only defined $\omega^\up$ over vectors tangent to the fibers of $TM$, thus on elements of $\mathcal{V}TM$. Once again, we notice that the 1-form $\omega^\up$ is invariant under the $\Z^n$-action, and it defines a vertical $1$-form on $\TM$.
%Alternatively, we can try to define $\omega^\up$ as a 1-form \emph{in the whole of TM} by the formula above, together with the requirement that $\omega_{(p,v)}$ vanishes over vectors in $\mathcal{H}_{(p,v)}$.
It is clear that 
if $X \in \mathfrak{X}(M)$, $\omega \in \Lambda^1(M)$, then
\[ (\omega^\uparrow, X^\uparrow)[p,v]=(\omega,X)(p). \]

\begin{Def}
\label{def:framing_forms}
Given the frame $e_1, \dots, e_n$ of $\Gamma$, we denote by $\theta_1,\dots,\theta_n$ the associated dual basis of 1-forms in $M$; i.e,
$\theta_i(e_j)=\delta_{ij}$ where $\delta_{ij}$ are the Kronecker symbols. 
We say that the basis of vertical $k$-forms defined as  
\[
\theta_{i_1}^\up\wedge\dots\wedge\theta_{i_k}^\up, \qquad
1\leq i_1<\dots<i_k\leq n,
\]
is the associated framing for $\Lambda^k(\mathcalV)$.
\end{Def}
With respect to the coordinates in $\TM$ induced from those in $TM$ as before, we have that
$\theta_i^\up=dv_i$, where $dv_i$ is applied only to vertical vectors.\footnote{To be precise, since $v_i$ is a function in $\TM$, $dv_i$ could be applied also to any vector in $\TM$, whether vertical or not; our use of $dv_i$ above means that we are taking its restriction to $\mathcalV$.}

\subsubsection{Vertical lifting of arbitrary tensors} 
The above definition yields a canonical way to  lift tensors in $M$ to vertical tensors.  Along the paper we will use multi-index notation 
$I,J$, for $I=i_1,i_2, \ldots, i_k$, $J=j_1,j_2, \dots, j_m$ when no confusion arises. Thus, e.g.  $X_I \in \mathfrak{X}(M)^k$ stands for the $k$-tuple of vector fields
$X_{i_1}, X_{i_2}, \ldots, X_{i_k}$

\begin{Def}
Let $T \in T^{k,m}(M)$ be a $(k,m)$-tensor in $M$. Consider arbitrary  $k$-tuples of vector fields$\{X_1\}_{i=1}^k \in \mathfrak{X}(M)$,
and arbitrary $m$-tuples of $1$-forms, $ \{\omega_j \}_{j=1}^m \in\Lambda^1(M)$. Then $T^\up \in T^{k,m}(\mathcal V)$ is defined by,
\[ T^\up ((X_1^\up,\ldots, X_k^\up, \omega_1^\up, \ldots \omega_m^\up):=T(X_1,\ldots, X_n, \omega_1, \ldots \omega_m)\]

\end{Def}

\subsection{Integration in \texorpdfstring{$\TM$}{}} We give $\TM$ the Sasaki metric induced by the metric in $M$.
Recall that integration in manifolds is always defined with regard to a volume form in the manifold. In the case of Riemannian manifolds, the most natural volume form is that assigning to each positively oriented orthonormal basis the value one. \footnote{Nonorientable manifolds require passing to the orientable double cover, but we will ignore them here since $M$ being parallelizable implies $\TM$ being orientable.} 

To construct such a form on $\TM$, we will take the wedge product of two $n$-forms. The first one will be $\pi^*\omega$, the pullback of the Riemannian volume form in $M$ by the projection map $\pi:\TM\to M$. For the second one, we start by taking the $1$-forms induced by the vertical frame fields $e_i^\up$, i.e, $\theta_i^\up$ extended also to nonvertical vectors as
\[
\theta_i^\up(X)=\ip{\e_i^\up}{X},
\]
where we use the Sasaki metric. Then we construct the $n$-form
\[
\eta:= g^{1/2}\,\theta_1^\up\wedge\dots\wedge \theta_n^\up
\]
where $g$ is the determinant of the matrix whose entries are 
\begin{equation}
\label{eq:inner_product}
\ip{e_i^\up}{e_j^\up}[p,v]= g(e_i, e_j)=g_{ij}(p).
\end{equation}
This $g^{1/2}$- term is necessary since the vectors $e_i^\up$ do not constitute an orthonormal set. 
Alternatively, if we denote by $dV_M$ the Riemannian volume form for $M$, it is clear that $\eta=(dV_M)^\up$, with the terminology from the last section. Thus, the volume form in $\TM$ with the Sasaki metric is given by 
\[
dV_{\TM}= \pi^*\omega\wedge\eta=\pi^*(dV_M) \wedge (dV_M)^\uparrow
\]
Integration with respect to this form will be denoted as 
$\int f[p,v]\, dv\, dp$.

For each $p\in M$, and coordinates $v^1, \dots, v^n$  
on $\TM_p$  in the frame at $p$, 
the corresponding volume element in each torus fiber of $\TM$ is just given by  $dV:= g^{1/2}dv^1\wedge\dots\wedge dv^n$, thus 
$\vol(\TM_p)=g^{1/2}(p)$
, where $g(p)= \det g_{ij}(p)$.

\begin{Def}[Normalized integrals on $\TM$] 
For $f \in L^1(\TM)$, define
\[
\oint_{\TM} f\, dv \,dp
:= 
\int_{\T M}\,  \vol(\TM_p)^{-1} f[p,v] \,dv\,dp \]
where the volume on $\TM$ is induced from the above metric. 
\end{Def}

There is an "integration along the fiber" formula that will be useful later.

\begin{lem}
    Let $f:\TM\to R$ an integrable function such that, for almost any $p\in M$, the function 
    $v\to f[p,v]$ is integrable in $\TM_p$; let $\tildef(p)=\int_{\TM_p} f[p,v]\,dv$.  Then
    \[
    \oint_{\TM}\,f[p,v]\, dv\,dp = \int_M\, \tildef(p)\, dp
    \]
\end{lem}
\begin{proof}
    We can take parametrizations in $\TM$ of the form 
    \[
    \Psi(x_1,\dots,x_n,v_1,\dots,v_n)=[\phi(x_1,\dots,x_n), v_1 e_1+\dots +v_n e_n].
    \]
    Then $\Psi^* dV_{\TM}=( g^{1/2}\circ\phi)\, (\phi^* dV_M)\wedge dv_1\wedge \dots\wedge dv_n$, and 
    \[
    \int_{\Psi(U\times [0,1]^n)} \vol(\TM_p)^{-1} f[p,v]\,dv\,dp = 
    \int_{U\times [0,1]^n}( g^{-1/2}\circ\phi)(x) f\circ \Psi(x,v)\,( g^{1/2}\circ\phi)(x)\,dv\,dx,
    \]
    where the right-hand side is just usual integration in $\R^{2n}$. 
    After cancellations, the result then follows from Fubini's Theorem.
\end{proof}

\begin{Cor}
\label{Cor:updown} 
Let $X \in \mathfrak{X}(M)$ , $\omega \in \Lambda^1(M)$.  Then 
\[ 
\int_{M} (X,\omega)\, dp= \oint_{\TM} ({X^\uparrow}, {\omega^\uparrow})\,dv\,dp 
\]
\end{Cor}
\begin{proof}
From the definition of $\omega^\up$ and $X^\up$, it holds that 
\[
({X^\uparrow} , {\omega^\uparrow})[p,v]=(X,\omega)(p);
\]
the corollary follows integrating along each fiber, since  $\vol(\TM_p)=\det g^{1/2}$. 
\end{proof}

\begin{Def}[Tensor average on the fibers]
\label{def:tensor_average}
Let $T \in T^{k,m}(\mathcal{V} )$ be a vertical $(k,m)$-tensor; then we define its
average $\tilde{T} \in T^{k,m}(M)$ by

\[ \tilde{T}(X_I, \omega^J)= \vol(\TM_p)^{-1} \int_{\TM_p} T(X_I^\up, \omega_J^{\up})\, dv 
:={\dashint}_{\TM_p} T(X_I^\up, \omega_J^{\up})\, dv\]
Notice that if 
\[
T = \sum_i T_I^J[p, v]\, e_I^\up \otimes \theta_J^\up
\]
in  coordinates with respect to the vertical frames, then setting 
\begin{equation}
\tilde{T}_I^J(p) = {\avint}_{\TM_p} T_I^J[p, v]\, dv,\,\, 
\end{equation}
we obtain that
\begin{equation} \label{average}
\tilde{T}= \sum \tilde{T}_I^J e_I \otimes \theta_J.
\end{equation}

\end{Def} 
\subsection{Vertical differential and vertical gradient}
Let $f:\TM\to\R$ a smooth function. Its \emph{vertical gradient} is the vertical vector field $\grad_vf$ along $\TM$ such that 
\[
\ip{\grad _vf}{e^\up}[p,v]=e^\up(f),
\]
where the inner product comes from the Sasaki metric. It is clear that in coordinates associated to the frame, 
\[
\grad_vf=\sum_{i,j}\,e_i^\up(f) g^{ij}(p) e_j^\up
\]

Similarly, we define its vertical differential $d^vf$ as the element of $\Lambda^{1}(\mathcalV)$ satisfying
\[
d^vf_{[p,v]}(X)=X(f)[p,v]=\sum_{i=1}^n\, \depa{f}{v_i}\cdot X_i,
\]
where $X=\sum_{i=1}^n\, X_i\cdot e_i^\uparrow[p,v]$ in frame coordinates (we are using \eqref {eq:up_and_coordinates}).

\subsection{Vertical divergence of vertical vector fields}
Let $X$ be a vertical vector field in $\TM$ with $X=\sum_i f_i[p,v] \frac{\partial}{\partial v_i}$. If $Y $ is some other vertical tangent vector at $[p,v]\in \TM$, we define 
\[
\nabla^\mathcalV_Y X(p,v)= \sum_i Y(f_i)\frac{\partial}{\partial v_i},
\]
or equivalently, 
\[
\nabla^\mathcalV_Y X[p,v]=\frac{d}{dt}X_{[p,v+tY
]},
\]
where the derivative is taken at $t=0$. 
% When  $Y=e_j^\up$, the above formula gives
% \[
% \nabla^\mathcalV_{e_j^\up} X(p,v)=\sum_i \frac{\partial f_i}{\partial v_j}[p,v] \frac{\partial}{\partial v_i}.
% \]

\begin{Def}
The \emph{vertical divergence} of $X$ is the function defined as
%\[
%\div^v X(p,v)=\sum_{i=1}^n\ip{\nabla^\mathcalV_{e_i} X}{e_i},
%\]
\[
(\div^v X)[p,v]= \trace (Y\mapsto \nabla^\mathcalV_Y X)
\]
where $Y\in\mathcalV_{[p,v]}$.
\footnote{
When $TM$ is given the Sasaki metric with flat fibers, this agrees with the standard vertical divergence for semibasic vector fields, as defined, for instance, in \cite[Section 6.1]{Knie}.}
%where $\{e_i\}$ is the frame basis of $\mathcalV_{(p,v)}TM$.
\end{Def}
When $\{e_i\}$ is the frame basis of $\mathcalV_{(p,v)}TM$,
and $X=\sum_i f_i[p,v] \frac{\partial}{\partial v_i}$, the above formula yields
\[
(\div^v X)[p,v]= \sum_{i=1}^n\frac{\partial f_i}{\partial v_i}
\]
% since the vertical metric on each vertical fiber is independent of $v$; i.e,  for the frame vertical fields, $\ip{e_i}{e_j}[p,v]$ is a function only of $p$.

\begin{lem}
For any $e\in T_pM$, and any smooth ${f}:\TM\to\R$, we have
\[
\div^v ({f}e^\uparrow)
=
e^\uparrow({f})
\]
\end{lem}
\begin{proof}
It is immediate to check that the standard formula
\[
\div^v ({f}X)=X({f})+{f}\div^v X
\]
still holds, so the lemma follows from
$\div^v e^\uparrow\equiv 0$.
\end{proof}

In the next lemma and the rest of the paper, whenever we speak about the divergence of a vector field in the torus bundle, we will consider the divergence with respect to the Sasaki metric in $\TM$ (see  Definition \ref{def:Sasaki_metric} for the definition).
\begin{lem}
\label{lem:vertical_divergence_versus_regular_divergence}
    For any $X\in \mathfrak{X}(\V)$, and the Sasaki metric on $\TM$,
    \[
    \div X =\div^v X.
    \]
\end{lem}
\begin{proof}
    The proof follows from a standard argument, since for any $f\in C^\infty(\TM)$,
    \[
\ip{\grad f}{X} = \ip{\grad_v f}{X},
    \]
    and the definitions of divergence, vertical divergence and the integration by parts formula
    \[
    \int_{\TM}\,\ip{\grad f}{X} \,dv\,dp = - \int_{\TM}\,f\cdot \div X \,dv\,dp.
    \]
\end{proof}
\subsection{Vertical derivative  and divergence of one-forms}

\label{sec:vertical_divergence}
Recall that, in a closed Riemannian manifold, the divergence of a 1-form $\omega$ is defined as the function $\div\omega$ such that for any smooth function $f\in C^{\infty}(M)$,
\[
\int_M \,({\omega},{\grad f})\,dp =
-\int_M \, f\div\omega\,dp.
\]
 In the above formula, the parentheses denote   the usual pairing; alternatively, using the $L^2$-product in one forms induced from the metric, we could have written the above as
\[
\int_M \, \ip{\omega}{df}\,dp =
-\int_M \, f\div\omega\,dp.
\]
For a one-form $\omega_0$ in the torus bundle $\TM$, we will define its \emph{vertical divergence} as the unique function $\div^v\omega_0\in C^{\infty}(\TM)$ that satisfies 
\begin{equation}
\int_{\TM} \, ({\omega_0},{\grad_v f})\,dp\, dv
=
-\int_{\TM}\, f \div^v\omega_0\, dp \,dv
\end{equation}
where $f\in C^\infty(\TM)$, and $\grad_v f$ is just the vertical part of the gradient $\nabla f$.

\subsection{Spaces of functions and  tensors} \label{subsec:functionspaces}
For functions defined on the manifold $M$ or in the torus bundle $\TM$, we will use the classical notation $C(M), L^2(M), L^2(\TM)$. For vector fields or tensors in $M$ we will use   $L^2 \mathfrak{X}(M), L^2 T^{k,m}(M)$ and  for vertical vector  fields and  vertical tensors, we use   $L^2 \mathfrak{X}(\V)$, $L^2T^{k,m}(\V)$.

In addition,  for functions defined on $\TM$, we will use mixed spaces to consider different regularities on the base and the fibers of the torus bundle.  
\begin{Def}
Let $\mathcal{A}$ be a function space in $M$, and $\mathcal{B}$ be some Banach space at each torus fiber $\TM_p$. 
\begin{enumerate}
\item We say that $f\in {\mathcal{A}} (M,{\mathcal{B}}( \TM_p))$ if the function 
$
 f^*(p)\to \|f[p,*]\|_{\mathcal{B}}
$
belongs to $\mathcal{A}$. 

\item If $\mathcal{A}$ is given by a norm $\|*\|_{\mathcal{A}}$, then we define a norm in 
${\mathcal{A}} (M,{\mathcal{B}}( \TM_p))$ as 
\[
\|f\|_{{\mathcal{A}} (M,{\mathcal{B}}( \TM_p))}:=\|f^*\|_{\mathcal{A}}.
\]
\end{enumerate}
Likewise we will denote vertical vector fields or vertical tensor fields such that their frame coordinate functions belong to ${\mathcal{A}} (M,{\mathcal{B}}( \TM_p))$ by $\mathcal{A}(M,\mathcal{B}\mathfrak{X}(\V_p))$, $\mathcal{A}(M,\mathcal{B}T^{k,m}(\V_p))$.
\end{Def}

 Throughout the paper, we will  always take
 $\mathcal{A}$ to be continuous functions, as this simplifies the presentation and avoids some technical details. 
 $\mathcal{B}$ will typically be either continuous,  $L^r$ spaces or Sobolev spaces  in the fibers. 
 As a matter of fact, since the fibers of $\TM$ are tori, we have that $L^r(\TM_p)\subseteq L^1(\TM_p)$ for any $1\leq r$. 
 A recurrent  choice for us will be functions $C(M, L^2(\TM_p))$ and vector and tensors fields in $\TVX, \TVkm$.
 Notice that sometimes mixed spaces agree with standard spaces. For example, by Fubini's Theorem, $L^2(M,L^2(\TM_p))=L^2(\TM)$.
 %We will also need a Sobolev space of ''periodic functions in the fiber".
 
   Finally, in each torus fiber  $\TM_p$, we will consider
the space $ \Hper(\TM_p)$ of functions with zero vertical mean, that is $\int_{\mathbb TM_p} u[p,v] dv=0$ for almost every $p$. Thus, Poincar\`e-Wirtinger inequality in the fiber $\TM_p$ allows us to consider the following norm for the space $L^2(M,\Hper(\TM_p))$:

\begin{equation}\label{eq:norm}
\|u\|_{L^2(M,\Hper(\TM_p))}=\oint_{\TM} |\grad_v u(p,v)|^2\,dp\,dv 
\end{equation}

Finally, by $\| \cdot\|$ we will indicate the supremum norm in $M$ or $\TM$.

\part[Voronoi domains and partitions of unity]{Nets and Voronoi decompositions}

\section{Voronoi domains and bump functions}
\label{sec:voronoi}
The results in this section apply to arbitrary closed Riemannian manifolds; we will not assume parallelizability anywhere. 
\subsection{Normal charts}
\label{subsec:normal_charts}
Recall that given a point $p\in M$,
 we can always consider a totally normal neighborhood $U$ of $p$ in $M$ and the diffeomorphism
 \[
 \phi_p=\exp_p^{-1}: U  \to V\subset T_pM.
 \]
We will drop the $p$ whenever no confusion may arise. 
By choosing orthonormal coordinates in $T_pM$,  we have a coordinate chart $(U,\phi)$ in a neighbourhood of $p$
such that
$\phi(p)=0$, and such that the metric satisfies
	\begin{equation}
	\label{eq:metric_at_zero}
g_{ij}(0)=\delta_{ij}, 
\quad
\partial_k g_{ij}(0)=0,
\quad
\det g_{ij}(v)=1+O(|v|^2).	
	\end{equation}

	In fact, in these coordinates, the differential
	\[
d\phi:TM|_{U}\to T\R^n	
	\]
is $C^1$-close to the identity whenever $U$ is  a small neighbourhood of $p$. To make this more precise, recall that $\phi:U\to \R^n$ is a diffeomorphism onto its image $V\subset\R^n$, and the tangent bundle of $\R^n$ over $V$ trivializes as 
$V\times \R^n$. Thus,  
$d\phi:TU\to TV$ is a diffeomorphism. If, for instance, we take the Sasaki metric $\hat{g}$ induced by $g$ on $TU$ (although many other metrics could be used at this point), and the Euclidean metric on $TV=V\times\R^n$,  $\hat{g}_0$, then the map $d\phi:TU\to TV$ can be taken as close as we want to an \textcolor{blue} isometry in the $C^1$-topology by reducing the size of $U$. This argument can be used to give a formal proof of the following Lemma. 

\begin{lem}
\label{lem:almost Euclidean charts}
Let $M$ be a compact Riemannian manifold. Then for  every positive constant $C$, there exists $\delta>0$ such that for any normal chart $(U,\phi)$ as above in $M$
with $\diam U\leq\delta$, we have that 
\begin{equation}
\label{eq:exponentialc1close}
\|d\phi_p(q)- I\| \le C(d(p,q)),
\end{equation}

\end{lem}
 Notice that \eqref{eq:exponentialc1close} yields that, by choosing $\delta$ small enough, we have that, $\exp$ and  $d\exp$ are arbitrarily close to the identity in a bilipschitz sense, uniformly in the manifold. In particular, they are  $L^\infty$ close to the identity.

%For most applications it will be sufficient to take $\varepsilon=1/2$.

\subsection{Nets and Voronoi domains in manifolds}

We consider a compact Riemannian manifold $(M,g)$ with distance function denoted by $d(p,q)$. 
\begin{defn}
Let $\varepsilon>0$.
\begin{enumerate}
\item A maximal $\varepsilon$-separated set is  a maximal set of points $\{x_i\}_{i\in I}$ in $M$ such that $d(x_i,x_j)\geq \varepsilon$. 
\item An $\varepsilon$-minimal net is a set of points $\{x_i\}_{i\in I}$ in $M$ with a minimal number of elements such that any point in $M$ lies in an $\varepsilon$-ball $B(x_i,\varepsilon)$.
\end{enumerate}
\end{defn}

\begin{lem}
 A maximal $\varepsilon$-separated set is a $\varepsilon$-minimal net
\end{lem}
\begin{proof}
If there was an $x$ with $d(x,x_i)\geq \varepsilon$, then we could add $x$ to $\{x_i\}$ to get an $\varepsilon$-separated set with one element more, contradicting maximality. This shows that 
  maximal $\varepsilon$-separated sets are $\varepsilon$-minimal nets.
 
\end{proof}

\begin{defn}
Let $\mathcalN(\varepsilon)=\{p_i\,,\,1\leq i\leq N(\varepsilon) \}$,  a maximal $\varepsilon$-separated set. The Voronoi domain corresponding to $p_i$ is defined as 
\[
D_i:=\{q\in M \,:\, \dist(q,p_i)\leq \dist(q,p_j) \text{ for all } 1\leq j\leq N\,\}
\]
\end{defn}

Standard counting arguments show that  the cardinal  of $\mathcal{N}(\varepsilon)$  grows as $\varepsilon^{-n}$ as $\varepsilon \to 0$ with $n=\dim M$.

\begin{lem}
\label{lem:vor_eps}
If $q\in D_i$, $\dist(q,p_i)\leq \varepsilon$.
\end{lem}

\begin{proof}
Suppose otherwise; then $\varepsilon<\dist(q,p_i)\leq \dist(q,p_j)$ for any $j$, and we could add $q$ to $\mathcalN(\varepsilon)$ to get a larger maximal $\varepsilon$-separated set.
\end{proof}

The sets $D_i$ cover all of $M$, intersecting only in sets of empty interior.

\begin{lem}
\label{lem:ball in interior Voronoi}
For any $i$, the ball $B(p_i,\varepsilon/2)$ is contained in $D_i$.
\end{lem}
\begin{proof}
If $\dist(p_i,q)<\varepsilon/2$, and $q\in D_j$, then $\dist(p_j,q)\leq \dist(p_i,q)$; thus 
\[
\dist(p_i,p_j)\leq \dist(p_i,q)+\dist(p_j,q)<\varepsilon,
\]
contradicting the separation condition except if $i=j$.
\end{proof}

 We will say that two Voronoi domains are \emph{adjacent} if $D_i\cap D_j\neq \emptyset$.

\begin{lem}\label{lem:distvoronoi}
 Suppose $D_i$ and $D_j$ are adjacent  Then 
\[
\varepsilon\leq \dist(p_i,p_j)\leq 2\varepsilon.
\]
\end{lem}
\begin{proof}
The first inequality comes from the definition of $\varepsilon$-separated set; 
the second is an immediate consequence of the triangle inequality and  Lemma \ref{lem:vor_eps} after choosing some $q\in D_i\cap D_j$.
\end{proof}

In what follows we will assume $\varepsilon$ is smaller than half the injectivity radius of $M$ so that there are unique geodesics between $p_i$'s belonging to adjacent Voronoi domains.

\begin{lem}
\label{lem:triangle}  
% Let $q_i,q_j,q_k \in M$ and consider $\triangle(q_i,q_j,q_k)$  be the corresponding minimal geodesic  triangle. \footnote{i.e, every side is a minimal segment.} Suppose that there is a universal constant $C$ such that

% \[  \varepsilon \le \dist(q_i,q_j),  \dist(q_i,q_k),  \dist(q_j,q_k), \le  2 \varepsilon \]
% Then there exist  $\epsilon_0,\delta$,depending on $M$  such that   for $\varepsilon \le \varepsilon_0$
% \[
% \measuredangle_{p_i}(p_j,p_k)\geq \delta.
% \]
Let $M$ be a closed Riemannian manifold. Then there are positive numbers  $\varepsilon_0$ and $\delta$  such that   for  any $\varepsilon \le \varepsilon_0$ and for any points 
$q_i,q_j,q_k \in M$ with
\[  
\varepsilon \le \dist(q_i,q_j),  \dist(q_i,q_k),  \dist(q_j,q_k) \le  2 \varepsilon 
\]
we have that 
\[
\measuredangle_{p_i}(p_j,p_k)\geq \delta,
\]
where $\measuredangle_{p_i}(p_j,p_k)$ is the angle at $p_i$ of the minimal geodesic triangle $\triangle(q_i,q_j,q_k)$.  \footnote{i.e, every side is a minimal segment.}
\end{lem}

\begin{proof}
Since the geodesics are minimal, we can use Toponogov's theorem \cite{CheEbin} to compare angles with those in the corresponding triangle in  the simply connected space form of constant curvature $k_0<0$, where $k_0$ is a lower bound for the sectional curvature of $M$,  and then use hyperbolic trigonometry.
Namely,  we consider a corresponding triangle in $\mathbb{H}(k_0)$ where the side lengths are $l_{ij}=a$, $l_{ik}=b$, $l_{jk}=c$ are the corresponding hyperbolic distances between the vertices, and  assume the  worst case $a=b=2\e$ and $c=\e$ .  We denote by $\alpha_i$ the angle at the vertex $p_i$
(still in the space form) and  use  the following  law of cosines.

 \[\cos \alpha=\frac{\cosh(\sqrt{|k_0}|a) \cosh(\sqrt{|k_0|} b)-\cosh(\sqrt{|k_0|} c)}{\sinh(\sqrt{|k_0|}a) \sinh( \sqrt{|k_0|} b) )}
. \]

 As $\sinh(x) =x+ O(x^3), \cosh{x}=1+x^2+ O(x^4)$ for small $x$, we have that there  exists $\varepsilon_0(k_0)>0$ such that for   $\varepsilon<\varepsilon_0$, there exists  a universal constant $c_0>0$ with

\[   \cos \alpha \le 1-c_0.  \]

Toponogov's theorem states that 

\[ \measuredangle_{p_i}(p_j',p_k')\geq \alpha \]
and the proof follows.
\end{proof}

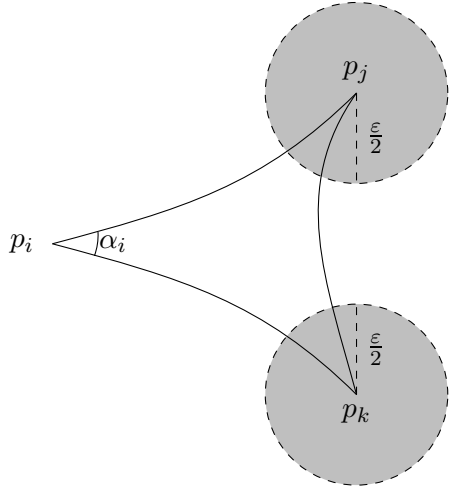
\begin{figure}
    \centering
       \begin{tikzpicture}[scale=2]
       \draw[dashed, fill=lightgray] (2,1) circle [radius=0.6];
        \draw[dashed, fill=lightgray] (2,-1) circle [radius=0.6];
   \draw (0,0) to [out=15, in=225] (2,1);
    \draw (0,0) to [out=-15, in=135] (2,-1);
    \draw (2,-1) to [out=105, in=235] (2,1);
	\draw[dashed] (2,1) -- (2,0.4) ;
	\draw[dashed] (2,-1) -- (2,-0.4) ;
   \node at (-0.2,0) {$p_i$};
    \node[above] at (2,1) {$p_j$};
    \node[below] at (2,-1) {$p_k$};
   \node[right] at (2,0.7) {$\frac{\varepsilon}{2}$};
      \node[right] at (2,-0.7) {$\frac{\varepsilon}{2}$};
    \draw (0.3,0.078) arc [radius=0.4, start angle=3, end angle= -20];
    \node at (0.4,0) {$\alpha_i$};
     \end{tikzpicture}
    \caption{A geodesic triangle of size comparable to $\e$.}
    \label{fig:geodesic triangle}
\end{figure}

\begin{lem}
\label{lem:lower bound for angles}
Let $D_j$ and $D_k$ be Voronoi domains adjacent to $D_i$. Denote by $p_j'$ and $p_k'$ the unit vectors at $p_i$ tangent to geodesics from $p_i$ to $p_j$ and $p_k$ respectively.
Then there is a $\delta=\delta(M)>0$ such that 
\[
\measuredangle_{p_i}(p_j',p_k')\geq \delta.
\]
\end{lem}

\begin{proof}

We consider the geodesic triangle with {vertices}  $p_i,p_j, p_k$. By Lemma~\ref{lem:distvoronoi}, the length of the corresponding geodesic $\varepsilon \le l_{ij},l_{ik},l_{jk}\le 2\epsilon$, thus the claim is a corollary of lemma~\ref{lem:triangle}

\end{proof}

In fact, $\alpha_0$ can be estimated precisely; for instance, for $k_0=0$, $\tan\alpha_0=1/2$.

\begin{lem}
\label{lem:bound_voronoi_neighbours}
Let $D_i$ be a given Voronoi domain generated by a maximal $\e$-separated   
net; the number of its adjacent domains is bounded above by a constant $I_0$ independent of $\epsilon$.
\end{lem}
\begin{proof}
Lemma \ref{lem:lower bound for angles} shows that, if $D_{i_1},\dots, D_{i_k}$ are adjacent to $D_i$, then the set of directions 
$p'_{i_1}, \dots, p'_{i_k}$ form an $\alpha_0/2$-separated net in the unit tangent sphere at $p_i$. 
Since this unit sphere has dimension $n-1$, a simple volume counting argument shows that the amount of these points is bounded above by $C\omega_{n-1}/\alpha_0^{n-1}$ for some $C=C_{n-1}>0$ independent of $\epsilon$ and $\omega_{n-1}$ the volume of the unit $(n-1)$-dimensional sphere.
\end{proof}
 
 We reformulate the last Lemma in a form more useful for the future.
 
\begin{lem}[Finite overlapping property]
\label{lem:FOP}
 Let $(M,g)$ be a closed Riemannian manifold and $\varepsilon>0$ be small enough. If $\{p_i\}$ is a maximal $\varepsilon$-separated set in $M$, then 
there exists some $K>0$ independent of $\varepsilon$ such that
$\sum_i\chi_{B(p_i,\varepsilon)}\leq K$, where $\chi_A$ is the characteristic function of the set $A$.
\end{lem}

Finally, in order to create a convenient partition of unity  we need to study the boundary of each Voronoi domain, and estimate
the measure of a tubular neighborhood.

%%%%%%%%%%%%%%%%%%%%%%%%%%%%%%%%%%%%%%%%%%%%%
%%%%% This part appears below; erase it there
%%%%%%%%%%%%%%%%%%%%%%%%%%%%%%%%%%%%%%%%%%%%%

\begin{lem}\label{lem:Voronoiboundary}
Let $D_i$ be one of the Voronoi domains. Then
its boundary $\partial D_i$ is a  piecewise  smooth submanifold of dimension  $n-1$. Indeed,
there exist constants $c$, $C'$ independent of $\varepsilon$ such that
\begin{equation}
\mathcal{H}^{n-1}(\partial D_i) \le C' \varepsilon^{n-1}
\end{equation}
and for
$\delta \le c \varepsilon$, 
\[\vol(\,\cup_{p\in\partial D_i} B(p,\delta)\,)\leq C'\varepsilon^{n-1}\delta.
\]
\end{lem}
\begin{proof}

Assume that $D_i$ and $D_k$ are adjacent Voronoi domains, and consider the common boundary  denoted by  
\begin{equation}
\label{eq:def_Bik}
B_{ik}:=D_i \cap \{   \,q\,:\,\dist(q,p_i)=\dist(q,p_k)\,   \}
\end{equation}
Let $f_{ik}:B(p_i,\e)\to\R$ be the function defined as 
 $f_{ik}(q)=\dist(q,p_i)-\dist(q,p_k)$; thanks to Lemma \ref{lem:vor_eps}, it is clear that $B_{ik}\subset f_{ik}^{-1}(0)$. 
The function $f_{ik}$ has gradient 
\begin{equation}
\label{eq:gradient_difference_distances}
\nabla f_{ik}(q)=v_{i,q}-v_{k,q},
\end{equation}
where $v_{i,q}$ is the unit vector at $q$ tangent to the unique geodesic from $q$ to $p_i$, and  $v_{k,q}$ is defined analogously.  In order to apply the coarea
formula, we need a lower bound on the gradient of $f_{ik}$. The argument is again of geometric nature.  For $q \in B_{ik}$ we consider the corresponding geodesic triangle $\triangle(p_i,p_k,q)$. Recall  that lemma~\ref{lem:distvoronoi} implies that  $\varepsilon \le \dist(p_i,p_k) \le 2\varepsilon$ and $\varepsilon \le \dist(p_i,p_k) \le 2 \dist(p_i,q)=2 \dist(p_k,q)\le 2\varepsilon$, where the first inequality follows from the triangle inequality and the second because $q_i \in \bar{D_i}$ and lemma~\ref{lem:vor_eps}. Therefore, lemma~\ref{lem:triangle}
yields an upper bound $1-c_0$ for the cosine of the angle between $v_{i,q}$ and $v_{k,q}$. Since both vectors are of unit length, we obtain that if $f_{ik}(q)=0$

\[ |\nabla f_{ik}(q)|^2 \ge 2c_0 \]

Now, if $f_{i,k}(q) \le \delta$, $\varepsilon \le 2 \dist(p_i,q)+\delta$. Thus, if $q \in B(p_i,2\varepsilon)$ the geodesics in the corresponding triangle, $\triangle(p_i,p_k,q)$, have length
between $\frac{1}{C(\delta)} \varepsilon$ and $C(\delta) \varepsilon$ with $\lim_{\delta \to 0} C(\delta)=2$. Hence for  $\varepsilon$ small enough, there exists a uniform $\delta$, independent of $\varepsilon$, and such that for every $i,k$,

\begin{equation}\label{eq:lbgradient} |\nabla f_{ik}(q)|^2 \ge c_0 \end{equation}
The same argument, interchanging $p_i$ and $p_k$, provides a similar bound when $-\delta<f_{ik}(q)$.

Next,  by using  the implicit
function theorem  and Lemma \ref{lem:almost Euclidean charts}, there are uniform constant $C',\delta_0$ such that  for every $i,k$, and for every $\delta \le \delta_0$
it holds that, 
\begin{equation}\label{eq:slices}
\mathcal{H}^{n-1}( |f_{ik}|^{-1}(\delta) \cap  B(p_i,2\varepsilon)) \le C' \varepsilon^{n-1}.\end{equation}
Denoting
\[
R_{ik}(\delta):=B(p_i,2\varepsilon)\cap \left\{-\delta<f_{ik}<\delta\right\},
\]
 we insert estimates \eqref{eq:lbgradient}, \eqref{eq:slices}  into 
the coarea formula

\[ \vol(R_{ik}(\delta)) \le \int_{-\delta}^\delta \int_{ f_{ik}^{-1}(\delta)} \frac{1}{|\nabla f_{ik}(q')|}\, d\mathcal{H}^{n-1}(q')\, dt \]
to obtain that 

\begin{equation}
\label{eq:vol_slice_voro}
\vol R_{ik}(\delta) < C_n \varepsilon^{n-1}\delta.
\end{equation}
where the constant $C_n$ is independent of $\epsilon$.

For the first part of the Lemma, use that the relative interior of each set $B_{ik}$ appearing in \eqref{eq:def_Bik} is a smooth manifold, since it is the level set of a regular value of the function $f_{ik}$.
The proof of the second part follows from inequality   \eqref{eq:vol_slice_voro}, together with the uniform upper bound on the number of possible faces of $D_i$ obtained in  
Lemma \ref{lem:bound_voronoi_neighbours}.
\end{proof}

 For the reader's convenience, we collect the results obtained in this section in the next Theorem, where we denote by $r_0>0$ a positive number that assures that any open set of diameter less than $r_0$ is a totally  normal neighborhood of each of its points. 

\begin{Theorem}[Voronoi decomposition Theorem]
\label{thm:big_Voronoi}
Let $(M,g)$ a closed Riemannian manifold, and $\varepsilon<<r_0$. For any maximal $\varepsilon$-separated set, let $\{D_i\}$ denote the corresponding Voronoi decomposition. Then
\begin{enumerate}
\item $M=\cup_i D_i$;
\item the interiors $\mathring D_i$ are disjoint sets and $B(p_i, \frac12\varepsilon)\subset\mathring D_i$;
\item 
\label{thm_sec:big_Voronoi_boundary}
The boundary $\partial D_i$ is a piecewise smooth submanifold of dimension $n-1$, with the bound
$\mathcal{H}^{n-1}(\partial D_i) \le C' \varepsilon^{n-1}$ and thus, 
\[
\vol(\cup_{p\in\partial D_i} B(p,\delta))\leq C\varepsilon^{n-1}\delta,
\]
where $C>0$ is independent of $\varepsilon$;
\item every $D_i$ is a totally normal neighborhood of each of its points.
\end{enumerate}
\end{Theorem}

\subsection{A refined partition of unity}
\label{sec:partition}
Our goal is to associate a smooth function $\psi_i$ for any $p_i\in \mathcalN(\varepsilon)$ and any sufficiently small $\delta>0$. Moreover, we will require that $\psi_i$ is supported in $D_i$, and $\psi_i(x)\equiv 1$ in points of $D_i$ with $\dist(x,\partial D_i)\geq \delta$.

\subsubsection{Construction}
Consider $p_i\in \mathcalN(\varepsilon)$ and let $I(p_i)$ be a set of indexes corresponding to adjacent regions to $D_i$. For each $k\in I(p_i)$, define  
\[
d_{ik}(x):=\dist(p_k,x)-\dist(p_i,x),
\]
and
\[
\phi_{ik,\delta}(x)=H_{\delta}\circ d_{ik}(x)
\]
%
%\[
%\Phi_{ik}(x)=H\circ d_{ik}(x)
%\]
where $H_{\delta}:\R\to\R$ is a smooth function with
\begin{itemize}
\item $H_\delta(t)=1$ for $t>\delta$, and 
$H_\delta(t)=0$ for $t< -\delta$;
\item $H_\delta$ is strictly increasing with $H_\delta'(t)\leq 2/\delta$ in  $[-\delta,\delta]$.
\end{itemize}
Then it is clear that 
\begin{enumerate}
\item $\phi_{ik}(x)=1$, if  
$\dist(p_i,x)<\dist(p_k,x)
-\delta$;
\item $\phi_{ik}(x)=0$, if $\dist(p_i,x)\geq\dist(p_k,x)+
\delta$.
\end{enumerate}

Moreover, the bounds of the gradient of $d_{ik}$ as in  \eqref{eq:gradient_difference_distances} and  the choice of $H_\delta$ yields that 
$|\nabla\phi_{ik}|\leq C\delta^{-1}$
for some constant $C$ that is independent of $\e$ and $\delta$.

Define next 
\begin{equation}
\label{eq:3}
\phi_i(x)= \prod_{k \in I(i)} \phi_{ik}(x).
\end{equation}

Then $|\nabla \phi_i| < C \delta^{-1}$ where $C>0$ is, again, $\varepsilon$ and $\delta$ independent.

The sets 
\begin{equation}
\label{eq:supp_inside_voronoi} 
 D_i^{\delta, +} : =\cap_{k \in I(i)} 
 \left\{\,x\,:\,\dist(p_i,x)\leq\dist(p_k,x)
+\delta\,
\right\} 
\end{equation}
and 
\begin{equation}
\label{eq:value_one_inside_voronoi}
  D_i^{\delta, -} :=
\cap_{k \in I(i)} \{\,x\,:\,\dist(p_i,x)\leq\dist(p_k,x)-
\delta \}  
%\cap D_i
\end{equation}
satisfy that
\[
 D_i^{\delta, -}\subset D_i\subset  D_i^{\delta, +};
\]
moreover, by the choice of the $\phi_{ik}$'s, we have that
\begin{equation}
\label{eq:thick_thin_voronoi_}
\supp \phi_i \subset 
 D_i^{\delta, +}, 
 \qquad
 \phi_i=1 \text{ in } D_i^{\delta, -},  
\end{equation}

Let $R_i:=\{\,0<\phi_i<1\,\}$, and notice that, from the above,
\[
R_i\subset\cup_{p\in\partial D_i} B(p,\delta),
\]
and from Theorem \ref{thm_sec:big_Voronoi_boundary}, we get that
\begin{equation}
\vol R_i\leq C'\varepsilon^{n-1}\delta,
\end{equation}

 At last, we define
 \begin{equation}
\label{eq:partition_of_unity_def} 
 \psi_i(x):=\frac{\phi_i(x)}{\sum_{p_k\in \mathcalN(\varepsilon)}\phi_k(x)}
 \end{equation}

\begin{Theorem}
\label{thm:partition_of_unity}
Let $(M,g)$ be a closed Riemannian manifold, and $\delta,\varepsilon >0$ as chosen before, and $\{p_i\}$ a maximal $\varepsilon$  set with associated Voronoi decomposition $D_i$.
Then
\begin{enumerate}
\item $\{\psi_i\}_i$ is a smooth partition of unity with 
\[ 
\psi_i=1 \text{ in } D_i^{\delta, -},
\qquad
\psi_i=0   \text{ in } 
 M\setminus D_i^{\delta, +}.
\]
\item For each $i$, the set $\{\nabla \psi_i\neq 0\,\} \subset D_i^{\delta,+} \setminus D^{\delta,-}_i \subset\cup_{p\in\partial D_i} B(p,\delta),$, hence
\begin{equation}
\label{eq:vol_nabla_psi}
\vol\{\, \nabla \psi_i\neq 0\,\}
\leq
C'\varepsilon^{n-1}\delta
\end{equation}
for some $C'>0$.
\item We have an inequality
\begin{equation} %\label{eq:volume of all gradients}
 \vol( \cup _i\supp \nabla \psi_i  )\leq  C''\delta\varepsilon^{-1}.
\end{equation} 
for some constant $C''>0$.
\item For any multi-index $\gamma \in \mathbb{Z}^n$,
\begin{equation} %\label{eq:size of iterated gradient}   
  | D^\gamma \psi_j (x)| \leq C_\gamma \delta^{-|\gamma|}.
\end{equation}   
\end{enumerate}
\end{Theorem}
\begin{proof}
Every statement follows from the previous comments. 
%In the second one, we need to use the finite overlapping property of the net, as stated in Lemma \ref{lem:FOP}. 
The first is immediate. For the second, observe that, by construction, 
\[
\supp\psi_i\subset (D_i^+\setminus D_i^-)\cup \cup_j[(D_j^+\setminus D_j^-)\cap D_i];
\]
inequality \eqref{eq:vol_nabla_psi} follows then from the third part of Theorem \ref{thm:big_Voronoi}.
The third follows because, for small $\varepsilon$, the number of points in the net grows as $\varepsilon^{-n}$, by a standard volume counting argument. To prove the last one, we can just work in normal coordinates and apply induction.
\end{proof}

\subsection{Application: the main Theorem }
For the rest of the paper, we will need to use Theorem \ref{thm:partition_of_unity} with specials $\varepsilon$ and $\delta$. 
This is due because, as the reader will see,  
 the crucial step for our approach to homogenization is the projection of the prototype test functions from $\TM$ to $M$,
 which requires a proper microlocalization procedure parametrized by $\e <<1$. To this end
 we introduce, for $1/2 <\beta <1$,  a maximal $\e^\beta$-separated net $\{p_i(\e)\}_{i=1}^{J(\epsilon)}$.
Assuming that 
$\inj_M > 3 \e^\beta$, where $\inj_M$ is the injectivity radius of $(M, g)$, we
see using simple volume-based arguments that 
\begin{equation} \label{J}
 J(\e) \leq C 
 \vol(M) \e^{-\beta n}.
\end{equation}

\begin{Theorem}
\label{thm:usable_partition_of_unity}
Let $(M,g)$ be a closed Riemannian manifold, and constants $1/2<\beta<\alpha<1$ such that 
$\e^\beta$ and $\e^\alpha$ satisfy the conditions in Theorem \ref{thm:partition_of_unity} for $\varepsilon$ and $\delta$ respectively. Let $\{p_i\}$ be a maximal $\e^\beta$-separated net with associated Voronoi decomposition $D_i$.
Then
\begin{enumerate}%[label=\arabic*}]
\item $\{\psi_i\}_i$ is a smooth partition of unity with 
\[ 
\psi_i=1 \text{ in } D_i^{ -},
\qquad
\psi_i=0   \text{ in } 
 M\setminus D_i^{ +},\] where  we denote $D^+_i=D_i^{\e^\alpha, +}$ and $D^-_i=D_i^{\e^\alpha, -}$

\item 
For each $i$, there is an inclusion $\{\nabla \psi_i\neq 0\,\} \subset D_i^{+} \setminus D^{-}_i$, hence
\begin{equation} \label{eq:volume of one gradient}
\vol\{\, \nabla \psi_i\neq 0\,\}
\leq
C'\e^{\beta(n-1)+\alpha}
\end{equation}
for some $C'>0$.
\item We have an inequality
\begin{equation} \label{eq:volume of all gradients}
 \vol( \cup _i\supp \nabla \psi_i  )\leq  C''\e^{\alpha-\beta}.
\end{equation} 
\item For any multi-index $\gamma \in \mathbb{Z}^n$,
\begin{equation} \label{eq:size of iterated gradient}   
  | D^\gamma \psi_j (x)| \leq C_\gamma \e^{-\alpha|\gamma|}.
\end{equation}   
\item \begin{equation}\label{eq:psijpsil}
\{ p: \psi_j \cdot \psi_{\ell}(p)\neq 0 \} \subset \{ \nabla \psi_j \neq 0 \}
\end{equation}
\item \begin{equation}\label{eq:psi^2}
\{ p: \psi_j - \psi_{j}^2(p)\neq 0 \} \subset \{ \nabla \psi_j \neq 0 \}
\end{equation}

\end{enumerate}
\end{Theorem}

\begin{proof}
    The first four parts of the Theorem are a direct application of Theorem \ref{thm:partition_of_unity}; the last two follow immediately from the construction. 
\end{proof}

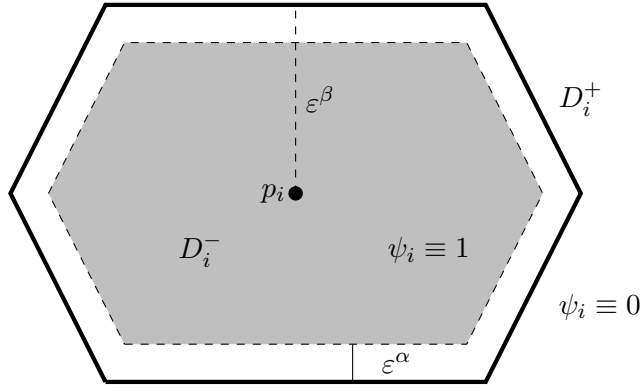
\begin{figure}
    \centering
        \begin{tikzpicture}[scale=2.5]
    \draw[ultra thick] (-2,-1) -- (0,-1) -- (0.5,0) -- (0,1) -- (-2,1) -- (-2.5,0) -- (-2,-1);  
    \draw[dashed, fill=lightgray] (-1.9,-0.8) -- (-0.1,-0.8) -- (0.3,0) -- (-0.1,0.8) -- (-1.9,0.8) -- (-2.3,0) -- (-1.9,-0.8) ;
    \draw[fill] (-1,0) circle [radius=1pt];
    \node[left] at (-1,0) {$p_i$};
    %\draw[dashed] (-1,0) -- (-1,-1);
     \draw[dashed] (-1,0) -- (-1,1);
     \node[right] at (-1,0.5) {$\varepsilon^\beta$};
     \node[right] at (-0.6,-0.9) {$\varepsilon^\alpha$};
     \draw (-0.7,-0.8) -- (-0.7,-1);
     \node at (-1.5,-0.3) {$D_i^-$};
      \node at (0.5,0.5) {$D_i^+$};
       \node at (-0.3,-0.3) {$\psi_i\equiv 1$};
       \node at (0.6,-0.6) {$\psi_i\equiv 0$};
      \end{tikzpicture}
    \caption{A Voronoi domain and its corresponding function in the partition of unity.}
    \label{fig:Voronoi_partition_unity}
\end{figure}

\part{Two scale convergence}\label{part:twoscale}

 \section{Oscillating test functions} \label{sec;testfunctions}

 The key concept of the two-scale convergence  in a domain $\Omega \subset \mathbb{R}^n$
 is associated with a fixed
 cell $Y$  and smooth  functions $f : \Omega \times \R^n\to\R$  that are periodic on the second variable with
 period $Y$. Such functions generate  
 test functions simply as $f_\epsilon(x)=f(x,\frac{x}{\epsilon})$. 
 However, on an  manifold there is, in general, no such concept as periodicity.  
 Our approach consists on ``pulling back" the notion of periodicity, with respect to the lattice
 $\Gamma$, from $\TM$ to $TM$ to $M$. 
%  Different tensors defined in $TM$ which are $\Gamma_p$-periodic
% on the $v$ variable are said to be defined on $\TM$.  
 The diffeomorphisms (onto its image) $H_{\epsilon,j}: D_j \to T_{p_j}M$ where $D_j$ is a small enough neighbourhood of $p_j$ (as the one provided by the Voronoi decomposition) and defined as 
 \[ H_{\epsilon,j}(q)= \frac{\exp_{p_j}^{-1}(q)}{\epsilon} \]
 will play a crucial role.

 Given a function $f:\TM\to\R$, denote by $\tilde{f}:TM\to\R$ the function defined as $\tilde{f}:=f\circ\Phi$, where $\Phi:TM\to\TM$ is the quotient map  that sends each fiber $T_pM$ to the torus $\TM_p$ after modding out the lattice defined by the frame $\{e_i(p)\}$ (see Section \ref{sec:torus_bundle}).

 The following definition gives a method to generate rapidly oscillating functions in the manifold $M$ from functions in the torus bundle $\TM$.
 \begin{defn}
 \label{def:test_functions} 
 Let $f:\TM\to\R$ be a  function in $C(M,\mathcal{B}(\TM_p))$. For any $\epsilon>0$  small enough, we define 
 $f^\epsilon:M\to\R$ as 
\begin{equation} 
\label{test_function}
 f^\epsilon(q)
 =
 \sum_j \psi_j(q)  \tilde{f}(p_j, %\frac{1}{\epsilon}  
 \exp_{p_j}^{-1}(q)/\epsilon)=
 \sum \psi_j(q) H_{ \epsilon,j}^* \tilde{f}(p_j,\cdot)(q), \quad q \in M.
\end{equation}
 \end{defn}
 
  In the definition of two scale convergence we just require smooth test functions. However, we can enlarge the class of test functions
with the following notion

\begin {Def} 
\label{def:admissible_function}
A function   $f\in  C(M,{\mathcal{B}}( \TM_p))$  is  \emph{an  admissible test  function}
if   the sequence $f^\epsilon$ satisfies
\begin{equation}\label{seq1}
\int_M | f^\epsilon(q)|^2\, dq \leq  
C\, \|f \|^2_{C(M,{\mathcal{B}}( \T_pM))}
\end{equation} 
for some $C>0$, and
\begin{equation}\label{seq2}
\lim_{\epsilon \to 0}\int_M  |f^\epsilon(q)|^2\,dq=  \oint_{\TM}\,
|f (p,v)|^2 \,dv\,dp
\end{equation}
where we use for $v\in\T_pM$ the coordinates associated to the frame, as in \eqref{framing  coordinates}, and $\vol(\TM_p)=g^{1/2}$ for $g(p)=\det g_{ij}(p).$
We say that the  sequence  $f^\epsilon $ { \bf is realized  by   }$f$.
\end{Def}

 We first  introduce the following
differential operator, which will often be encountered after integration by parts.

\begin{Lemma} 
\label{lem:P_kf}
Let $k \in \mathbb{Z}^n$. For $f:\R^n\to\R$, we set $P_kf$ as 
\[ 
P_kf= \frac{\ip{k}{\nabla f}}{2\pi i |k|^2} .
\]
Then it holds that 
\[ |P_kf| \le \frac {|\nabla f|}{2\pi|k|}.
\]
\end{Lemma}

Many of our arguments rely on the fact that after integrating by parts,  we can discard Fourier modes with non 
constant character.  This is a consequence of  the following basic estimate, which follows from the properties
of the partition of unity. Since it will appear recurrently we state it as a separate lemma, where we need to explain some notation: suppose $\psi_j$ is one of the functions belonging to the partition of unity constructed in Section \ref{sec:partition}; assume that its support lies inside the $\e^\alpha$-neighbourhood of the Voronoi domain $D_j$ corresponding to the point $p_j$.  
Then $\exp_{p_j}^{-1}$ is well-defined in 
$\supp \psi_j$,
 and therefore we can define unambiguously 
$
  f(\exp_{p_j}^{-1}q)\cdot\psi_j(q)
$
for any function $f:M\to\R$ by extending it as zero outside of $\supp\psi_j$. 

\begin{Lemma}
\label{lem:byparts1} 
Let $\psi_j$ be one of the functions belonging to the partition of unity and $\alpha$, $\beta$ as in Theorem \ref{thm:usable_partition_of_unity}. Then for every $k \in \mathbb{Z}^n$, and $N \in \mathbb{N}$, there exists $C_N>0$ such that

\begin{equation} \label{eq:E1}
\left|\, \int_M\psi_j(p) \,e^{2\pi i k\cdot \exp_{p_j}^{-1}(p)/\epsilon}\,dp \,\right|
 \leq C_N |k|^{-N} \e^{\beta(n-1)+N(1-\alpha)+\alpha}. 
\end{equation}
\end{Lemma}

\begin{proof}

We write the above integral in the chart 
$\exp_{p_j}:T_{p_j}M\simeq\R^n\to M$ (where we are reducing the domain and codomain to open sets, where the exponential is a diffeomorphism). Then, using $v\in\R^n$ as coordinate and $g:=g(v)=\det g_{ij}(v)$, the left-hand side integral is written as
\begin{equation} \label{eq:local_E1}
 \int_{\R^n}\,\psi_j(\exp_{p_j}v)\, g^{1/2}(v)\,e^{2\pi i k\cdot v/\epsilon}\,dv 
\end{equation}

Due to the compact support of $\psi_j$, we can integrate by parts $N$ times to get
\begin{equation}
\begin{split} 
\label{by_parts0}
\left|\,
\int_{\R^n}\psi_j(\exp_{p_j}v)\, g^{1/2}(v)\,e^{2\pi i k\cdot v/\epsilon}\,dv
\,\right|
\\
=\left|\,
 (-\e)^N\int_{\R^n}  P_k^N 
\left[ \psi_j(\exp_{p_j}v)\, g^{1/2}(v) \right] e^{2\pi i k\cdot v/\epsilon}\,dv\,
\right|
 \\ \le C \left(\frac{\epsilon}{|k|}\right)^{N}
 |\{x: D \psi_j \neq 0\}|\,
  \sup |D^N(\psi_j)|
\\  \le C\left(\frac{\epsilon}{|k|}\right)^{N}  \epsilon^{-N \alpha}\,  \epsilon^{\beta(n-1)+\alpha}
 \le C_N|k|^{-N} \epsilon^{N(1- \alpha)+n\beta+\alpha-\beta}
\end{split}
\end{equation}
where we have used part \eqref{eq:volume of one gradient} in Theorem \ref{thm:usable_partition_of_unity}, equation \eqref{eq:size of iterated gradient} and the uniform bounds on the derivatives of the metric
and the exponentials. 
\end{proof}

To study the admissibility condition, we start with an important result:
 \begin{Lemma}[Riemann-Lebesgue Lemma] 
 \label{RL} 
 Let $f \in C(M, L^1(\TM_p))$ and
 $f^\epsilon$ as in \eqref{test_function}. Then it holds that
$$
\lim_{\epsilon \to 0} \int_{M}\, f^\epsilon (q)\, dq=\oint_{\TM}\,  f[q,v]\, dv\,dq.
$$
\end{Lemma}
 \begin{proof}
Since $\epsilon$ is small enough, we can  use normal coordinates in $M$  associated with $p_j$. Moreover,
in $T_{p_j}M$ we use the coordinates $\{v^{i}\}_{i=1}^n$ given by the frame $\Gamma_{p_j}$.
Then 
we have
 \[ 
 \int_M\, \psi_j(q)\, \tilde{f}(p_j, \exp_{p_j}^{-1}(q)/\epsilon)\, dq
 =\int_{\R^n}\psi_j(\exp_{p_j}v)\,\tilde{f}(p_j, {v}/\epsilon)\,
g^{1/2}(v) \,dv,
\]
where $g(v)=g_j(v)$ is the Jacobian of the Riemannian normal coordinates centered at $p_j$. Recall from Section \ref{subsec:normal_charts} that 
$g(v)=1+O(|v|^2)$.

Since $\tilde{f}(p, v) $
is  $\mathbb{Z}^n$-periodic with respect to the frame, we take its  Fourier series 
\[
\tilde{f}(p,v)= \sum_{k \in \mathbb{Z}^n}\hat{f}(p; k)\,e^{2\pi i k\cdot v},
\] 
where recall that 
\[
\hat{f}(p; k)=\int_{[0,1]^n}\,f(p,\xi)\,e^{-2\pi i k\cdot \xi}\, d\xi. 
\]
Then,  from the  weak convergence of the Fourier series, 

\begin{equation}
\int_M\,\psi_j(q)\,\tilde{f}(p_j,\exp_{p_j}^{-1}(q)/\epsilon)\,dq = 
\\
\int_{[0,1]^n}\psi_j(\exp_{p_j}v) \sum_k \hat f_j(k)\,e^{2\pi i k\cdot v/\epsilon}\,g^{1/2}(v)\,dv,
\end{equation}
where we abbreviate $\hat f_j(k)=\hat f(p_j; k)$.
We split the right-hand side as 
\begin{align}
&\int_{[0,1]^n}\psi_j(\exp_{p_j}v)\, \hat f_j(0)\,g^{1/2}(v)\,dv + \sum_{k\neq 0} E_{k, j}(\epsilon) \nonumber
\\ &= \int_{M}\psi_j(q)\left( \int_{[0,1]^n} f(p_j, \xi) \,d\xi \right) dq+ \sum_{k\neq 0} E_{k, j}(\epsilon).
\end{align}

Now, the terms with $k\neq 0$ are estimated   by a direct application of Lemma~\ref{lem:byparts1}.   We obtain
that 

\[ 
|E_{k,j}| \le 
C_N'\, \epsilon^{\beta(n-1)+1}   \frac{ |\hat{f}(p_j,k)|}{|k|^{N}}\]
where $C_N'=\epsilon^{N(1-\alpha)+\alpha-1}C_N$. Notice that since $\alpha<1$ and $N\geq 1$, $N(1-\alpha)+\alpha-1\geq 0$.

On the other hand, using that $0\leq \psi_j\leq 1$ everywhere, we get that
\begin{equation}
\begin{aligned}\label{Lebesgue}
&\left|\, \int_M \int_{[0,1]^n}\psi_j(q) f(p_{j}, \xi)\, d\xi\, dq- 
\int_M\int_ {[0,1]^n}\psi_j(q) f(q, \xi)\, d\xi\, dq\,\right| \le \\ &
%c \e^{\beta n}
\sup_{q \in D^+_j} \left| \int_{[0,1]^n} f(p_{j}, \xi)\, d\xi -\int_{[0,1]^n} f(q, \xi)\, d\xi\,\right|
\end{aligned}\end{equation}

Thus using the finite overlapping property  (see Lemma \ref{lem:FOP}), and the fact that $ d\xi=  g^{-1/2}dv$ we get 
\begin{multline}
\left|\int_M f^\epsilon dq- \oint_{\TM} \,  f(q,v)\,dq\,dv\,\right| 
\le  
 C_N\e^{\beta(n-1)+1} \max_{p\,\in M} \sum_{|k|  \geq 1} \frac{|\hat f(p; k)|}{|k|^N}\\
  + \max_{p,q \in M,\, d(p,q)\leq\epsilon^\beta}  \left|\int_{[0,1]^n} f(p,\xi)\,d \xi -\int_{[0,1]^n} f(q,\xi) d \xi\right|.
\end{multline}

Now $f(p,v) \in C(M,L^1(\TM_p))$ implies that   the second term tends to zero by definition. It also implies
that    the Fourier coefficients are bounded and thus taking $N=n+1$ the sum in the first term is bounded and hence
tends to zero as $\epsilon$ goes to zero.
\end{proof}

The definition of test functions in a manifold seems a priori  very rigid, in difference to the Euclidean situation. For example, in $\R^n$, 
$f(x,\frac{x}{\epsilon}) g(x,\frac{x}{\epsilon})=fg(x,\frac{x}{\epsilon} )$ and thus $f^\epsilon g^\epsilon=(fg)^\epsilon$. The next lemma
shows that this is still the case in manifolds up to a vanishing error, and thus test functions behave like an algebra.

\begin{Prop}[Test functions are almost an algebra]\label{prop:algebra} Let $f,g \in C(\TM)$. Then 

 \begin{equation}\label{eq:approxsquares}
%\left| \int_M f^\epsilon g^\epsilon dq -\int (fg)^\epsilon dq \right| \le C \epsilon^{\alpha-\beta}
\lim_{\e\to 0}\left| \int_M f^\epsilon g^\epsilon dq -\int (fg)^\epsilon dq \right| = 0.
 \end{equation}

\end{Prop}
\begin{proof}
Fix $\epsilon$ and write $f^\epsilon=\sum \psi_j f_j$, $g^\epsilon=\sum \psi_\ell g_\ell$ 
where 
\[f_j(q)= f(p_j, \frac{1}{\e} \exp^{-1}(q)), \qquad g_\ell(q)= g(p_\ell, \frac{1}{\e} \exp^{-1}(q)),
\]
so that, 
\begin{equation} \label{fe}
\int_M  f^\epsilon g^\epsilon dq=\int _M \left(\sum \psi_j f_{j}\right)
\left(\sum \psi_\ell g_{\ell} \right)\,dq.
\end{equation}
Now Parts (3) and (5) in Theorem \ref{thm:usable_partition_of_unity}, combined with the finite overlapping property of the Voronoi domains imply that  

\begin{equation} 
\label{off-diagonal}
\left|\,\int_M\, \sum_{j \neq \ell} \psi_j\psi_\ell f_{j}   g_{\ell} \,dq\, \right|
\leq 
\|f\|_{C(\TM)} \|g\|_{C(\TM)} \cdot \int_M\, \sum_{j \neq \ell} \psi_j\psi_\ell  \,dq 
\\ 
 \leq C\epsilon^{\alpha-\beta} \|f\|_{C(\TM)} \|g\|_{C(\TM)}
\end{equation}

Thus, we are mainly concerned with computing the limit of $\sum_j \psi^2_j  f_jg_j$.  In order to do that, we now combine parts (3) and (6)
of Theorem~\ref{thm:usable_partition_of_unity} to estimate
\begin{equation} 
\label{remainder1}
\int_M \sum_j (\psi_j(q)-\psi^2_j(q)) f_jg_j (q)\, dq \leq C \e^{\alpha- \beta}\, \|f\|_{C(\TM)} \|g\|_{C(\TM)} 
\end{equation}
Therefore 

\[ \lim_{\epsilon \to 0}  \left|\int_M  f^\epsilon g^\epsilon dq- \int_M \sum \psi_j f_j g_j dq \right| \to 0 \]
since $(fg)_j=f_jg_j$ and $(fg)^\epsilon=\sum \psi_j f_j g_j$. Thus,  \eqref{eq:approxsquares} is proven. 
\end{proof}
\begin{Rem} We warn the reader that indeed for an analytic function $F:\mathbb{R}^N \to \mathbb{R}$ and $f_1, \ldots, f_N$ smooth test functions, the above argument implies
that 
\[ \lim_{\epsilon \to 0} \int |(F(f_1,\ldots, f_N))^\epsilon- F(f_1^\epsilon,\ldots, f_N^\epsilon)|\, dp=0 \]

\end{Rem}

If we take $g=f$ the above proposition implies that any  $f \in C(\T M)$ is an admissible function  but we can have unbounded test functions as well. The technical detail  is that
instead of  $|\int_{U} f dq | \le \|f\|_{\infty} \vol(U)$, we
need   the following lemma, which takes advantage of the periodicity of $f$. For a set
$U \in \mathbb{R}^n$, we will denote by $U_\e=\{y \in \mathbb{R}^n \, \dist(y,U) \le \e\}$ its $\e$-tubular neighborhood. 

\begin{Lemma}
\label{lem:periodic}
Let $f:\mathbb{R}^n \to \mathbb{R}$ be $[0,1]^n$-periodic and in $L^1([0,1]^n)$.  Then for any open set $U$,
\[ \int_U\, f(x/\e) \, dx \le \|f\|_{L^1([0,1]^n)} \vol(U_\epsilon) \]
\end{Lemma}
\begin{proof}
We can assume, replacing $f$ by $|f|$ if necessary, that $f$ is nonnegative.
Declare $y=x/\e$. Then 
\[\int_U f(x/{\epsilon})\, dx=\epsilon^n \int_{\frac{1}{\epsilon}U} f(y)\, dy\]
By the periodicity of $f$, this is bounded above by  
\[ N(\epsilon)\, \epsilon^n   \int_{[0,1]^n} f(y) dy\]
where $N(\epsilon)$ is the number of $[0,1]^n$ cells with corner in $\mathbb{Z}^n$ which intersect $\frac{1}{\epsilon}  U$. This is, exactly,   the number of $\epsilon$-cells which intersect $U$.  Such cells are disjoint and contained in $U_\epsilon$, thus $N(\epsilon) \epsilon^n \le \vol(U_\epsilon)$. The result follows.
\end{proof}

\begin{Theorem}
\label{thm:admissible_functions}
Any function $f \in C(M; \, L^2(\TM_p))$ is an admissible test function.
\end{Theorem}

\begin{proof}
We will first prove inequality \eqref{seq1} for this case. Once again, write $\tilde{f}:TM\to\R$ for the lifted function $f\circ \Phi$ with $\Phi:TM\to\TM$ the quotient projection, and let 
\[
f_j(q):=\tilde{f}(p_j, \exp_{p_j}^{-1}(q)/\e),
\quad
\hat{f_j}(v):=f_j\circ \exp_{p_j}(v),
\]
so that
\begin{equation}
\label{eq:split_sum}
({f^\epsilon})^2=
\sum_j \psi_j^2 f_j^2
+
\sum_{j\neq \ell} \psi_j\psi_\ell f_{j} f_{\ell}.
\end{equation}

For the first summand, after taking exponential coordinates at $p_j$ for each $j$-th term and integrating, its integral reads as
\begin{equation}\label{localization}
\sum_j \int_M \psi_j^2 f_j^2\,dq=
%\sum_j \int_M \psi_j^2(q)  f^2\left(p_j, \frac{1}{\e}\exp_{p_j}^{-1}(q)\right)\,dq 
%\\
\sum_j \int_{\supp(\psi_j\circ \exp_{p_j})} 
\psi_j^2\circ \exp_{p_j}(v)\, \hat{f_j}^2({v}/{\e})\, g_j^{1/2}(v)\,dv,
\end{equation}
where $g_j$ is the determinant  of the metric in exponential coordinates centered  at $p_j$.
By Lemma \ref{eq:metric_at_zero}, we can assume that $1/2<g_j(v)\leq 2$. Since  $\psi_j^2<1$ and $\e <<\e^\beta$, a direct application of Lemma~\ref{lem:periodic} yields
that 

\[ \int_M \psi_j^2 f_j^2\,dq \le C \epsilon^{\beta n} \| f(p_j, \ast) \|^2_{L^2 (\TM_{p_j})}\]

Finally, recall that $\{p_j\}$ is an $\e^\beta$-net; thus standard volume counting arguments show that, in an $n$-dimensional compact manifold,  
\begin{equation} \label{eq:mainterm}
\sum_j \int_M\,\psi_j^2 f_j^2 \,dq\leq C\vol(M) \sup_p\|f(p,\ast)\|^2_{L^2(\TM_p)}=C\vol(M)
\|f\|^2_{C(M,L^2(\TM_p))}
\end{equation}
for some constant $C>0$. 

In order to deal with the mixed terms in \eqref{eq:split_sum} and also to obtain \eqref{seq2} we need to deal carefully with tubular neighborhoods of  boundaries  of the  Voronoi domains. Recall that for each $j$, the second part in 
Theorem \ref{thm:partition_of_unity} implies that $\supp {\nabla \psi_j}$ is contained in  an $\epsilon^\alpha$ tubular neighbourhood of $\partial D_i$, a piecewise regular $(n-1)$-hypersurface of measure at most $C\epsilon^{\beta(n-1)}$. Since $\epsilon$ is small enough, the same holds for $\supp(\nabla \psi_j \circ \exp_{p_j}) \subset \mathbb{R}^n$  due to  Lemma~\ref{lem:almost Euclidean charts}. 
Hence, since $\epsilon<\epsilon^\alpha$, it follows that $\supp(\nabla \psi_j \circ \exp_{p_j})_\epsilon$ is contained in a tubular neighbourhood of radius $\epsilon^\alpha+\epsilon \le 2\epsilon^\alpha$ of the same $n-1$
hypersurface in Euclidean space. Thus, another use of the coarea formula yields that

\begin{equation}\label{eq:volume}
\vol[ ( \supp(\nabla \psi_j \circ \exp_{p_j})_\epsilon] \le C \epsilon^\alpha \epsilon^{\beta(n-1)}
\end{equation}

Therefore, for each pair $j\neq\ell$, 

\[\begin{aligned} \int_{M} \psi_j\psi_\ell f_{j} f_{\ell} dq &\le \int_{ \supp(\nabla \psi_j) \cap (\supp \nabla \psi_\ell) } 
\psi_j\psi_\ell f_{j} f_{\ell} dq \\ & \le \left(\int_{\supp(\nabla \psi_j)} \psi_j^2 f_j^2 dq\right)^{\frac 12} \left(\int_{\supp(\nabla \psi_{\ell})} \psi_{\ell}^2 f_{\ell}^2 dq\right)^{\frac 12}
\end{aligned}\]

Next, we treat each of the terms by localizing, using that $g_j$ is almost constant and that 
$\psi_j \le 1$, exactly as in \eqref{localization}. A direct application of Lemma~\ref{lem:periodic}  and the estimate \eqref{eq:volume} yields that

\[ 
\left(\int_{\supp(\nabla \psi_{\ell})} \psi_{\ell}^2 f_{\ell}^2\, dq\right)^{\frac 12} \le  \| f(p_{\ell}, \ast) \|_{L^2(\mathbb{T}_{p_{\ell}}M)}  (C\epsilon^\alpha \epsilon^{\beta(n-1)})^{\frac12} \]

Thus, 

\begin{equation}\label{eq:remainder}
\left|\int_M \sum_{j,\ell} \psi_j \psi_{\ell} f_j f_{\ell} dq\right| 
\le K\|f\|^2_{C(M,
L^2(\TM_{p_j})}  \sum_j C\epsilon^\alpha \epsilon^{\beta(n-1)} 
\le C\epsilon^{\alpha-\beta} \|f\|^2_{_{C(M,
L^2(\TM_p)}}
\end{equation}
where $K$ is the finite overlapping constant from Lemma~\ref{lem:FOP} where we have merged several constants independent of $\e$ in the final $C$.
Thus \eqref{eq:mainterm} and \eqref{eq:remainder} provides the proof of condition \eqref{seq1} on admissibility. 

Because of \eqref{eq:remainder}, in order to prove \eqref{seq2},  we just need to upgrade  \eqref{eq:mainterm}  to
\begin{equation}
\label{eq:squares_in_2}
\lim_{\e\to 0}
\int _M \, \sum_{j} \psi_j^2f_{j}^2 dq =
\oint_{\TM}\, f(p,v)^2 dv\, dp
\end{equation}
However notice that now \eqref{eq:psi^2} plays the role in \eqref{eq:psijpsil} and 
thus, we can repeat word by word the estimation on the cross terms to get that

\begin{equation} 
%\label{remainder1}
\int_M \sum_j (\psi_j(q)-\psi^2_j(q)) f_j^2(q)\, dq \leq C \e^{\alpha- \beta}\, \|f   \|^2_{C(M,L^2(\TM_p))}.
\end{equation}
Thus, instead of  \eqref{eq:squares_in_2}, we show that 
 \begin{equation}
\label{eq:power1_in_2}
\lim_{\e\to 0}
\int _M \, \sum_{j} \psi_j f_{j}^2 =
\oint_{\TM}\,
 f(p,v)^2 dv\, dp.
\end{equation}
But when  $f\in C(M;L^2(\TM_p))$, then 
$f^2\in C(M;L^1(\TM_p))$ and the limit in \eqref{eq:power1_in_2} follows from Proposition \ref{RL}.

\end{proof}

\section{Two-scale convergence of functions} \label{sec:2sfunctions}
 \begin{Def}
 \label{def:two-scale}
A sequence of functions $u_\epsilon \in  L^2(M)$  two-scale converges to $u_0 \in  L^2(\TM)$ 
if $u_\epsilon$ are uniformly bounded in $L^2(M)$ and,  for any test function  $f \in C_0^\infty(\TM)$,
\begin{equation}
\label{gamma}
\lim_{\epsilon \to 0}\int_M\, u_\epsilon f^\epsilon dq= \oint_{\TM} u_0(p,v) f(p,v)\,dv\,dp,
\end{equation}
where $f^\e$ is as in Definition \ref{def:test_functions}.

We will denote two-scale convergence by 
\[
u_\epsilon \2s  u_0. 
\]
If, in addition, $u_\epsilon$ satisfies that

\begin{equation}
\lim_{\epsilon \to 0} \int_M |u_\epsilon|^2\,dq=\oint_{\TM}\,|u_0(p,v)|^2\, dv\, dp,
\end{equation}
then we say that $u_\epsilon$ 
\emph{strongly two-scale converges to} $u_0$
and denote it as  
\[ u_\epsilon \S2s u_0 \]
%where $\mathbb T_p$ is the cell $T_pM/\Gamma(p)$.
\end{Def}

\begin{Rem}
Notice that, in general, two-scale convergence does not imply strong two-scale convergence. A counterexample is given in the circumference $\mathbb{S}^1$ with the standard framing when $u_\e(p)=\sin{(p/\e^2)}$. A simple computation, using the $\e$-net placed at points $p_j=j\e$, yields that $u_\e\2s 0$, while $ \int_{\mathbb{S}^1} |u_\epsilon|^2\,dp\to\pi.$
\end{Rem}

Our definitions already provide a wealth of two-scale convergent sequences.
\begin{lem}[Examples of two scale convergent sequences]\label{lem:examples}
Let 
$f,g:\TM\to \R$, $h:M\to \R$ functions, and $\pi:\TM\to M$ the bundle projection.

\begin{enumerate}
 \item If $h \in C(M)$, consider the constant  sequence $h_\epsilon=h$. Then
\begin{equation} \label{lem;constant} h_\epsilon \S2s h\circ \pi \end{equation}
\item If $g \in C(M,L^2(\TM_p))$, then, 
\begin{equation} g^\epsilon \S2s g 
\end{equation}
\item  If $f,g \in C(\TM)$, then 
\begin{equation}\label{lem:product} f^\epsilon g^\epsilon \S2s fg 
\end{equation}
\end{enumerate}
\end{lem}

\begin{proof}
We start with (1). 
Since 
\[\int |h|^2 dp=\oint_{TM} |h(p)|^2\,dv\,dp, 
\]we only need to deal with two-scale convergence. Then for any $f(p,v) \in C^\infty(\TM)$
\[
\lim_{\epsilon \to 0} \int_M h f^\epsilon= \int (h(p)f(p,v))^\epsilon\, dq+ \int_M \sum_{j} \psi_j(q) (h(q)-h(p_j)) f(p_j,\exp_{p_j}q/\e)\, dq 
\]
By continuity of $h$, $|h(q)-h(p_j)|=O (|q-p_j|)$; thus, the second term in the right-hand side approaches zero, and the result follows from Lemma \ref{RL}. 

As for (2), the two-scale convergence follows from  Proposition~\ref{prop:algebra} and the Riemann Lebesgue Lemma~\ref{RL} (or directly from \ref{RL} applied to $f\pm g$),  
 and the strong convergence follows from Theorem~\ref{thm:admissible_functions}.

Finally, part (3) follows from part (2), Proposition \ref{prop:algebra}, and the boundedness of $f$ and $g$. Alternatively, given any $\hh \in C(\TM)$ we apply Proposition~\ref{prop:algebra}
twice to get that 

\[ 
\lim_{\epsilon \to 0} \int_M\,f^\epsilon g^\epsilon \hh^\epsilon\, dq=
\lim_{\epsilon \to 0} \int_M\, (fg\hh)^\epsilon\, dq=\oint_{\TM}\, fg\hh \,dv\,dp 
\]
where the last equality follows from  Lemma~\ref{RL}. As for the strong convergence, we are concerned with 
 the limit of $\int\, (f^\epsilon g^\epsilon)^2\, dp$. Applying Proposition~\ref{prop:algebra} for consecutive times, we get 
that

\[ 
\lim_{\epsilon \to 0} \int_{M}\,  |(f^\epsilon)^2 (g^\epsilon)^2 - (f^2g^2)^\epsilon|\, dq \to 0 
\]
and since $f^2g^2$ is smooth, the strong convergence follows again from Lemma~\ref{RL}.
\end{proof}

\begin{lem} Let $g_\epsilon, h_\epsilon \in L^2(M)$. Suppose that 
\begin{equation}
    \label{eq:L2_convergence to same limit}
\lim_{\epsilon \to 0} \int_{M}\, |g_\epsilon -h_\epsilon|^2\, dq \to 0.
\end{equation}
% If $g_\epsilon$ two-scale converges (strongly two-scale converges) to $g$, then $h_\epsilon$ two-scale converges (strongly two-scale converges) to $g$.
Then $g_\e$ and $h_\e$ have the same  two-scale limit, in case the limit exists. The statement also holds for strong two-scale convergence.
\end{lem}
\begin{proof}
For two-scale convergence, we take $f \in C(\TM)$. Then since $f$ is bounded
\[ \lim_{\epsilon \to 0} \left|\int_{M} g_\e f^\epsilon dq -\int h_\epsilon f^\epsilon dq \right|=0 \]
(Recall that $M$ is compact and thus $L^2$ convergence implies $L^1$). For the strong limit, simply notice that,

\[ \lim_{\epsilon \to 0}\int (|g_\epsilon|^2-|h_\epsilon|^2 )\,dq= 0 \]
\end{proof}
There are many variants of the result; for example, if there is $L^1$ convergence to zero replacing \eqref{eq:L2_convergence to same limit}, we get the same two-scale limit, which will be a strong two-scale limit if both sequences have equiintegrable $L^2$ powers.

Next we prove the compactness of $L^2$ with respect to two-scale convergence. The proof is very similar to that of  the Euclidean space.

\begin{Prop}
\label{prop:L2_bounded_2_scale}
 Assume that $u_\epsilon$ is a bounded sequence in $L^2(M)$.  Then there 
exists a subsequence $u_{\epsilon_k}$ and a function 
%$u_0 \in L^2(M; L^2(\TM_p))$
$u_0 \in L^2(\TM)$
 such that $u_{\epsilon_k}$  
two-scale converges to $u_0$.
\end{Prop}

\begin{proof} For any admissible test function, let us define
%$T_\epsilon (f)=\int_M u_\epsilon f^\epsilon\, dq$.
\[
T_\epsilon:{\mathcal C}(\TM)\to\R,\qquad T_\epsilon (f)=\int_M u_\epsilon f^\epsilon\, dq
\]

Since $f$ is admissible, condition \eqref{seq1} in Definition \ref{def:admissible_function} yields
\begin{equation}
\label{measure}
|T_\epsilon(f)|\leq (\int_M| u_\epsilon|^2)^{1/2}(\int_M| f^\epsilon|^2)^{1/2} 
%\\
\leq C \|u_\epsilon\|_2\|f\|_{{\mathcal C}(\TM)}
\leq C'\|f\|_{{\mathcal C}(\TM)}
%\leq
% C\|f\|_{{\mathcal C}(M; L^2(\TM_p))}.
\end{equation}
Hence $T_\epsilon$ is  in the dual of ${\mathcal C}(\T M)$ which can be identified with ${\mathcal M}(\T M)$.
Thus,
  there  exists a Borel measure $ \mu_\epsilon$ on $\T M$   such that 
$T_\epsilon (f) =  ( \mu_\epsilon, f )$ (dual pairing). 

Since  ${\mathcal C}(\T M)$  is  separable, and the family $\mu_\epsilon$ is uniformly bounded in its dual,  there exists  a subsequence $\mu_ \epsilon $  and  a measure $\mu_0$  such that  
$\mu_\epsilon \to 
 \mu_0, $ in the weak* topology of ${\mathcal M}(\T M)$. This means that there exists $\mu_0$ such that
$$
\langle \mu_0,f\rangle = \lim_{\e \to 0}\, \langle \mu_\epsilon,f\rangle.
$$
Moreover, using \eqref{seq2}, we also get
\begin{equation}
\label{eq:measure_bound_inL2}
|T_\epsilon(f)|\leq (\int_M| u_\epsilon|^2)^{1/2}(\int_M| f^\epsilon|^2)^{1/2}\leq
% C''\|f\|_{{\mathcal C}(M; L^2(\TM_p))}
C''\|f\|_{L^2(\TM)}
%{{\mathcal C}(M; L^2(\TM_p))}
\end{equation}

%Moreover, going to the limit in \eqref{measure} and using \eqref{seq2}, we see that,
% for  an admissible test function $f$,
Thus, going to the limit in \eqref{eq:measure_bound_inL2}, we see that, for an admissible test function $f$,
\[
|\langle \mu_0,f\rangle| \leq C \|f\|_{L^2(\T M)}
\]
%From the  above  we obtain that for  an admissible test function $f$
%$$|\langle \mu_0,f\rangle | \leq C\|f\|_{{\mathcal C}(\T(M))}$$  
Since the class of the admissible test-functions is dense in $L^2(\T M)$, the limit measure is 
given by a function $\tilde u_0 \in  L^2(\T M)$. Hence,   
\[
\langle \mu_0,f\rangle= \int_{\T M}\, \tilde u_0 f\,dv\, dp.
\]
Take $u_0=\vol (\TM_p )\tilde u_0$.
 \end{proof}
 
 The next lemma relates two-scale convergence with the usual weak convergence. 
 \begin{Lemma} \label{lem:extra}
 Let $v_\epsilon $ be a bounded  sequence in $ L^2(M)$ that  two-scale converges  to 
 $v \in L^2(\T M)$. Then $v_\epsilon $ converges  weakly in $L^2(M)$ to the function 
 $\tilde{v}(p)={\avint}_{\TM_p} v(p,\xi)\,d\xi$.
 \end{Lemma}
\begin{proof} 
Let $\phi \in C^1(M)$. Lift it to a function $f_\phi \in C^1(\TM)$ by $f_\phi(p, \xi)=\phi(p).$ 
Then
\begin{equation}
\label{eq:sequence_in_L2_convergence}
\int_M v_{\e}(q) f_\phi^{\e}(q) \,dq \to \oint_{\T M} v(p, \xi)\, \phi(p)\,d\xi\, dp=
\int_M \left({\avint}_{\TM_p} v(p, \xi)\,d\xi \right)\,\phi(p)\,dp.
\end{equation}

Observe that
\[
f_\phi^{\e}=\sum_j \psi_j \phi(p_j),
\quad
\text{ and }
\quad
f_\phi^{\e}-\phi=\sum_j \psi_j\cdot (\phi(p_j)-\phi).
\]

Since $\supp \psi_j\subset B(p_j,\e^\beta)$ and the $\psi_j$'s add to one,
\[
|f_\phi^{\e}-\phi|\leq \e^\beta \|\phi\|_{C^1(M)}
\]
Thus $\|f_\phi^{\e}-\phi\|_{L^2(M)}\leq C \e^\beta \|\phi\|_{C^1(M)}$, and it  follows from \eqref{eq:sequence_in_L2_convergence}
that
\[
\int_M v_{\e}(q) \phi(q) \,dq \to \int_M \tilde{v}(q) \phi(q)\, dq,   
\qquad \tilde{v}(q)= 
{\avint}_{\TM_q} v(q, \xi)\,d\xi.
\]
At last, due to the density of $C^1(M)$ in $L^2(M)$ and the uniform boundedness of $\|v_{\e}\|_{L^2(M)}$,
the above equation implies the desired result.
\end{proof}

%================================
% Ejemplo de dependencia con respecto a la red

\begin{Rem}\label{rem:old2scale}
In the Euclidean setting, the  classical two-scale convergence uses test functions of the form $f(x,\frac{x}{\epsilon})$ which are $[0,1]^n$-periodic with respect to the canonical lattice (placed at the origin). However, these choices are amenable to changes.
For instance,  the lattice of periodicity could change from point to point and the origin could be shifted. For a smooth 
$GL(\mathbb{R}^n)$-valued function $A(x)$ and a sequence $x_0^\epsilon$, we could define, for $f$ as above,  test functions   $f^\epsilon(x)=f(x,A(x)[\frac{x-x_0^\epsilon}{\epsilon}])$.  Our choice of
the underlying net  provides yet  an additional degree of freedom. 

Notice that by a careful choice of the net in the definition of Voronoi domains, our definition of two--scale convergence in the Euclidean setting coincides with the classical one. For example, given $\frac12<\beta<1$,  one considers the net with 
coordinates   $\ell \epsilon^\beta$, for $0\le \ell \le \epsilon^{-\beta}$. Then, if we  choose $\epsilon$ so that
$\epsilon^{-\beta} \in \mathbb{N}$,  it is straightforward to see that definition~\ref{def:two-scale} agrees with the classical notion.
\end{Rem}
\begin{Rem}\label{rem:nonuniqueness}
In principle, although the weak limit in \ref{lem:extra} is always  unique, the two-scale limit may depend on the chosen sequences of nets. Indeed, 
for $M=S^1$, we can consider $N:=\e^\beta/2\pi$, $f[p,v]=\sin{v}$, and the nets given as
\[
p_\ell=\frac{2\pi\ell}{N}, 
\qquad
\tilde{p_\ell}=\frac{2\pi\ell}{N}+\frac{\pi}{2\e}.
\] 
In this case, we can see that $\tilde{f}^\e(q)$ in the second net, 2-scale converges to the function $v_0[p,v]=\cos{v}$.  
However, this corresponds to a simple shift of the origin in the classical definition of test functions in the two scale convergence.  
\end{Rem}

Next, we take a closer look  at  strong two scale convergence.

\begin{lem}\label{lem:approx}
Let $u_\epsilon \S2s  a$ and $f$ be admissible. Then, 
\[ \lim_{\epsilon \to 0} \int_M |u_\epsilon -f^\epsilon|^2dq=\oint_{\TM} |a-f|^2 dvdp \]
\end{lem}
\begin{proof}

Indeed, 
$$
\int_M|u_\epsilon-f^\epsilon|^2\,dq = 
\int_M|u_\epsilon|^2 \,dq +\int_M|f^\epsilon|^2\,dq-2\int_M u_\epsilon f^\epsilon\,dq
$$
$$
\to \oint_{\TM} |a |^2\, dv\,dp  + 
\oint _{\TM }| f |^2\,dv\, dp 
-2\oint _{\TM} a f \,dv\, dp 
$$
$$=\oint_{\TM} |a (p,\xi)-f(p,\xi)|^2\,dp\, d\xi,
$$
as stated.

\end{proof}

The following lemma is the  analogue of Lemma \ref {lem:examples} for sequences  of functions
\begin{Lemma}[Compensated compactness] \label{compensated}
 Assume that  $u_\epsilon $  strongly two-scale converges  to $u_0$  and $v_\epsilon$ two-scale
 converges to $v_0$.  
  
  \begin{enumerate}
  \item[(a)]
  Then 
 \[
 u_\epsilon(p) v_\epsilon(p)  \to\vol(\TM_p)^{-1} \int _{\TM_p}u_0(p, v)\,v_0(p, v)\,dv\quad \hbox{in} \quad{\mathcal D}'(M).
\] 
 
 \item[(b)] If  $u_0\in{ \mathcal C} (M;\, L^2(\TM_p))$, then  $\|u_\epsilon- u_0^\epsilon\|_{L^2(M) }\to 0$.
   \end{enumerate}

 \end{Lemma}
 
\begin{proof} 
 We want to show that for a test function $ \varphi\in {\mathcal D}(M)$,
 $$
\int_M\,\varphi\, u_\epsilon v_\epsilon\,dq \to
 \int _M \varphi \left(\vol(\TM_p)^{-1}\int_{\T_p}u_0(p,\xi)v_0(p,\xi)\,d\xi \right)\, dp.
$$
 Let $\psi_k(p,\xi)$ be a sequence in $\mathcal{C}^\infty_0(\T M) $
which is  $L^2(\T M)$-convergent to $u_0(p, \xi)$. It follows from lemma~\ref{lem:approx} that, 
\begin{equation} \label{29.1}
\lim_{\epsilon \to 0}\int_M|u_\epsilon-\psi_k^\epsilon|^2\,dq
 =\oint_{\T M}|u_0-\psi_k|^2\,dpd\xi.
\end{equation}

To continue, write
\begin{equation}
\label{eq:integral_in_compensated}
\int_M \varphi\, u_\epsilon v_\epsilon \,dq= 
 \int_M \varphi\, (u_\epsilon -\psi_k^\epsilon)v_\epsilon\,dq
 +\int_M \varphi\,  \psi_k^\epsilon v_\epsilon\,dq.
\end{equation}
 Due to the smoothness of $\varphi $, the function $\varphi(p)\, \psi_k(p, \xi)$ is a test-function.
Using  the two-scale-convergence of $v_\epsilon$, the last term in \eqref{eq:integral_in_compensated} converges to 
$\oint_{\T M}\varphi\,  \psi_k  v_0\,dp\, d\xi$ as $ \epsilon\to 0$.

Then 
$$
\left| \int_M \,\varphi\, u_\epsilon v_\epsilon\,dq-
\int_{M} \varphi(p)
\left(\vol(\TM_p)^{-1}\int_{\T_p} u_0(p, \xi) v_0(p, v)\,dv\right) dq\, \right|
$$
$$ 
\leq \left|\int_M \varphi\, v_\epsilon (u_\epsilon- \psi_k^\epsilon )\,dq \right|+
\left|\int_M \varphi\, \psi_k^\epsilon v_\epsilon\,dq-
\oint_{\T M} \varphi(p)\, \psi_k(p,v) v_0(p, v)\,dv \, dp\, \right|
+$$
$$
+\left|\oint_{\T M} \varphi\, v_0 ( \psi_k -u_0) \,dv \, dp\ \right|\leq $$
\begin{multline}
\leq \| \varphi v_\epsilon\|_{L^2(M)}\,\|u_\epsilon-\psi_k^\epsilon\|_ {L^2(M)} 
+\left|\int_M \varphi \psi_k^\epsilon v_\epsilon dq-\oint_{\T M} \varphi  \psi_k  v_0 \,dv \, dp \right| + \\
+ \|\vol(\TM_p)^{-1}\varphi v_0\|_{L^2(\T M)}\,\cdot\,\|u_0-\psi_k\|_{L^2(\T M)}.
\end{multline}

Now let $\delta >0$ be arbitrary, and define
$C$ as the maximum of the following two numbers $\|\vol(\TM_p)^{-1}\varphi v_0\|_{L^2(\T M)}$ and $\sup_{\e>0} \|\varphi v_\epsilon\|_{L^2(M)}$.
  Take also some $k$ such that $\|u_0-\psi_k\|_{L^2(\T M)}\leq \frac{\delta}{2 C}$.  Using \eqref{29.1},
we then choose $\e_0$ such that for $\e <\e_0$ we have 
$\|u_\e-\psi_k^\e\|_{L^2(M)}\leq \frac{\delta}{2 C}$.
Since $\phi(p) \psi_k(p, v)$ is a proper test-function, for a fixed $k$,
\[
\left|\int_M\,\varphi\, \psi_k^\epsilon\, v_\epsilon \, dq-\oint_{\T M}\varphi  \psi_k  v_0 \,dv \, dp\ \right| \to 0,
\quad \hbox{as}\,\, \e \to 0.
\]
Summarizing the above, we see that
$$ \lim_{\e \to 0}
\left| \int_M \varphi u_\epsilon v_\epsilon-\oint_{\TM} \varphi  u_0 v_0 \right|\leq
 \limsup_{\epsilon \to 0} 
 \left|\int_M \varphi \psi_k^\epsilon v_\epsilon-\oint_{\T M} \varphi  \psi_k  v_0 \right|
 +\delta= \delta,
 $$
 thus proving the first part of the Proposition.
 
	To prove the second statement, assume that $u_0$ is an admissible test function and argue as in the proof of \eqref{29.1} with $\psi_k=u_0$.
 
 \end{proof}

 \begin{Rem} \label{rem:testfunctions} Notice that  the above theorem implies that, if $v_\epsilon$ two-scale converges to $v_0$, then  for every admissible $f  \in C(M,{\mathcal{B}_p}( \TM))$, we have that
 \[ 
 \lim_{\epsilon \to 0} \int_{M} \varphi \, v_\epsilon \, f^\epsilon\, dp=\int_{M} \varphi  \vol(\TM_p)^{-1}\int_{T_p M} v_0 \, f \,dv \,dp 
 \]
for every $\varphi \in  {\mathcal D}(M)$. This follows  directly from part  (a) of Lemma~\ref{compensated}  because, as discussed in  Lemma  \ref{lem:examples},  $f^\epsilon$ strongly  two-scale converges to $f$.
 
 \end{Rem}

Compensated compactness lemma is stated for  bilinear quantities, i.e. it analyzes the limits of products $a_\epsilon, b_\epsilon$ of two sequences. In fact, if we have a family of sequences $\{a^{i}_\epsilon\}_{i=1}^n$
such that $a^{i}_\epsilon \S2s a^{i}$ and $b_\epsilon \2s b$, the same lemma describes the distributional limit of 
$\Pi_{i=1}^n a^{i}_\epsilon b_\epsilon$ thanks to our next result

\begin{Lemma}\label{lem: strong product} 
Let $a_\e $ and $b_\e$  be  bounded sequences  of functions  on $M$ which strongly two-scale converge to $a$ and  $b$ respectively, Then $a_\e b_\e$ strongly two scale converges to $ab$.
\end{Lemma}
\begin{proof}
Take  smooth $L^2(\T M)$  approximations $a_k$ and $b_k$ of $a$ and $b$ respectively.  Firstly, observe that  by boundedness of $a,b$ it holds
\begin{equation}\label{uno}
\lim_{k \to \infty} \oint_{\TM} |a_k b_k|^2 \,dv \, dp =\oint_{\TM} |ab|^2 \, dv \, dp
\end{equation}

As in the proof of Lemma~\ref{compensated}, strong two scale convergence and Lemma~\ref{lem:approx}  imply that 
 \begin{align}\label{eq:doubleaproximation}
  &\lim_\e \int_M|a_\e- (a_k)^\e|^2dq= \oint_{\T M}|a- a_k|^2\,dv\, dp 
   \textrm{ and }\\ &  \lim_\e \int_M|b_\e- (b_k)^\e|^2 dq= \oint_{\T M}|b- (b_k) |^2\,dv\,dp \end{align}

Now we write  \[ a_\epsilon b_\epsilon=(a_\epsilon-(a_k)^\epsilon+ a_k^\epsilon) (b_\epsilon-(b_k)^\epsilon+ (b_k)^\epsilon) \]
and using the $L^\infty$-boundedness of the sequences $a_\e, b_\e, f^\e_k$ and $g^\e_k$ together with  Cauchy-Schwartz inequality  and \eqref{eq:doubleaproximation}
we have that 

\begin{equation}\label{dos}
\limsup_{\epsilon \to 0} |\int_M |a_\e b_\e|^2 \, dq- \int_M |(a_k)^\e (b_k)^\e|^2\, dq | \le C(\|a-a_k\|_{L^2} + \|b-b_k\|_{L^2}), \end{equation}

On the other hand, by Proposition~\ref{prop:algebra}

\begin{equation} \lim_{\epsilon \to 0}  \int_M |(a_k)^\e (b_k)^\e|^2 \, dq= \oint_{\TM} |a_k b_k|^2 \,dv \, dp \end{equation}

Finally, we write, 

\[ \begin{aligned}|\oint_{\TM} |ab|^2 dq-\int_{M} |a_\epsilon b_\epsilon|^2 dq| &\le |\oint_{\TM} (|ab|^2-|a_kb_k|^2) \, |\\ &+ |\int_M |(a_k)^\epsilon (b_k)^\epsilon|^2 \, dq  - \oint_{\TM} |a_kb_k|^2|\\ &+|\int_M (|a_\epsilon b_\epsilon|^2-|(a_k)^\epsilon (b_k)^\epsilon|^2)  \,|
\end{aligned}\]

and conclude by choosing $k$ large enough to make the first and third term as small as we wish  by \eqref{uno} and \eqref{dos} and for such $k$ choose $\epsilon$ to make the second term small as well. 
	
\end{proof}

\section{Two-scale convergence of tensors}\label{sec:2stensors}
Along the entire section, $\{X_i\}_{i=1}^k \in \mathfrak{X}(M)$ will be a set of $k$ arbitrary vector fields,
and  $ \{\omega^j \}_{j=1}^m \in \Lambda^1(M)$ will be a set of  $m$ arbitrary $1$-forms. We will use the shortcuts $X_I$, $\omega^J$, where no confusion 
is possible.

\begin{Def}
\label{def:2stensors}
Let $T_\epsilon \in L^2 T^{k,m}(M)$ be a sequence of $(k,m)$-tensors, and $T \in L^2T^{k,m}(\V)$ be a vertical $(k,m)$-tensor.  Then  we say that
 $T_\epsilon$ two-scale (strongly) converges  to $T$,
\[T_\epsilon \2s T,
\qquad (T_\epsilon \S2s T)\]
if  for every $(X_I,\omega^J) \in \mathfrak{X}(M)^k \times \Lambda(M)^m$ the functions %for every $n,m$ tuple of vector fields and forms, \\
 $T_\epsilon(X_I, \omega^J)$ two-scale (strongly) converge to $ T(X_I^\up,{\omega^J}^\up)$.
\end{Def}

Notice that the convergence is invariant under raising or lowering indexes of forms and  vector fields. Thus, we can consider  $Y_\e\in\mathfrak{X}(M)$ as linear functionals on $\Lambda^{1}$
or   $Y_\e \in T^{0,1}(M)$  through   the scalar  product
\[
\langle Y_\e,X\rangle= Y_\e^b(X).
\]
Notice also that if $T_\epsilon \S2s T$, it follows that
\[ 
\lim_{\e \to 0} \int\, |T_\epsilon|^2 \, dq=\oint\, |T|^2 \,dv\,dp 
\]

\begin{lem}\label{lem:constanttensors} Let $T \in L^2T^{k,m}(M)$ and declare $T_\epsilon=T$. Then 
\[ T_\epsilon \S2s T^\up \]
\end{lem}

\begin{proof}
Let $T \in T^{k,m}(M)$ and $(X_I,\omega^J) \in \mathfrak{X}(M)^k \times \Lambda(M)^m$. If $\pi:\TM\to M$ denotes the bundle projection,
 then, by Lemma~\ref{lem:examples}, the constant sequence two-scale converges to 
 \[ 
 T(X_I,\omega^J) \S2s
 T(X_I,\omega^J)\circ \pi  =T^\up (X_I^\up,{\omega^J}^\up) 
 \]
which is precisely the definition. 

\end{proof}

Next, we state a compensated compactness lemma for tensor and vector fields. 
\begin{lem} [Compensated compactness for tensors]\label{lem:cctensors}
Suppose we have sequences  $A_\epsilon \in T^{2,0}(M)$, 
$X_\e, Y_\epsilon \in \mathfrak{X}(M)$, such that 
\[
A_\epsilon\S2s A, \quad 
X_\epsilon \2s X, \quad 
Y_\epsilon \S2s Y.
\]
Assume further than $A_\epsilon$ and $Y_\epsilon$ are uniformly bounded in $M$. 
%$\xrightharpoonup{smash{2s}}
 Then
 \[ 
 \lim_{\epsilon \to 0} \int_{M} \, A_\epsilon[X_\epsilon, Y_\epsilon]\, dq=\oint_{\TM}   A[X,Y]\, dv\,dp 
 \]
\end{lem}
\begin{proof} After expanding on the frame basis, we obtain that 
\[ 
\lim_{\epsilon \to 0} \int_{M} \,A_\epsilon [X_\epsilon, Y_\epsilon]\, dq=\sum_{i,j} \lim_{\epsilon \to 0} \int_{M} (X_\epsilon)_i\,  (Y_\epsilon)_j\,  A_\epsilon[e_i,e_j] \, dq
\]
We have the convergences $(X_\epsilon)_i \2s X_i$ , $(Y_\epsilon)_j 
%\xrightarrow{\smash{2s}}
\S2s Y_j$ and $A_\epsilon[e_i,e_j] \S2s A[e_i^\up, e_j^\up]$.
Then by Lemma~\ref{lem: strong product}, $ A_\epsilon[e_i,e_j]\,( Y_\epsilon)_j \S2s A[e_i^\up, e_j^\up]\,Y_j$ since both are bounded. Therefore by Lemma~\ref{compensated}, that is,  the compensated compactness lemma
for functions, it holds that 
\[ 
\lim_{\epsilon \to 0} \int_{M} (X_\epsilon)_i  (Y_\epsilon)_j \,A_\epsilon[e_i,e_j]\, dq
=
\int_{M} \int_{\TM_p }   \vol(\TM_p)^{-1} X_i Y_j \,A[e_i^{\up},e_j^\up]\, dv\, dp
\]
and the result follows. 
\end{proof}
 As discussed before, our notion of two-scale convergence is invariant under lowering 
and raising indexes, thus the result holds for $(1,1)$ tensors as well. We have stated the result for $(2,0)$-tensors as the notation becomes slightly easier, but there is a corresponding statement for tensors of arbitrary order.

\subsection{Oscillating  tensor fields via  pullback of its coordinate functions}
Next, we send vertical tensors to oscillating tensor fields in the base manifold in the spirit of what  we already did for functions; this is natural, since we can consider functions as $(0,0)$-tensors. There are various possibilities to do this.
The first one is simply to oscillate the coordinate functions with respect to the frame.

\begin{Def} 
\label{def:Tensorcoordinates}
Let $T\in T^{k,m}(\mathcalV)$ and   $(X_I,\omega^J) \in \mathfrak{X}(M)^k \times \Lambda(M)^m$. Then we define the sequence of tensors $\overline{T}^\epsilon \in T^{k,m}(M)$ by 
\[ 
\overline{T}^\epsilon(X_I,\omega^J) =(T (X_I^\up,{\omega^J}^\up)) ^\epsilon 
\]
\end{Def}
In order to obtain an explicit description of  $\overline{T}^\epsilon$ it suffices to  use  Definition~\ref{def:Tensorcoordinates}  in some basis,  canonically the  frame $\{e_i\}_{i=1}^n$ and its dual 1-forms frame $\{\theta_i \}_{i=1}^n$. 
For example, if $X \in \mathfrak{X}(\mathcal V)$ is a smooth vertical  vector field in $\TM$,  $X$ can be written as
\begin{equation} 
\label{eq:vertical_vector_field}
X[q,v]=\sum_{i=1}^n\,f_i[q,v]\,e_i^\up[q,v],
\end{equation}
for some smooth functions $f_i:\TM\to\R$. 
Given that $X(\theta_i^\up)=f_i$
Definition~\ref{def:Tensorcoordinates} yields that
$\overline{X}^\epsilon(\theta_i)=f_i^\epsilon$. Letting $i$ run from $1$ to $n$, we obtain the following explicit expression for  $\bar{X}^\e_q$
,
\begin{equation}\label{equation:alternative test vector fields}
\bar{X}^\e_q:=\sum_j\, \psi_j(q)\sum_{i=1}^n\,f_i[p_j,\exp_{p_j}^{-1}(q)/\e]\,e_i(q).
\end{equation}

Similarly  for $(X_I,\omega^J) \in \mathfrak{X}(\V)^k \times \Lambda(\V)^m$, we will consider $(\bar{X}_I)^\epsilon,(\bar{\omega}^J)^\epsilon \in   \mathfrak{X}(M)^k \times \Lambda(M)^m$ to extend this observation to the general case. 

\begin{lem} For $T \in T^{k,m}(\mathcal V)$, 
$T= \sum T^{I}_{J} e_{I}^\up \otimes {\theta^{J}}^\up$ for
multi-indexes $I={i_1, \ldots i_k}$, $J={j_1, \ldots j_m}$,
we have that
\begin{equation}
\label{eq:oscillation_one}
\bar{T}^\epsilon= \sum \psi_j  \sum T^{I}_{J} [p_j,\exp_{p_j}^{-1}(q)/\e]\, e_I \otimes \theta^J
\end{equation}
\end{lem}

The next lemma will be useful to prove  that the two-scale strong  limit of $\bar{T}^\e_q$ is $T$ as it is natural. 

\begin{lem}\label{upversusepsilon}
Let $f \in C_0^\infty(\mathbb T M)$ and  $X \in \mathfrak{X}(M)$. Then, 
\begin{equation} \lim_{\epsilon \to 0} \|{X^\up}^\epsilon- X\| =0,
\qquad \lim_{\epsilon \to 0}  \| (f X^\up)^\epsilon -f^\epsilon X \|=0
\end{equation}
\end{lem}
\begin{proof}
Notice that  for $X=\sum x^{i}(q) e_i(q)$
\[ (X^\up)^\epsilon=\sum_j\, \psi_j(q)\sum_{i=1}^n x^{i}(p_j)\,e_i(q).\] Hence by the finite overlapping property and Theorem \ref{thm:usable_partition_of_unity},

\[ 
\| (X^{\up})^\epsilon(q) -X(q) \|_{L^\infty}  \le C \sup_{j,i}  |x^{i}(p_j)-x^{i}(q)|  \le C' \epsilon^\beta 
\]
which yields  the first limit.

The second limit follows in the same way, since we have that

\[ (fX^\up)^\epsilon=\sum_j\, \psi_j(q)\sum_{i=1}^n x^{i}(p_j) f [p_j,\exp_{p_j}^{-1}(q)/\e] e_i(q).\] 
and 
\[ f^\epsilon X= \sum_j\, \psi_j(q)\sum_{i=1}^n x^{i}(q) f [p_j,\exp_{p_j}^{-1}(q)/\e] e_i(q) \]
which is identical except for $x^{i}(q)$ taking the place of $x^{i}(p_j)$. 
\end{proof}

All in all, we have the following direct consequence of lemma~\ref{lem:examples}. 

\begin{Prop} 
\label{prop:strongbarT} 
Let $T \in \TVkm$. Then it holds that, 
\[ \overline{T}^\epsilon \S2s T. \]
\end{Prop}

\begin{proof}

The coordinate functions imply the convergences of $(k,m)$-tuples formed from the frame and the dual frame. Since the manifold is parallelizable, every field and form is a linear combination of these
elements, and the result follows. 

\end{proof}

Notice that the above allows us to describe two-scale convergence of tensors, by pairing with test vector fields and forms. That is the statement of the following lemma.

\begin{lem} 
\label{lem:weak_2s_convergence_tensors}
Let   $(X_I,\omega^J) \in \mathfrak{X}(\V)^k \times \Lambda(\V)^m$  be vertical vector fields and forms, $T_\epsilon \in L^2T^{k,m}(M)$ and $T \in L^2T^{k,m}(\V)$.
Then  $  T_\epsilon \2s T $ if and only if  
\begin{equation}\label{2spairing for tensor}
\int_{M} T_\epsilon(({X}_I)^\epsilon, ({{\omega}^J)^\e})\,
 dq \to \oint T(X_I, \omega^J )\, dv\,dp.
\end{equation}
 In particular, if $Y_\epsilon \in \mathfrak{X}(M)$ and $Y \in \Vvectors $, then $Y_\e \2s Y$ if and  only if for every $X \in \mathcal{X}(\mathcal{V})$

\begin{equation}
\lim_{\epsilon \to 0} \int_{M} \langle Y_\epsilon, X^\epsilon \rangle \, dq \to \oint \langle Y,  X \rangle \,dv\, dp.
\end{equation}

\end{lem}

\begin{proof} Let us start by proving that two-scale convergence implies \eqref{2spairing for tensor}. By multilinearity, for each choice of multi-indexes appearing in the definition of $T$  we need to deal with the limit of sums of integrals of terms like
 \[  (T_\epsilon)_{i_1, \ldots i_k}^{j_1, \ldots j_m} (x^{i_1}_{1})^\epsilon \cdots  (x_k^{i_k})^\epsilon \cdot (\omega_{1,j_1})^\epsilon  \cdots  (\omega_{m,j_m})^\epsilon
  \]
  where $x^{i_l}_l$ stands for the coordinates of the vector field $X_l$,  $\omega_{r,j_r}$   are the coordinates of the form $\omega_r$, and    
\[
(T_\epsilon)_{i_1, \ldots i_k}^{j_1, \ldots j_m}=T_\epsilon (e_{i_1}, \ldots, e_{i_k}, \theta_{j_1}, \ldots, \theta_{j_m}) 
\]
are the coordinates of $T_\epsilon$ 
with respect to the frame and the  dual frame. 

Notice that $T_\epsilon \2s T$ implies that
\[ 
(T_\epsilon)_{i_1, \ldots i_k}^{j_1, \ldots j_m}\2s  T (e_{i_1}^\up, \ldots, e_{i_k}^\up, \theta_{j_1}^\up, \ldots, \theta_{j_m}^\up). 
\]
  
By an iterated use of Proposition~\ref{prop:algebra}  we have that 

\begin{equation}\label{eq:equal} \lim_{\e \to 0} \int_{M} ( (x^{i_1}_{1})^\epsilon \cdots  (x_k^{i_k})^\epsilon \cdot (\omega_{1})^\epsilon_{j_1} \cdots  (\omega_{m})^\epsilon_{j_m} -(x^{i_1}_{1} \cdots  x_k^{i_k} \cdot \omega_{1,j_1} \cdots  \omega_{m,j_m} )^\e) \,dq=0 \end{equation}
and notice that the last expression under the integral, $(x^{i_1}_{1} \cdots  x_k^{i_k} \cdot \omega_{1,j_1} \cdots  \omega_{m,j_m} )^\e$ is a test function. 
Therefore 

\[ \begin{aligned} &\lim_{\epsilon \to 0} \int_{M}  (T_\epsilon)_{i_1, \ldots i_k}^{j_1, \ldots j_m} (x^{i_1}_{1})^\epsilon \cdots  (x_k^{i_k})^\epsilon \cdot (\omega_{1,j_1})^\epsilon  \cdots  (\omega_{m,j_m})^\epsilon\, dq=
 \\ & \lim_{\e \to 0} \int_{M} T_\epsilon (e_{i_1}, \ldots, e_{i_k}, \theta_{j_1}, \ldots, \theta_{j_m}) (x^{i_1}_{1} \cdots  x_k^{i_k} \cdot \omega_{1,j_1} \cdots  \omega_{m,j_m} )^\e \,dq=
  \\ &\oint_{\TM}  T (e_{i_1}^\up, \ldots, e_{i_k}^\up, \theta_{j_1}^\up, \ldots, \theta_{j_m}^\up)\,  x^{i_1}_{1} \cdots  x_k^{i_k} \cdot \omega_{1,j_1} \cdots  \omega_{m,j_m}  \, dv\,dp
\end{aligned}
\]
where in the equality of the limits we have used \eqref{eq:equal} and Cauchy-Schwarz inequality,  and in the second equality
that by Definition~\ref{def:2stensors},
$T_\epsilon (e_{i_1}, \ldots, e_{i_k}, \theta_{j_1}, \ldots, \theta_{j_m})$ two-scale converges to the expression 
$T (e_{i_1}^\up, \ldots, e_{i_k}^\up, \theta_{j_1}^\up, \ldots, \theta_{j_m}^\up)$ and also that 
$(x^{i_1}_{1} \cdots  x_k^{i_k} \cdot \omega_{1,j_1} \cdots  \omega_{m,j_m} )^\e$ is a test function. Therefore, as our $(k,m)$-tuple of vectors and forms was arbitrary, we obtain \eqref{2spairing for tensor}.

Assume now that \eqref{2spairing for tensor} holds, and let us obtain two-scale convergence. Let  $X_1, X_{I_1}, \omega_J \in \mathfrak{X}(M) \times \mathfrak{X}(M)^{k-1} \times \Lambda M^m$ and $f \in C^\infty(\TM)$. 
Then by multilinearity of the tensor, 
\[ \lim_{\epsilon \to 0} \int_M T_\epsilon(X_1, X_{I_1},\omega_J) f^\epsilon dq= \lim_{\epsilon \to 0} \int_M T_\epsilon(f^\epsilon X_1, X_{I_1}, \omega_J)  dq \]
By  Lemma~\ref{upversusepsilon}, this equals to 

\[
\lim_{\epsilon \to 0} \int_M T_\epsilon(f X_1^\up)^\epsilon, (X_{I_1}^\up)^\epsilon, (\omega_J)^\epsilon) \, dq=\oint_{\TM}  T( f X_1^\up, X_{I_1}^\up, \omega^\up)\, dv \,dp=\oint_{\TM}  fT( X_1^\up, X_{I_1}^\up, \omega^\up)\, dv\, dp
  \]
which yields the two-scale convergence of the functions $T_\e(X_1,X_{I_1},\omega_J)$  and thus the  two-scale convergence of $T_\epsilon$.

\end{proof}

\subsection{Oscillating tensor fields via pullbacks of vertical tensors}
Here we will consider a second form to obtain oscillating tensor fields in the base manifold $M$ from vertical tensor fields. Although our new definition will be somewhat different from that appearing in Definition \ref{def:Tensorcoordinates}, we will show that the difference will disappear with higher oscillations.

Recall that 

\begin{equation}\label{eq:pullback} f^\epsilon= \sum \psi_j(q) H_{ \epsilon,j}^* \tilde{f}(p_j,\cdot)(q), \quad q \in M. \end{equation}
 In order to extend this view  to tensors, we start by elaborating on the definition of $H_{\e,j}$. 
  
We consider the map
\[
H_{\e,j}:D_j\to \TM, 
\quad
H_{\e,j}(q)=[p_j,\dfrac{1}{\e}\exp^{-1}_{p_j}(q)]
\]
defined in the Voronoi set $D_j$ corresponding to $p_j$ constructed in Section \ref{sec:partition}. If $\psi_j$ is the member of the partition of unity related
to the Voronoi set $D_j$,

\[H_{\e,j}(q)=[p_j,  
\psi_j(q)\cdot \dfrac{1}{\e}
\exp^{-1}_{p_j}(q)]\]
yields an extension 
\[
H_{\e,j}:M\to\TM,
\]
that we do not relabel. This extended $H_{\e,j}$ maps everything outside of $D_j$ to $[p_j,0]$.  More important, observe that the image of $H_{\e,j}$ is entirely contained in the vertical torus $\TM_{p_j}$. 
Next, we specify how $H_{\e,j}$ acts on tensors.  We denote $[p_j,v]=H_{\e,j}(q)$.

\begin{Def} 
Let $T \in T^{k,m}(\mathcal{V})$ a vertical tensor. Then  for $q \in D_j$, we define a $(k,m)$-tensor with support in $D_j$ as
\begin{equation}
\label{eq:oscillating_tensors_2}
(H_{\e,j})_*(T[p_j,v])(X_I,\omega^J)(q)=(T\circ H_{\e,j}(q))(  \epsilon  (H_{\e,j})_*(X_I),  \epsilon^{-1}   (H_{\e,j}^{-1})^*(\omega^J)) 
\end{equation} 
\end{Def}

Recall that, if in frame coordinates, we have that  
\[
X=\sum x^{i} e_i, 
\qquad 
\omega=\sum \omega_i (\exp_{p_j}^{-1})^*\theta^{i}, 
\quad
\text{ and } 
v=H_{\e,j}(q)\]
then a simple computation yields
\begin{equation}\label{eq:pushforwards}  ( (H_{\e,j})_* X)_v =\epsilon^{-1}\sum  x^{i}(q) d\exp_{p_j}^{-1} e_i, \qquad ( (H_{\e,j}^{-1})^*\omega)_v=\epsilon \sum \omega_i (q) \exp_{p_j}^{*} \theta^{i}. 
\end{equation}

This leads to the following alternative definition of oscillating tensors:

\begin{Def}\label{def:osctensorpull}  Let $T$ be a smooth vertical tensor field in $T^{k,m}(\mathcal{V})$.  For $\e>0$, and for any $\e$-net $\{p_j\}_j$ as in Theorem \ref{thm:big_Voronoi}, 
we define $T^\epsilon \in T^{k,m}(M)$ by the sum
\begin{equation}\label{eq:defTepsilon}
T^\e_q= \epsilon^{k-m} \sum_j\, \psi_j(q) \,(H_{\e,j})_*(T[p_j,v])
\end{equation}
\end{Def}

Let us look at equation \eqref{eq:defTepsilon} in coordinates. We start with vector fields. 
Since by definition $(X,(H_{\e,j}^{-1})^* \theta)=(\theta, (H_{\e,j}^{-1})_* X) $ we obtain that

\[ X^\epsilon=\sum \psi_j (H_{\e,j}^{-1})_* X   \]

In order to express in coordinates, we carefully explain how $d \exp_{p_j}$ acts on the lifted frame. 
Recall, that we are considering parallelizable  manifolds, with the corresponding  frame  $\{e_i\}_1^n$ forming a basis of the tangent space at every point of $M$. Denote by $\{e_i^\up\}$ their vertical  lifts to $TM$ as indicated in Definition \ref{def:vector_lift}. For $\e>0$, and for any $\e$-net $\{p_j\}_j$ as in Theorem \ref{thm:big_Voronoi}, define for $q=\exp_{p_j}(v)$
\begin{equation}
\label{def:downvector}
e_i^{\downarrow}(p_j;q):= (d\exp_{p_j})_{v}(e_i(p_j)).
\end{equation}
 in the Voronoi set  $D_j$, and zero otherwise. 

Notice that here we are somewhat abusing notation, identifying the tangent space to $T_{p_j}M$ at $\exp^{-1}(q)$ with $T_{p_j}M$. To be more rigorous,
as $\exp:TM\to M$, we should have written 
\begin{equation}
e_i^{\downarrow}(p_j;q):= (d\exp_{p_j})_{v}(e_i^\up(p_j,v))=\e^{-1}(H_{\e,j}^{-1})_*(e_i^\up(p_j,v)),
\end{equation}
since $(d\exp_{p_j})_{v}:T_v(T_{p_j}M)\to T_qM$, and $e_i(p_j)\in T_{p_j}M$, while $e_i^\up(p_j,v)\in T_vT_{p_j}M$. 

Observe that, if $\{\psi_j\}_j$ is the partition of unity associated to $D_j$,   the fields 
$\psi_j\,e_i^\downarrow[p_j;\cdot]$ are   smooth vector fields well-defined everywhere in $M$, since $\supp\psi_j\subset D_j$. Writing $$X=\sum_if_i[p,v]e_i^\uparrow,$$
we get that 
\begin{equation}
\label{eq:epsilon_vector field}
X^\e_q:=\sum_j\, \psi_j(q)\sum_{i=1}^n\,f_i[p_j,\exp_{p_j}^{-1}(q)/\e]\,e_i^\downarrow(p_j;q)\end{equation}

Similarly, for elements of the dual frame $\theta^{i}$, we define 
\begin{equation}\label{eq:downform
}
\theta_i^{\downarrow}(p_j;q):= (\e (H_{\e,j})^*(\theta_i^\up(p_j,v)),
\end{equation}
and notice that in coordinates, abusing the notation in a self-explanatory way for a second, 
\begin{equation}
\label{eq:downform}
\theta_i^{\downarrow}(p_j;q):= (d\exp^{-1}_{p_j})_{v}(\theta_i(p_j)).
\end{equation}
In any case we obtain, the expression
\begin{equation}
\label{eq:epsilon_oneform}
\omega^\e_q:=\sum_j\, \psi_j(q)\sum_{i=1}^n\,\omega_i[p_j,\exp_{p_j}^{-1}(q)/\e]\,\theta_i^\downarrow(p_j;q)\end{equation}

Once we understand vectors and forms, the result for general tensors follows easily. To simplify the writing, we are following Einstein's summation convention of adding over repeated indexes.

\begin{lem} 
Let $T \in T^{n,m}(\mathcal V)$,  with $T= T^{i_1, \ldots i_n}_{j_1, \ldots j_m} e_{i_1} \otimes  \ldots \otimes e_{i_n} \otimes \theta^{j_1} \otimes \ldots \otimes \theta^{j_n}$. Then, it holds that, 

\[T^\e=(T^{i_1,\ldots, i_n}_{j_1,\ldots, j_m} )^\e e_{i_1}^\downarrow  \otimes  \ldots \otimes e_{i_n}^\downarrow  \otimes (\theta^{j_1})^\downarrow  \otimes \ldots  \otimes  (\theta^{j_n})^\downarrow\]

\end{lem}

Notice that, comparing with equation \eqref{eq:oscillation_one}, we are shifting the frame. However, the next lemma shows that for $\e$ small enough, both basis are very close. 

\begin{lem} \label{lem:down_vectors_variation}
 Let  $e_i,\theta_i$ be members of the frame and its dual 1-forms, and consider  $e_i^\downarrow[p_j;q] \in \mathfrak{X}(M)$ be defined by \eqref{def:downvector} and $\theta_i^\downarrow[p_j;q] \in \Lambda(M)$ be defined by \eqref{eq:downform} We have that 
\[
|e_i^\downarrow[p_j;q]-e_i(q)\,|+ |\theta_i^\downarrow[p_j;q]-\theta_i(q)\,|   \leq 
C\e^\beta
\]
for any $q\in D_j$.
\end{lem}
\begin{proof}
It follows immediately from Lemma \ref{lem:almost Euclidean charts} since normal coordinates $\phi$ are written in terms  of the exponential map centered at $p_j$. 
\end{proof}

\begin{lem} Let $T \in T^{n,m}(\mathcal V)$. Then 
\[ \lim_{\epsilon \to \infty} \int_{M} |\bar{T}^\epsilon-T^\epsilon| \,dq=0 \]
\end{lem}

Therefore, we get the  following corollary of  Proposition \ref{prop:strongbarT} and Lemma \ref{lem:weak_2s_convergence_tensors}.

\begin{lem}\label{lem:convergenceofT}
Let $T \in \TVkm$. Then 

\begin{itemize}
\item [i)] For $T^\e$ as in Definition \ref{def:osctensorpull}, we have that  $T^\epsilon \S2s T$.
\item [ii)] Furthermore, 
$T_\epsilon \2s T$ if and only if
\begin{equation}
\int_{M} T_\epsilon(X_1^\epsilon, \ldots, X_k^\epsilon, \omega_1^\epsilon, \ldots \omega^\epsilon_m) \,dp \to \oint T(X_1, \ldots, X_k, \omega_1,\ldots,  \omega_m ) \,dv \, dq
\end{equation}
 In particular, if $Y_\epsilon \in \mathfrak{X}(M)$ and  $Y \in \mathfrak{X}(\mathcal{V} )$, $Y_\e \2s Y$ if and  only if for every $X \in \mathcal{X}(\mathcal{V})$

\begin{equation}
\lim_{\epsilon \to 0} \int_{M} \langle Y_\epsilon, X^\epsilon\rangle \,dq \to \oint \langle Y, \bar X \rangle \,dv \, dp\
\end{equation}
\end{itemize}
\end{lem}

The following Lemma extends Proposition~\ref{prop:L2_bounded_2_scale}
 and Lemma~\ref{lem:extra} from functions to arbitrary tensors.

\begin{Lemma} \label{lem:compactnesstensors}
Let $T_\e \in T^{k,m}( M)$ be an $L^2$--bounded sequence.
\begin{enumerate}
\item Up to
a subsequence, $T_\e$ two-scale converges to a vertical tensor  $T_0 \in L^2( T^{k,m}(\mathcalV))$.
\item Up to a subsequence  $T_\epsilon$ converges weakly in $L^2$ to the  average on the fiber $\tilde{T}$.
\end{enumerate}
\end{Lemma}
\begin{proof} 

The statement of the Lemma is a direct corollary of  Proposition~\ref{prop:L2_bounded_2_scale}
 and Lemma~\ref{lem:extra}.
%\footnote{There is an equivalent proof using pairings with vector fields}
\end{proof}

\subsection{Gradients and differential of oscillating  functions}

We have  the following result relating the sequence of gradients of functions $f^\e$ and the $\e$--sequence obtained from the vertical differential of a function $f:\TM\to\R$ which will be important in relating two-scale limits of functions and their gradients.

\begin{Prop}
\label{prop:gradient_epsilon_and_epsilon_gradient}
Let $f:\TM\to \R$ a differentiable function, 
%.
Then
\begin{equation}
\lim_{\e\to 0}  \,\|\e\grad_q (f^\e) - (\grad_v f)^\e\|_{L^\infty(M)}   =0.
\end{equation}

\end{Prop} 

We separate the proof into two parts.  The first lemma proves a similar statement for differentials, where the commutation between pullbacks and exterior derivatives makes the proof specially simple. The second takes into account the metric, and says that, in the limit, the musical isomorphisms
and our way of pulling back vector fields and forms commute. 

\begin{Lemma}
\label{lem:1-form-version-of-gradientepsilon}
\label{gradients}
Let $f:\TM\to \R$ be smooth function. 
Then 
\begin{equation}
\| |\e\cdot d_q (f^\e) - (d^v f)^\e \|_{L^\infty(M)}  =O(\e^{1-\alpha}) ,
\end{equation}
as $\e\to 0$.
\end{Lemma}

\begin{proof}

Apply the exterior differential to  formula \eqref{eq:pullback} to obtain that

 \[
 df^\epsilon_q = 
 \sum \psi_j(q) dH_{ \epsilon,j}^* f[p_j,\cdot](q)+ d\psi_j\cdot  H_{ \epsilon,j}^* f[p_j,\cdot](q),\qquad q \in M. 
 \]
Notice that since $dH_{ \epsilon,j}^* f[p_j,\cdot](q)=H_{ \epsilon,j}^* df[p_j,\cdot](q)$,  
it follows from Definition~\ref{def:osctensorpull} that 

\[ \epsilon \sum \psi_j(q) dH_{ \epsilon,j}^* f[p_j,\cdot](q)= (d_v f)^\epsilon \]
Therefore,
\[
d(f^\e)_q-\frac{1}{\e}(d^vf)^\e_q=\sum_j\,f[p_j,\exp_{p_j}^{-1}(q)/\e]  d\psi_j.
\]
Multiplying both sides by $\e$,  using \eqref{eq:size of iterated gradient}  and the finite overlapping of the partition of unity, prove the Lemma.

\end{proof}

For the second part, recall that in a Riemannian manifold $(M,g)$, we can define the musical isomorphisms
\[
X\in\chi(M)\to X^\flat\in\Lambda^1(M),
\qquad
X^\flat(Y):=g(X,Y),
\]
and
\[
\theta\in\Lambda^1(M)\to\theta^\sharp\in\chi(M),
\qquad
g(\theta^\sharp, Y):=\theta(Y)
\]
to raise and lower indexes.  In particular, recall that for a smooth function $f$ in a Riemannian manifold,
$ (df)^\sharp=\grad f$.  Because exterior derivative commutes with pullbacks, it is somewhat easier to deal with 
differentials and then lower indexes. The following lemma states that this is a legitimate operation for oscillating sequences.

\begin{lem}\label{lem: musicalconmut}
Let $\omega\in\Lambda^1(\V)$. Then
\[ \lim_{\e \to 0 }\| (\omega^\e)^\sharp -(\omega^\sharp)^\e\|_{L^\infty(M)} =0,\]

\end{lem}
\begin{proof}
It suffices to prove that
\[\langle(\omega^\e)^\sharp, e_{\ell}\rangle  -\langle (\omega^\sharp)^\e, e_{\ell}\rangle  \to  0 \]
for any vector $e_\ell$ in the frame. 
If $\omega=\sum_k \omega_k[p,v]{\theta^k}^\uparrow$ is   the form in coordinates with respect to the   dual frame ${\theta ^k}^\uparrow$ given by the parallelization,
we have that
\[ 
(\omega^\sharp)^\e =\sum_j\,\psi_j(p)\sum_{i,k} \omega_k[p_j,\frac{ \exp_{p_j}^{-1}(p)}{\e}]\, g^{ki}(p_j)\, e_i
\]

Then, on one hand  we have
\begin{equation}
    \label{eq:omega_sharp_epsilon}
\langle (\omega^\sharp)^\e, e_{\ell}\rangle=\sum_j\psi_j(p)\sum_{i,k}\omega_k[p_j,\frac{ \exp_{p_j}^{-1}(p)}\e]g^{ki}(p_j)g_{i\ell}(p)
%=\sum_j\psi_j(p) \omega_{\ell}[p_j,\frac{ \exp_{p_j}^{-1}(p)}{\e}]
.
\end{equation}

Observe that, when $p\in\supp\psi_j$, $|g_{i \ell}(p)-g_{i\ell}(p_j)|\leq C\cdot d(p,p_j)\leq C\e^\beta$, thus we can replace \eqref{eq:omega_sharp_epsilon} by 
\[
\sum_j\psi_j(p)\sum_{i,k}\omega_k[p_j,\frac{ \exp_{p_j}^{-1}(p)}\e]g^{ki}(p_j)g_{i\ell}(p_j)=
\sum_j\psi_j(p)\cdot \omega_{\ell}[p_j,\frac{ \exp_{p_j}^{-1}(p)}{\e}]
\]
as $\e\to 0$. 

On the other hand,
\[
\langle(\omega^\e)^\sharp, e_l\rangle =\omega^\e(e_l)=\e \sum_j\psi_j(p)H^*_{\e,j}(\omega)(e_l)\]
\[=\e \sum_j\psi_j(p) \sum_k \omega_k[p_j, \frac{ \exp_{p_j}^{-1}(p)}{\e}]({\theta^k}^\uparrow)((dH_{\e,j})_p(e_l)).\]

Thus, it suffices to show that 
\begin{equation}
\label{eq:limit}
 \sum_j\psi_j(p)\cdot \left( \omega_{\ell}[p_j,\frac{ \exp_{p_j}^{-1}(p)}{\e}]- \e\sum_k \omega_k(p_j, \frac{ \exp_{p_j}^{-1}(p)} \e)({\theta^k}^\uparrow)((dH_{\e,j})_p(e_l))  \right) \to 0
\end{equation}
as $\e\to 0$.
But thanks to Lemma \ref{lem:almost Euclidean charts}, we have that 
\[
(dH_{\e,j})_p(e_\ell)=\frac 1\e d_p\exp_{p_j}^{-1}(e_\ell)= 1/\e(e_\ell^\uparrow+ O(d(p,p_j))
\]
for $p\in\supp \psi_j$, and therefore
\[
({\theta^k}^\uparrow)((dH_{\e,j})_p(e_l))=
\delta^k_\ell+ O(d(p,p_j)).
\]
Replacing in \eqref{eq:limit}, the desired limit follows.

\end{proof}

Now, we are ready to transfer the information from Lemma~\ref{lem:1-form-version-of-gradientepsilon} to gradients. 

\begin{proof}[Proof of Proposition~\ref{prop:gradient_epsilon_and_epsilon_gradient}]
We have 
\[ \lim_{\e\to 0}\| \e \grad (f^\e)_q - (\grad^v f)^\e\|
=
\lim_{\e\to 0}\| \e (d_qf^\e )^\sharp- ((d^v f)^\sharp )^\e\|
\]
Then from lemma \ref{lem: musicalconmut}:
\[=\lim_{\e\to 0}\| \e (d_qf^\e)^\sharp- ((d^v f)^\e )^\sharp\|=\lim_{\e\to 0}\|  (\e\cdot d_q (f^\e) - (d^v f)^\e )\| =0
\]
where we used Lemma \ref{lem:1-form-version-of-gradientepsilon}.

\end{proof}

\section{Integration by parts} \label{sec:part}

\subsection{Integration by parts: Prerequisites}
\label{sec:by_parts_technical}
In this section, we collect, for the reader's convenience, a set of results that are necessary for the proof of Theorem \ref{vector_limit}. They can be skipped in a first reading, and return to them when going through Section \ref{sec:by_parts_main}. 

\subsection{A technical lemma}
Let $\tilde{f}:TM\to\R$ be a smooth function, and for $p\in M$ consider an open neighborhood $U$ of $p$ where $\exp_p^{-1}$ exists. Define the map
\[
f:U\to\R, \qquad
f(q):=\tilde{f}(p,\exp_p^{-1}(q)/\e),
\]  
where $\e>0$. Recall from definitions \ref{def:vector_lift} and  \eqref{def:downvector} that given a tangent vector $e\in T_pM$, we have  fields $e^\uparrow{(p,v)}\in \V_{p,v}$ and $e^\downarrow(p;q)$ in $T_pM$ and a neighborhood of $p$ respectively. 

\begin{lem}
\label{lem:df_divided_bi_epsilon}
With the above notation, we have 
\[
df_q(e^\downarrow(p;q))=
\frac{1}{\e}\,e^\uparrow_{(p,(\exp_p^{-1}q)/\e)}(\tilde{f}).
\]  
\end{lem}
\begin{proof}
Let $v=\exp_p^{-1}(q)$;
recall that $e^\downarrow(p;q)=(d\exp_p)_v(e^\uparrow(p,v))$ and that 
$e^\uparrow(p,v)$ is tangent to $t\to (p,v+te)$ at $t=0$. Thus,
\[
df_q(e^\downarrow(p;q))=
df_q(d\exp_p)_v(e^\uparrow(p,v))=
d(f\circ\exp_p)_v(e^\uparrow),
\]
and this agrees at $t=0$ with
\[
\frac{d}{dt}f\circ\exp_p(v+te)=
\frac{d}{dt}\tilde{f}(p,(v+te)/\e)=
\frac{1}{\e}e^\uparrow_{(p,v/\e)}(\tilde{f})
\] 
as we wanted to show.
\end{proof}

\subsection{Divergence of down vectors}
With our notation of nets as in Theorem \ref{thm:usable_partition_of_unity}, given a vector $e\in T_{p_j}M$, construct, as in equation \eqref{def:downvector}, the vector fields $e^\downarrow:= e^\downarrow(p_j;q)$ for $q$ in the Voronoi set $D_j$.
\begin{lem}
\label{lem:divergence_down}
With the above notation, we have that 
\[
|(\div e^\downarrow)(q)|\leq C\dist(p_j,q)
\]
\end{lem}
\begin{proof}
Recall that $e^\downarrow(p_j; q)=d\exp_{p_j}(e^{\uparrow}(p_j,v))$, thus in the normal coordinate chart $(U,\exp_{p_j})$ the coordinates of $e^\downarrow(p_j; q)$ are those of $e^{\uparrow}(p_j,v)$, that is a constant vector field in $T_{p_j}M$.  

Shorten the notation by denoting $V_q=e^\downarrow(p_j; q)$, thus, its $i$-th component $V^i(q)$ is constant for every $i$.
Then, using the formula for the divergence of a vector field in coordinates, we get
\[
\div V (q)= \frac{1}{\sqrt{g}}\partial_i(\sqrt{g}\, V^i)(v) 
=
\frac{1}{\sqrt{g}}V^i\partial_i(\sqrt{g})(v)
\]
where $v=\exp_{p_j}^{-1}(q)$, and $\|v\|=d(p_j,q)$. 
Since $\partial_i(\sqrt{g})(0)=0$, 
\[
|\partial_i(\sqrt{g})(v)|\leq C\|v\|=Cd(p_j,q)
\]
and the inequality follows.
\end{proof}

\subsection{{The \texorpdfstring{$P_k$}{} operator}}

Let $k\in\mathbb{Z}^n$ be a multi-index. 
Recall from Lemma \ref{lem:P_kf} the definition of the $P_k$ operator in functions $\tilde{u}$ of $\R^n$,
\[
P_k\tilde{u}:=\dfrac{\ip{k}{\nabla \tilde{u}}}{2\pi i|k|^2}.
\]
 We will often use the following  basic integration by parts formula in $\mathbb{R}^n$, that we label for future use: given $\tilde{u}:\mathbb{R}^n\to \mathbb{R}$ with compact support, it holds that
\begin{equation}
\label{parts}
\int\, \tilde{u}(v)\cdot \exp{\left(\frac{2 \pi i}{\e}\ip{v}{k}\right)}\,dv = -\epsilon\int P_k \tilde{u}(v)\cdot \exp{\left(\frac{2 \pi i}{\e}\ip{v}{k}\right)\, dv}.
\end{equation}

The following estimates are trivial 
for our operator $P_k$. Let $\psi_j:M\to\mathbb{R}$, a member of the partition of unity from Theorem~\ref{thm:usable_partition_of_unity} and define $\widechi_j:T_{p_j}M\to\mathbb{R}$ as $\widechi:=\psi_j\circ\exp_{p_j}$.
Therefore, in normal coordinates at $p_j$, the pairing $(d\psi_j,e(p_j;\cdot))$ has the local expression 
 $(d\widechi_j, e^\uparrow)$,
Similarly, for any function $u:M\to\R$, the $n$-form $u\cdot d\vol_M$  lifts to a form $\widetilde{u}(v)\, g(v)^{1/2} dv$, where $g(v)= {\det (g_{ij})}_{q}$ with $q= \exp_{p_j}(v)$, and $\widetilde{u}=u\circ\exp_{p_j}$.
 
 Then, we claim that

\begin{equation}\label{p1}
P_k^m((d\widechi_j, e^\uparrow))\leq C_1\frac{\sup_{|\gamma|=m+1}|D^{\gamma}\psi_j|}{|k|^m}\leq C'_1\frac{\e^{-\alpha(m+1)}}{|k|^m}
\end{equation}
\\
and 
\begin{equation}\label{p2}
P_k(\widetilde{u} g^{1/2})\leq C_2\frac{|\nabla \widetilde{u}|+|\widetilde{u}|}{|k|}, 
\end{equation}
for some constants $C_1$, $C_2>0$.
 
 To see this, observe first that,  in any basis of $T_{p_j}M$, $e^\uparrow$ has constant coordinates. 
 Thus, the function $(d\widechi_j, e^\uparrow)$ is just a linear combination of the derivatives of $\widechi$ with constant coefficients, when taking its gradient, we will get another linear combination of terms in $D^2\widechi$, etc. Hence, we get the first inequality in \eqref{p1}; the second follows from inequality \eqref{eq:size of iterated gradient}.
 Finally, in order to prove \eqref{p2}, use \eqref{eq:metric_at_zero} combined with the product rule for derivatives.

Additionally, we observe that, when $u:M\to\R$ is a function in $H^1(M)$,  and choose some $p_j$ in the net, and use the exponential normal coordinates to write $\tilde{u}=u\circ \exp_{p_j}$, then the gradient of $u$ is written locally as
\[
\nabla\tilde{u}(v)=\sum_i\sum_j \,g^{ij}(v)\frac{\partial \tilde{u}}{\partial v_j}\frac{\partial }{\partial v_i} 
\]   
Therefore, use once again,  \eqref{eq:metric_at_zero}, we get the following estimate: 
\begin{equation}
\label{eq:local-global}
\int_{\{\nabla\widetilde{\psi_j}\neq 0\}}\, |\tilde{u}|+|\nabla\tilde{u}|\, dv
\leq
C \int_{\{\nabla \psi_j\neq 0\}}\, |u|+|\nabla u|\, dp
\end{equation}
for some universal constant $C$.

\subsection{Integration by parts for test vector fields}
\label{sec:by_parts_main}
The usual formula of integration by parts in a Riemannian manifold without boundary states that
\begin{equation}
\label{eq:int_by_parts}
\int_M\, X(u) + u\div X \, dq =0,
\end{equation}
since $\int_M\,\div (uX)\,dq=0$. 
The key difference between our situation and this case is that integrating by parts will not be that obvious in the presence of two-scale convergence. 

Let $X$ be a  smooth vertical vector field in $\TM$, written as
\[
X_{[p,v]}=\sum_{i=1}^n {f}_i[p,v] e_i^\uparrow,
\]
and  recall the formulas in equation \eqref{eq:epsilon_vector field} and Definition \ref{def:tensor_average}, 
\begin{equation}
\label{eq:test_vector_field}
X^\e_q:=\sum_j\, \psi_j(q)\sum_{i=1}^n\,f_i[p_j,\exp_{p_j}^{-1}(q)/\e]\,e_i^\downarrow(p_j;q).
\end{equation}

\begin{equation}\label{eq:average_vector_field}
\widetilde{X}_q:=
\vol(\T M_q)^{-1}\int_{\T M_q} {X}[q,v] dv = \sum_{i=1}^n\,\left(\vol(\T M_q)^{-1}\int_{\TM_q}\, f_{i}[q,v]\,  dv\,\right)\cdot e_i  \in \mathfrak{X}(M).
\end{equation}

 We show that for our test functions and vector fields, integration by parts works up to an error tending to zero with $\e$. 
In order to state the result, we will shorten notation denoting by   $h[p,v]=  (\div^{v}X)[p,v]$.

\begin{Theorem} 
\label{vector_limit}
Let $u \in H^1(M)$. Then, as $\e \to 0$,
\begin{equation} \label{12.2.17}
\left |
-\int_M \langle \grad u,X^{\e} \rangle \, dq
- 
\frac{1}{\epsilon}\int_{M} u\cdot  h^{\epsilon}\,dq
-\int_M u\cdot\div\widetilde{X} \, dq \right | 
\to 0,
\end{equation} 

Moreover, this convergence is uniform when $u$ lies in a bounded set in $H^1(M)$.
 \end{Theorem}

This has the immediate consequence

\begin{Cor}
\label{cor:H1_orthogonal_to_divergence}
Let $u\in H^1(M)$ and $X$ a smooth vertical vector field in $\TM$. Then, 
\[
\lim_{\e\to 0} \int_M\,u\cdot (\div^v X)^\e\,dq = 0.
\]
Moreover, this convergence is uniform over bounded sets in $H^1(M)$. In other words

\[ \lim_{\epsilon \to 0} \| (\div^v X)^\e\|_{H^{-1}(M)}=0 \]
\end{Cor}
\begin{proof}[Proof of the Corollary]
Multiply equation \eqref{12.2.17} by $\e$, and let it tend to zero. 
\end{proof}

The proof of Theorem \ref{vector_limit} requires not only the tools appearing in the previous sections, but also new ones. We will give the proof later, but to motivate what follows, observe that, when applying equation \eqref{eq:int_by_parts} to a test vector field $X^\e$ and the function $u$, we obtain
\begin{equation} 
 -\int_M\, (du, X^\e) \, dq = \int_M\, u\div X^\e \, dq.
\end{equation}
Recall that $(du, X^\e)= \langle \grad u, X^\e \rangle$
pointwise. 

%\subsection{The divergence of a test vector field}
%\label{subsec:diver_test}
From the above, we will need to compute the divergence of a test vector field as in \eqref{eq:test_vector_field}. For the reader's convenience, we consider the simpler case when $X_{[p,v]}=f[q,v]e^\uparrow$ for some function $f:\TM\to\R$, and some vector field in the frame, giving the parallelization of $M$. 
Then equation \eqref{eq:test_vector_field} reads
\[
X^\e_q:=\sum_j\, \psi_j(q)f[p_j,\exp_{p_j}^{-1}(q)/\e]\,e^\downarrow(p_j;q),
\]

Thus, writing $\bar{f_j}(q)=f[p_j,\exp_{p_j}^{-1}(q)/\e]$, we get
 \begin{multline}
 \label{by_parts}
-\int_M\, (du, X^\e) \, dq =
\sum_j \int_{M}   u(q) 
\bar{f_j}(q)
 ({d_q\psi_j(q)},{e^{\downarrow}(p_j;q)}) dq 
\\
+\sum_j \int_M u(q) \psi_j(q)  
({d_q f(p_j, \exp^{-1}_{p_j}(q)/\epsilon)},{e^{\downarrow}(p_j;q)})   dq
\\
+  
\sum_j\int_M u(q) \psi_j(q) \bar{f_j}(q) 
\div   (e^{\downarrow}(p_j; q))\, dq.      
\end{multline}

We will treat each of the terms appearing in the right of \eqref{by_parts} separately.

\subsection{Oscillation of the coefficients}
For each $p_j$ in the net, we consider normal coordinates at $p_j$, 
\[
\exp_{p_j}:U_j\subset T_{p_j}M\to M, \qquad v\to \exp_{p_j}(v).
\]   
In these coordinates, the function $\bar{f}_j$ is written as $v\to f[p_j,v/\e]$. For each $p$ in $M$, we take the $\ZZ^n$-periodic function induced by $f[p,v]$ in $TM$, and its Fourier expansion 
\begin{equation}
\label{eq:fourier}
f(p, v)=\hat f(p, 0)+  \sum_{|k|\geq 1} \hat f(p, k)\, e^{2 \pi i \ip{v}{k}}.
\end{equation}

In these coordinates, the first term in equation \eqref{by_parts} reads as
\begin{multline}
\sum_j \int_{M}   u(q) 
f(p_j, \exp^{-1}_{p_j}(q)/\epsilon)
({d_q\psi_j(q)},{e^{\downarrow}(p_j;q)}) dq = \\
\sum_j \int_{T_{p_j}M}\, \bar{u}(v)\,f(p_j,v/\e)\,e^\uparrow(\widechi_j)\,g^{1/2}(v)\,dv,
\end{multline}
where $\bar{u}=u\circ\exp_{p_j}$, $\widechi_j=\psi_j\circ\exp_{p_j}$ and $g(v)=\det(g_{ij})$, where $(g_{ij})$ is just the metric tensor in these coordinates.

Replacing $f$ by its Fourier expansion as in \eqref{eq:fourier}, leaves two types of terms, depending on whether $|k|\geq 1$ or $k=0$. 
We will need to be a bit more delicate to discard non-constant Fourier modes as we have to deal with 
derivatives of the cut-off function, which yields powers of $\epsilon^{-\alpha}$. We start by proving the following key lemma.

\begin{Lemma}
\label{hfma}
Let $k$ be a nonzero multi-index in $\Z^n$. 
For any function $u$ in $H^{1}(M)$, it holds that
\begin{equation}
\left|\,  \int_{T_{p_j}M}\, \bar{u}(v)\,\left(\exp{\frac{2\pi i}{\e}\ip{k}{v}}\right)
\,e^\uparrow(\widechi_j)\,g^{1/2}(v)\,dv\,\right|  \le C |k|^{-1} \int_{\nabla \psi_j \neq 0}
|u|+|\nabla u| dp
\end{equation}
\end{Lemma}
\begin{proof}
The proof of the lemma follows by integrating parts sufficiently many times and using
  that \eqref{p1} provides  the bound 
 \begin{equation} \label{eq:rec} |P_k^m(e^\uparrow(\widechi_j))| \le  C  \frac{\e^{-\alpha(m+1)}}{|k|^m};
 \end{equation}

As a matter of
fact, we claim that for $m\ge 1$ it holds that
 \begin{multline}
  \label{eq:inductionstep}
\left|\,\int_{T_{p_j}M}\, \bar{u}(v)\,\left(\exp{\frac{2\pi i}{\e}\ip{k}{v}}\right)
\,e^\uparrow(\widechi_j)\,g^{1/2}(v)\,dv\,\right| \le\\  
\sum_{n=1}^m \frac{\epsilon^{n(1-\alpha)}}{|k|^n} 
\left( \int_{\nabla \psi_j \neq 0}
|u|+|\nabla u|\, dp \right) 
+\epsilon^m \left|\int_{\R^n}
\bar{u}(v)\,\exp{\left(\frac{2\pi i}{\e}\ip{k}{v}\right)}
\,P_k^m(e^\uparrow(\widechi_j))\,g^{1/2}(v)\,dv\right|, 
\end{multline}
  Due to 
\eqref{eq:rec}, the last term is bounded by $C \epsilon^{m-(m+1)\alpha}$, that tends to zero as $m$ goes to infinity. 
%%because lim m-(m+1)\alpha=\infinito and 
Observe also that for every $m$, we have a bound $\sum_{n=1}^m \frac{\epsilon^{n(1-\alpha)}}{|k|^n} \le C |k|^{-1}$ for some constant $C$. Thus, if \eqref{eq:inductionstep} holds, we have proven  Lemma~\ref{hfma}.

In order to prove inequality \eqref{eq:inductionstep}, we will argue by induction. 
The case $m=1$ follows from a direct integration by parts 
 for the integral
\[
 \int_{T_{p_j}M}\, \bar{u}(v)\,\exp\left({\frac{2\pi i}{\e}\ip{k}{v}}\right)
\,e^\uparrow(\widechi_j)\,g^{1/2}(v)\,dv\,,
\]
which yields 
\begin{multline}
\label{eq:sum_of_integrals}
\e\int_{T_{p_j}M}\, \,\exp\left({\frac{2\pi i}{\e}\ip{k}{v}}\right)
\,e^\uparrow(\widechi_j)\,P_k(\bar{u}g^{1/2})(v)\,dv +
\\
+
\e\int_{T_{p_j}M}\, \bar{u}(v)\,\exp\left({\frac{2\pi i}{\e}\ip{k}{v}}\right)
\,P_k(e^\uparrow(\widechi_j))\,g^{1/2}(v)\,dv
\,
\end{multline}

It is therefore enough to bound the first term by 
\[
\frac{C}{|k|}\e^{1-\alpha}\int_{\{\nabla\psi_j\neq 0\}}\,(| u|+|\nabla u|)\,dp.
\]
This follows by a use of inequality \eqref{p2} for 
 the term $ P_k(\bar{u}g^{1/2})$ and \eqref{eq:size of iterated gradient} for $e^\uparrow(\widechi_j)$.  Observe also that the integrand vanish whenever $\nabla\widechi_j=0$, since it contains a factor $e^\uparrow(\widechi_j)$; finally, we use \eqref{eq:local-global} to change the integral to the manifold. 

For general $m$, suppose now that \eqref{eq:inductionstep} holds for a given $m$, and we want to prove it for $m+1$. We integrate the last term
by parts using again \eqref{parts} to obtain the expression  

\begin{multline}
\label{eq:sum_of_integrals_next_step}
\e\sum_j\, \int_{T_{p_j}M}\, \,\exp\left({\frac{2\pi i}{\e}\ip{k}{v}}\right)
\,P_k^m(e^\uparrow(\widechi_j))\,P_k(\bar{u}g^{1/2})(v)\,dv +
\\
+
\e\sum_j\, \int_{T_{p_j}M}\, \bar{u}(v)\,\exp\left({\frac{2\pi i}{\e}\ip{k}{v}}\right)
\,P_k^{m+1}(e^\uparrow\widechi_j)\,g^{1/2}(v)\,dv
\,
\end{multline}

Once again, in the first term, we bound $P_k^m(e^\uparrow(\widechi_j))$ and $P_k(\bar{u}g^{1/2})$ using \eqref{p1} and \eqref{p2} respectively, changing the integrals to the manifold by \eqref{eq:local-global}, thus finishing the proof.
\end{proof}

\begin{Lemma}
\label{hfm}
Let $u\in H^1(M)$, and let $f:\TM\to\R$ be a smooth function. Then for any $s>\frac{n}{2}-1$, there exists a constant $C_s>0$ such that
\begin{multline}
\label{eq:sumandotodo}
\left|\sum_j \sum_{|k|\ge 1} \hat{f}(p_j,k) \int_{T_{p_j}M}\, \bar{u}(v)\,\exp{\left(\frac{2\pi i}{\e}\ip{k}{v}\right)}
\,e^\uparrow(\widechi_j)\,g^{1/2}(v)\,dv\,\right|  
\\
\le  C_s \epsilon^{(\alpha-\beta)/2}\, \|u\|_{H^1(M)}\, \|f\|_{C(M,H^s(\TM_p))}
\end{multline}

\end{Lemma}
\begin{proof}
By applying the  triangle inequality in \eqref{eq:sumandotodo}, and Lemma~\ref{hfma} to each term in the resulting sum, we obtain the following upper  bound (up to a constant) for the left-hand side in \eqref{eq:sumandotodo}: 
\[ 
\sum_j \left(\int_{\nabla \psi_j \neq 0}
|u|+|\nabla u|\, dp \right)\cdot \sum_k \frac{|\hat{f}(p_j,k)|}{|k|}  .
\]
By  H\"older inequality we get that for each $j$ and exponent $s$

\[ 
\sum_k \frac{|\hat{f}(p_j,k)|} {|k|}  \le \|f\|_{C(M,H^s(\TM_p))} \left(\sum |k|^{-2(s+1)}\right)^\frac12 \]
which converges if $s>\frac{n}{2}-1$. 
Thus, the left-hand side of \eqref{eq:sumandotodo} is bounded by 

\[ \|f\|_{C(M,H^s(\TM_p))} \sum_j   \int_{\{\nabla \psi_j \neq 0\}}
|u|+|\nabla u| dp=  \|f\|_{C(M,H^s(\TM_p))} \int_{ \cup_j \{\nabla \psi_j \neq 0\}} |\nabla u|+|u|\, dp \]
which applying Cauchy-Schwarz in $M$ is bounded by 

\[ \|f\|_{C(M,H^s(\TM_p))} \|u\|_{H^1(M)} 
\cdot \vol( \cup _j\supp \nabla \psi_j  )^\frac12 
%| \cup_j \nabla \psi_j \neq 0|^\frac12  
\]
and thus, the Lemma follows from inequality \eqref{eq:volume of all gradients} in part 3 of Theorem~\ref{thm:usable_partition_of_unity}.
\end{proof}

Once higher Fourier modes are discarded, the $0$-th order Fourier modes will have a contribution in the limit.
To simplify the writing, we will again state the result for the vertical vector field $X[p,v]=f[p,v]e^\uparrow$ for some  frame vector field $e$ in $M$.

\begin{Lemma}\label{0fm}
Let $u \in H^1(M)$ and $X[p,v]=f[p,v]e^\uparrow$ a vertical vector field in $\TM$, where $e$ is some vector field in the frame defining the parallelization of  $M$. If $\widetilde{X}$ is the average vector field of $X$ in $M$ as it appears in Definition \eqref{eq:average_vector_field}, then 
\begin{equation}
\label{eq:zero_order_fourier}
\lim_{\epsilon \to 0}\sum_j 
\int_{M}   u(q) 
\hat f(p_j,0)\,
e^{\downarrow}(p_j;q)(\psi_j)
%\ip{d_q\psi_j}{e^{\downarrow}(p_j;q)} 
dq =
\int_M u \div\widetilde{X}\, dq.
\end{equation}

\end{Lemma}

\begin{proof}
It is clear that
\[
\div(u\psi_je^\downarrow(p_j,\cdot))=
u\cdot e^\downarrow(p_j;\cdot)(\psi_j)+
\psi_j\cdot e^\downarrow(p_j;\cdot)(u)+
u\psi_j\div e^\downarrow(p_j;\cdot)
\]
Thus, using that the integral of a divergence vanishes in a manifold without boundary, we get that each of the integrals above equals
\begin{multline}
\label{eq:parts_to_put_away_chi}
\int_{M}   u(q) 
\hat f(p_j,0)\,
e^{\downarrow}(p_j;\cdot)(\psi_j)
%\ip{d_q\psi_j}{e^{\downarrow}(p_j;q)} 
dq = 
\\
= 
-\int_{M} \psi_j\hat f(p_j,0)\cdot e^\downarrow(p_j;\cdot)(u)\,dq
- \int_{M} u\psi_j\hat f(p_j,0)\div e^\downarrow(p_j;\cdot)\,dq
\end{multline}

We start by examining the first integral in the right-hand side of the above equation \eqref{eq:parts_to_put_away_chi}, trying to write it in a more compact form. To these aim, we claim first that 
\[
\lim_{\e\to 0} \left|\sum_j \int_{M} \psi_j\hat f(p_j,0)\cdot e^\downarrow(p_j;\cdot)(u)\,dq -
\int_{M} \psi_j\hat f(q,0)\cdot e^\downarrow(p_j;\cdot)(u)\,dq 
\right| = 0
\]

To see this, observe that the function 
$
q\in M\to \hat{f}(q,0)
$
is uniformly continuous, since $f$ itself can be considered continuous, $M$ is compact and $\hat{f}(p,0)=\int_{\TM_p}\,f[p,v]\,dv$ is given integrating over a compact fiber. Thus, for any  $\eta>0$, we can find some $\e>0$ such that
$|\hat f(p,0)-\hat f(q,0)|<\eta$ whenever $d(p,q)\leq \e^\beta$. Therefore,
\begin{multline}
\sum_j \int_{M} \psi_j |\hat f(p_j,0)-\hat f(q,0)| \cdot |e^\downarrow(p_j;\cdot)(u)|\,dq
\\
\leq 
\eta\cdot\sum_j  \int_{M} \psi_j \cdot |e^\downarrow(p_j;\cdot)(u)|\,dq
\leq
\eta\cdot\sum_j  \int_{M}  \psi_j |e^\downarrow(p_j;\cdot)|\cdot|\nabla u|\,dq
\end{multline}
But on the support of $\psi_j$, we can estimate by above $|e^\downarrow(p_j,\cdot)|\leq 2|e_{p_j}|$, since  $e^\downarrow(p_j,\cdot)=(d\exp_{p_j})_v(e_{p_j})$, and from Lemma \ref{lem:almost Euclidean charts} we had an upper bound $\|d\exp_{p_j}\|\leq 2$.  
Thus, using that $\sum_j\psi_j=1$, the above integral can be bound above by
\[
\eta \cdot\max_{q}\|e(q)\|\cdot \int_M |\nabla u|\,dq \leq 
C\eta
\]
where we have also used that $u$ is in $H^1(M)$ and Cauchy-Schwarz. Since $\eta$ is arbitrary, this proves our claim.

A similar argument using again Lemma \ref{lem:down_vectors_variation} shows also that
\[
\lim_{\e\to 0}\left|\int_{M} \hat f(q,0)\cdot e(u)(q)\,dq-
\sum_j\int_{M} \psi_j\hat f(q,0)\cdot e^\downarrow(p_j;\cdot)(u)\,dq 
 \right| =0
\]
where in the first term, we have used again that $\sum_j\psi_j=1$. 
Thus, we get that
\begin{equation}
\label{eq:first_integral_in_parts}
\lim_{\e\to 0}\sum_j\int_{M} \psi_j\hat f(p_j,0)\cdot e^\downarrow(p_j;\cdot)(u)\,dq =
\int_{M} \hat f(q,0)\cdot e(u)(q)\,dq
\end{equation}
and, using integration by parts on the right side, we get
\begin{equation}
\label{eq:first_integral_in_parts_buena}
\lim_{\e\to 0}\sum_j\int_{M} \psi_j\hat f(p_j,0)\cdot e^\downarrow(p_j;\cdot)(u)\,dq =
- \int_M\, u\div (\hat{f}\cdot e)\,dq,
\end{equation}
in which we recognize $\hat{f}\cdot e$ as the average vector field $\tilde{X}$.

We now turn our attention to the second integral in \eqref{eq:parts_to_put_away_chi}. We claim now that
\[
\lim_{\e\to 0}\sum_j \int_{M} u\psi_j\hat f(p_j,0)\div e^\downarrow(p_j;\cdot)\,dq =0
%\int_M\, u \cdot\sum_j \psi_j \div(\hat f(p_j,0)e^\downarrow(p_j;\cdot))\, dq
\]
To see this, observe that $\psi_j\div e^\downarrow(p_j;\cdot) =0$ outside of the support of $\psi_j$, that lies in a ball of radius less than $C\e^\beta$.
Furthermore, by Lemma \ref{lem:divergence_down}, $|\div e^\downarrow(p_j,\cdot)(q)|\leq Cd(p_j,q)$ inside the support of $\psi_j$.
Thus,
\[
\left|\sum_j \int_{M} u\psi_j\hat f(p_j,0)\div e^\downarrow(p_j;\cdot)\,dq \right|
\leq
C\e^\beta\sum_j\int_{M} |u|\psi_j|\hat f(p_j,0)|\,dq 
\leq C_1\e^\beta
\]
where we are bounding above $|\hat f(p_j,0)|$ by the maximum of the function $q\to|\hat f(q,0)|$, and we have used once Cauchy-Schwarz to obtain the $L^2$ norm of $u$.

The combination of these two claims yields

\begin{equation}
\label{eq:after_two_claims}
\lim_{\epsilon \to 0}\sum_j 
\int_{M}   u(q) 
\hat f(p_j,0)\,
e^{\downarrow}(p_j;\cdot)(\psi_j)
%\ip{d_q\psi_j}{e^{\downarrow}(p_j;q)} 
dq =
\int_M u \div\widetilde{X}\, dq,
\end{equation}
as desired.
\end{proof}

\subsection{Proof of Theorem \ref{vector_limit}}
We have finally all the elements necessary to prove the main Theorem of Section \ref{sec:by_parts_main}.  

\begin{proof}[Proof of Theorem \ref{vector_limit}]
We will start with the simpler case $X=f[q,v]e^\uparrow$.
As mentioned before, we write
\[
 -\int_M\, (du, X^\e) \, dq = \int_M\, u\div X^\e \, dq,
\]
and replace $X^\e$ by its formula  to derive equation \ref{by_parts}, that we recall here for the reader's convenience:
 \begin{multline}
 \label{by_parts2}
-\int_M\, (du, X^\e) \, dq =
\sum_j \int_{M}   u(q) 
f(p_j, \exp^{-1}_{p_j}(q)/\epsilon)
({d_q\psi_j(q)},{e^{\downarrow}(p_j;q)}) dq 
\\
+\sum_j \int_M u(q) \psi_j(q)  
({d_q f(p_j, \exp^{-1}_{p_j}(q)/\epsilon)},{e^{\downarrow}(p_j;q)} )  dq
\\
+  
\sum_j\int_M u(q) \psi_j(q) f(p_j, \exp^{-1}_{p_j}(q)/\epsilon) 
\div   (e^{\downarrow}(p_j; q))\, dq.      
\end{multline}
We explain now how to treat each of the three terms in the right:
\begin{itemize}
\item Once we expand $f$ in its Fourier series, the first integral breaks in two types depending on whether $k=0$ or $|k|\geq 1$, $k$ being the multi-index in the series:  
\[\begin{aligned}
& \sum_j \int_{M}   u(q) 
f(p_j, \exp^{-1}_{p_j}(q)/\epsilon)
({d_q\psi_j(q)},{e^{\downarrow}(p_j;q)}) dq \\ & = \sum_j  \sum_k \hat f(p_j, k) \int_{T_{p_j}M}\, \bar{u}(v)\,\left(\exp{\frac{2\pi i}{\e}\ip{k}{v}}\right)
\,e^\uparrow(\widechi_j)\,g^{1/2}(v)\,dv \end{aligned} \]

 Lemma \ref{hfm} 
implies that the sum of all terms with $|k|\geq  1$ is bounded by a constant times $\e^{\alpha-\beta}$ and thus tends to zero with $\e$. 

 Furthermore, Lemma \ref{0fm} shows that  the remaining integral containing $\hat{f}(0,\cdot)$ approaches $\int_M \,u\,\div\tilde{X}\,dq$.
\item For the second integral in  \eqref{by_parts2}, we use Lemma \ref{lem:df_divided_bi_epsilon}; namely,
\[
({d_pf(p_j,\exp_{p_j}^{-1}(q)/\e)},{e^\downarrow(p_j;q)})=
\frac{1}{\e}e^\uparrow(\tilde{f})_{\exp_{p_j}^{-1}q/\e}.
\]
Observe that, since $\div^v e^\uparrow=0$, this agrees with
\[
\div^v(f\cdot e^\uparrow)=
f\cdot\div^v(e^\uparrow)+e^\uparrow(\tilde{f})
\] 
at the point $[p_j,\exp_{p_j}^{-1}q/\e]$.

Plugging this in the integral at hand, we obtain
\begin{multline}
\sum_j \int_M u(q) \psi_j(q)  
({d_q f(p_j, \exp^{-1}_{p_j}(q)/\epsilon)},{e^{\downarrow}(p_j;q)})   dq = \\
\frac{1}{\e}\int_M\,u(q)\,(\sum_j\psi_j\cdot \div^v(f\cdot e^\uparrow)[p_j,\exp_{p_j}^{-1}q/\e])\, dq=\frac{1}{\e}
\int_M\,u(q)\,h^\e(q)\,dq
\end{multline}
where recall that we were denoting $h[q,v]=(\div^v X)[p,v]$.

\item Finally, the last integral in \eqref{by_parts2} will approach zero when $\e\to 0$. The argument is entirely similar to that at the end of the previous section. Namely, by lemma \ref{lem:divergence_down}, $|\div(e^\downarrow(p_j;q))|\leq C d(p_j,q)\leq C'\e^\beta$ in the support of $\psi_j$. Using Cauchy-Schwarz will result in an upper bound of the form $C\|u\|_{H^1(M)}\e^\beta$.
\end{itemize}

The case for general $X$ is now immediate: just write it as $\sum_{i=1}^n f_i\cdot e_i^\uparrow$ for $e_i$ in the frame, and use linearity. 
\end{proof}

\section{Two-scale convergence of gradients} \label{sec:2sdifferentials}
In this section, we establish the relationship between the two-scale convergence of functions and its gradients.  We recall the function space ${L^2(M,\Hper(\TM_p))}$ introduced in Section~\ref{subsec:functionspaces} and the corresponding norm,
\begin{equation*}
\|u\|^2_{L^2(M,\Hper(\TM_p))}=\oint_{\TM} |\grad_v u(p,v)|^2 \,dv\,dp 
\end{equation*}

We start by a lemma that resembles the orthogonal decomposition of gradients and divergence free vector fields, but we will establish it only on the vertical part of $\TM$. Here we work
in the $L^2$ setting, and thus we need to  give a distributional definition of divergence free vector fields. Namely, we say that $Y \in L^2(\Vvectors)$ satisfies that $\div^v Y=0$ if

\[ 
\int_{\TM} \langle \grad_v u, Y \rangle \,dv\,dp =0, 
\]
for every  $u \in L^2(\Omega,\Hper(\TM_p))$.  In the next proposition we state the orthogonality relationship against smooth vector fields but, since being divergence free is 
a linear condition, it is equivalent to state the orthogonality condition for general $L^2$ vector fields. 
\begin{Prop}\label{prop:verticalLeray}
Let $X[p,v] \in  L^2(\Vvectors)$ such that 
\begin{equation}
\label{orthogonality} 
\int \langle X, Y \rangle \, dv\,dp=0 
\end{equation}
for every $Y\in  \Vvectors$ such that $\div_v Y=0$. Then there exists  a unique $u \in L^2(M,\Hper(\TM_p))$ 
\[X=\grad_v u \]
\end{Prop}

\begin{proof} Let $X[p,v]=X^{i}[p,v] \partial_{v_i}$, where we have used coordinates $v_i$ such that $X^{i}[p,v] $ is $\mathbb{Z}^n$ periodic. (For the sake of clarity, we identify coordinates in  $\TM_p$ with  those in $TM_p$ and we identify $\tilde{X^{i}}$ with $X^{i}$.)

 Then we use
classical Fourier series to write for each coordinate,

\[ X^{i}[p,v]= \sum_{k \in \mathbb{Z}^n}\hat{X^{i}}(p; k)\,e^{2\pi i k\cdot v},
 \]
 
 Now, if the metric in the frame were the identity, the  multiplier of the Laplacian would be $|k|^2$. Instead, we define $ |k|^2_{p}=k_i k_j g^{ij}[p]$. 
 Then we fix $p$ and declare 
 \[
 u[p,v]=\sum_{|k|>0} \frac{k_i}{|k|_p^2}\hat{X^{i}}(p; k)\,e^{2\pi i k\cdot v},
 \]
 which is measurable in $[p,v]$ and  satisfies that 
   \[ \int_{\TM_p} u\, dv=0 \]
 for almost every $p$. Moreover, by construction,  
 \[ \div^{v} ( \grad_v u-X)=0 \]
as distributions. Thus, for  $Y=\grad_v u-X$, it holds that
  \[\int_{\TM} \langle \grad u, Y  \rangle \,dv\,dp=0 \]
  while by approximation with smooth vector fields, \eqref{orthogonality} still  implies that 
 \[
 \int_{\TM} \langle X, Y \rangle \,dv\,dp=0 .
 \]
 Therefore we get that
\[ 
 \int_{\TM} | \grad_v u-X|_{G(p)} ^2 \,dv\,dp =0,
 \]
that is  $X[p,v]=\grad_v u$  as $L^2$ functions. 
 Finally, we estimate
 
 \[ \int_{\TM} |\grad_v u|^2 dpdv= \int  \sum_i \sum_k |\hat{X}_i[p,v]|^2 dp dv  \le \|X\|_{L^2(\TM)} \]
 which yields the required bound in the space   $L^2(M,\Hper(\TM_p))$
\end{proof}

The next Theorem is the main result  of this  section. 

\begin{Theorem}[Convergence of gradients and divergence-free fields] 
\label{differential}
Let $M$ be a closed Riemannian manifold and let $u_\e:M\to\R$ be  a sequence of functions in $H^1(M)$.
\begin{enumerate} 
\item If $u_\e \rightharpoonup u$ in  $H^1(M)$,
then $u_\e$  two-scale converges to $\tilde{u}=u \circ \pi$,
where $\pi:\TM \to M$ is the projection $[p,v]\to p$. Moreover,   $\g u_\e(p)$ $\,$ two-scale converges to $(\g_p u)^\uparrow+ \grad_v u_1(p, v)$, for some  function $u_1 \in  L^2(M,H^1(\T M_p))$.
\item Let $u_\e \in H^1(M)$ uniformly bounded. Then, there is $u_0 \in L^2(M, H^1(\TM_p))$
such that, up to a subsequence, $u_\e$ two-scale converges to $u_0$, and
$\e\, \g_p u_\e$  two-scale converges to $\g_v u_0$.

\item 
 Let $X_\e \in L^2\mathfrak{X}(M)  \2s X_0 \in L^2 \mathfrak{X}(\V)$ and 
$\div_q X_\e=0$. Then, $\div^v X_0=0$ and $\div_q \tilde X=0$, where $\tilde{X}$ is the average of $X$ over the fibers of $\TM$
\end{enumerate}
\end{Theorem}

\begin{proof}
Corollary~\ref{cor:H1_orthogonal_to_divergence} states that   for any function $u$ in $H^1(M)$, and for any smooth vertical test vector field $X$ it holds that 
\[
\lim_{\e\to 0}\int_M\,u\cdot(\div^v X)^\e dq =0,
\]
and the convergence is uniform for $u$'s lying in a compact set in $H^1(M)$. Hence, we have that 
\[
\lim_{\e\to 0}\int_M\,u_\e\cdot(\div^v X)^\e dq =0,
\]
and if $u_0[p,v]$ is the two-scale limit of $u_\e$, then 
\[
\int_{\TM}\,\vol (\T M_p)^{-1} u_0\cdot \div^v X \,dv\,dp =0.
\]

We claim that this actually implies that $u_0$ is constant along the fibers of $\TM$, that is, $u_0[p,v]=u(p)$ for some function $u:M\to\R$. 
One way to see this is to use Lemma \ref{lem:vertical_divergence_versus_regular_divergence} to get that for a vertical vector field $X$, 
$\div^v X=\div X$. Then
\[
\int_{\TM}\, X(\vol (\T M_p)^{-1}u_0) \,dv \, dp = -\int_{\TM} \vol (\T M_p)^{-1}\,u_0\cdot \div X \,dv\,dp =0,
\]
and since $X$ is an arbitrary smooth vertical field, we get our claim on $u_0=u\circ\pi$ for some function $u$ on $M$. 

Let now $X$ be a vertical vector field with $\div^v X=0$. Observe that in  \eqref{12.2.17}, the term with $h^\e$ vanishes since $h=\div^v X$.  Using once again uniform convergence in compact sets in Theorem \ref{vector_limit}, we obtain that
\[
\lim_{\e\to 0} \int_M\,  \langle \grad u_\e,X^\e \rangle \,dq=
-\lim_{\e\to 0} \int_M\,u_\e \cdot\div \tilde{X}\, dq
\]
If we denote by $Y\in L^2 (\T M)$ the two-scale limit of $\grad u_\e$, the above yields
\[
\oint_{\T M} \langle Y,X \rangle [p,v]\,dp \,dv= - \int_{M}\, u\div\tilde{X} \,dp
=\int_{M}\, \tilde{X}(u)\,dp
\]
where we have integrated by parts in $M$. Recalling that for a vertical vector field of the form $X= \sum_{i=1}^n\,f_{i}[q,v]\,e_i^\uparrow$, its average
$\tilde{X}$ is defined as
\[
\widetilde{X}_q:=
 \sum_{i=1}^n\,\left(\vol (\T M_q)^{-1}\int_{\TM_q}\, f_{i}[q,v]\,  dv\,\right)\cdot e_i, 
\]
 we obtain, using \eqref{eq:one_form_lift_defn},
\begin{multline}
\int_{\T M} \vol (\T M_p)^{-1}\langle Y, X\rangle [p,v]\,dp \,dv=
%\int_{M}\, \tilde{X}(u)\,dp =
\sum_{i=1}^n\,\int_{M}\,\left(\vol (\T M_p)^{-1}\int_{\TM_p}\, f_{i}[p,v]\,  dv\,\right)\cdot e_i(u)\, dp
=
\\
\sum_{i=1}^n\,\int_{\TM}\,\vol (\T M_p)^{-1}\,f_{i}[p,v] \langle \grad u^\uparrow,e_i^\uparrow \rangle \, dv\,dp
=
\int_{\TM}\,\vol (\T M_p)^{-1} \langle  \grad u^\uparrow,X \rangle \, dv\,dp.
\end{multline}

This implies that the vector field  $\vol (\T M_p)^{-1}((Y)-(\grad u)^\uparrow)$ integrates to zero when paired with an arbitrary vertical vector field with $\div^vX=0$. Then, from Proposition  \ref{prop:verticalLeray}, we deduce the existence of a unique  function $\tilde u_1 \in L^2(M, H^1(\TM_p))$ such that 
\[\grad_v \tilde{u_1}=\vol (\T M_p)^{-1}((Y)-(\grad u)^\uparrow) \]
which yields that $Y$, the two-scale limit of $\grad u_\epsilon$ satisfies that 
\[Y=\vol{\TM_p} (\grad u)^\uparrow+\grad_v u_1,\]
for $u_1=\vol (\T M_p)\tilde u_1$. %=\textrm{2-s}- \lim (d u_\e)$ yielding the first claim of the theorem. 

 To prove (2), denote by $Y$ the two-scale limit of $\e\cdot  \grad u_\e$.
Multiplying equation \eqref{12.2.17} by $\e$ and passing to the limit, we  obtain

\begin{equation}
\label{eq:step1_limit}
\oint_{\TM} u_0\cdot h [p,v]\, dv\, dp=
-\oint_{\TM} \langle Y,X\rangle [p, v]\,dp\, dv.
\end{equation}
Integrating by parts in $v$ in the left-hand side of this equation, we see that, 
since $h[p,v]=\div^v X[p,v]$ and $X$ is vertical,
\begin{equation}
\label{eq:step2_limit}
%\int_{\T M} u_0(p, v) g(p,v) dv dp=
\oint_{\TM}\, u_0\cdot\div^v X [p,v]\, dv\, dp=
\oint_{\T M}\,  \langle \grad_v  u_0, X \rangle [p,v]\,dv\, dp
%-\int_{\T M} \left< \omega_0(p, v),  {\mathcal F}(p, v) \right> \,dp dv.
\end{equation}
From \eqref{eq:step1_limit} and \eqref{eq:step2_limit}, and since $X$ is an arbitrary vertical vector field, and $Y$ as well as  $( \grad_v u_0)$ are vertical vector fields, we obtain that 
$Y= ( \grad_v u_0)$. This concludes the first two parts of the Theorem.

To prove the last, recall first that  Lemma \ref{lem:compactnesstensors}  states that $\tilde X \in L^2 \mathfrak{X}(M)$ is the weak limit of the sequence $X_\e$. Then for any function $f\in C^{\infty}(M)$, 
\[
\lim_{\e\to 0}\int_M\,f\cdot \div X_\e\, dq=
\lim_{\e\to 0}\int_M\,  \langle X_\e, \grad f \rangle \, dq =
\int_M\,      \langle \tilde X , \grad f \rangle        \, dq =\int_M\,f\cdot \div \tilde X\, dq,
\]
and since $\div X_\e=0$ and $f$ is arbitrary, $\div \tilde X =0$.

In order to deal with the vertical divergence, observe that $\TM\simeq M\times \T^n$, and under this identification we can consider functions 
$f:\TM\to\R$ of the form $f[p,v]=\phi(p)\xi(v)$; more precisely, if $\varphi:\TM\to M\times \T^n$ is this fiber preserving diffeomorphism, we define
$f[p,v]=\phi(p)\cdot \xi(\pi_2\circ\varphi[p,v])$, where $\pi_2:M\times\T^n\to\T^n$ is the projection onto the second factor. 
These functions satisfy the hypothesis of Proposition \ref{prop:gradient_epsilon_and_epsilon_gradient}, thus 
\begin{equation}\label{eq:commutatoragain}
\lim_{\e\to 0}\int_M \,\|\e\grad_q (f^\e) - (\grad_v f)^\e\| \,dq =0.
\end{equation}
Then
\begin{equation}
0=\e\int_M\,\div X_\e\cdot f^\e\, dq=
-\int_M\,\langle X_\e,\e\grad f^\e\rangle\,dq,
\end{equation}
and from \eqref{eq:commutatoragain},
\begin{equation}
0=\lim_{\e\to 0}\int_M\,\langle X_\e,(\grad_v f)^\e\rangle\,dq =\oint_{\TM}\, \langle X_0,\grad_v f\rangle\, dq\,dv,
\end{equation}
since $X_0$ is the two-scale limit of the $X_\e$. 
Integrating by parts,
\begin{equation}
\oint_{\TM}\, \div^vX_0\cdot f\, dq\, dv = 0,
\end{equation}
and since the choices of $\xi$ and $ \phi$ are arbitrary in the formula for $f$,  by density we get  $\div^vX_0\equiv 0$ in the weak sense. 

\end{proof}

\part{Elliptic equations}\label{part:equations}
\section{Two scale convergence and   elliptic equations}\label{sec:2selliptic}

We start by defining the correct analogue in the manifold $M$ and in the torus bundle to matrices which are elliptic and symmetric. We will use 
 the horizontal-vertical decomposition of the tangent space of $\TM$ as explained in \eqref{eq:horiz_and_vertical_in_torus_bundle} and the general theory of 
 vertical tensors developed in Section~\ref{sec:2stensors}. 

 \begin{Def}
We denote by $ T^{1,1}_K(M)$ the class of  sections of the bundle of linear endomorphisms of $\mathfrak{X}(M)$ which are uniformly  elliptic with constant $K$ and  self-adjoint with respect to the metric $g$, i.e. $A \in T^{1,1}_K(M)$
 if
$$K^{-1}\langle X,X\rangle  \leq \langle A[q]X,   X\rangle \leq  K\langle X,X\rangle$$ 
and, taking coordinates respect to the frame, we have that 

\begin{equation}
\label{} 
\sum_{\ell} A_i^\ell [q]g_{\ell k}(q)=\sum_{\ell}A^\ell_{ k}[q] g_{\ell i} (q)
\end{equation}
for $q \in M$.
\end{Def}

 Given $A$ in $ T^{1,1}(M)$  a non necessarily self-adjoint endomorphism,    we define its adjoint $A^{ad}$ by the formulas

\begin{equation} \label{eq:adjoint}
\sum_{\ell} A_i^\ell [q]g_{\ell k}(q)=\sum_{\ell}(A^{ad})^
\ell_{ k}[q] g_{\ell i} (q) 
\end{equation}
That is to say, in matrix form $A^{ad}=GA^tG^{-1}$, where $A^t$ is the transpose of $A$. When we are  considering a  domain $\Omega\subset  M$, we
use the notation $ T^{1,1}_K(\Omega)$.

\begin{Def}
\label{def:elliptic_self_adjoint_vertical_endomorphism}
We denote by  $ T^{1,1}_K({\mathcal{V}})$
the class of sections of  vertical linear     bundle endomorphism of $\mathcalV$ which are uniformly elliptic with constant $K$ and which are self-adjoint with respect to the scalar product given in  \eqref{eq:inner_product}.
That is, for each $[p,v] \in \TM$, we have a linear map
\[
A[p,v]:\mathcalV_{[p,v]}\to \mathcalV_{[p,v]}
\]
depending on $[p,v]$ which satisfies that  for $w \in \mathcalV_{[p,v]}$

\[ \frac{1}{K} \ip{w}{w}\le \ip{A[p,v] w}{w} \le K \ip{w}{w} .\]

\end{Def}

 We recall that $\ip{*}{*}$ is the vertical inner product  given in  \eqref{eq:inner_product}, where we are using the frame $e_i(p)$ to construct the frame $e_i^\uparrow[p,v]$.
 Thus, we can write $A$ in matrix form as
\begin{equation}
A[p,v]e_i^\uparrow := \sum_\ell \,A_{i}^\ell [p,v] \,e_\ell^\uparrow,
\end{equation}
for some functions $A_{i}^\ell:\TM\to\R$ such that 
\begin{equation}\label{eq:selfadjoint}
\sum_{\ell} A_i^\ell g_{\ell k}=\sum_{\ell} A^\ell_{ k} g_{\ell i}. 
\end{equation}
 Explicitly, if $\tildeX=\sum_{i=1}^n\, f_i[p,v] e_i^\uparrow$, then
\[
A\tildeX=\sum_{\ell}\,\sum_i \left(A_i^\ell[p,v] f_i\right)\,e_\ell^\uparrow.
\]
 % For a function space $C(M,\mathcal{B}(\TM_p))$, we say that $A \in C(M,\mathcal{B}(\TM_p, \Sigma_K^{\mathcal{V}}))$, if the coordinates functions  $A_i^\ell$ are in $C(M,\mathcal{B}(\TM_p))$. 
For the regularity on the coordinate functions of $A$, we will use the terminology of subsection \ref{subsec:functionspaces}.

Notice that we have stated the result for $(1,1)$-tensors for readers with a Euclidean mind. Since we are dealing with self-adjoint maps, we could also identify $A$ with the  symmetric
 $(2,0)$-tensor given by
 \[
 A[w_1,w_2]=\langle Aw_1, w_2 \rangle,\quad  \text{where } w_1,w_2\in\calV_{[p,v]} 
 \]
When working on a specific domain $\Omega$ of $M$, we will denote by  $T^{1,1}_K(\V_\Omega)$ the vertical  $(1,1)$-tensors on the torus bundle over $\Omega$.

Given a sequence $A_\epsilon \in  T^{1,1}_K(M)$, we denote by   $u_\epsilon \in H^1_0(\Omega)$ the unique solution to 
\begin{equation}
\label{eq:div_eqn}
\textrm{div} (A_\epsilon \g u_\epsilon)=f,
\end{equation}
where we have dropped the $f$-dependence on the notation.

The next theorem is central  in  the paper, for it is where we identify the weak limit (and the two-scale limit) of $u_\epsilon$.

\begin{Theorem}\label{2stheorem} 
Let $\Omega$ be a domain in $M$. Let $A_\epsilon \in  T^{1,1}_K(\Omega) $ be  strongly two-scale convergent to $A \in T^{1,1}_K(\V_\Omega)$.
Consider $u_\e$ the solution of
\begin{equation} \div (A_\epsilon)(\grad u_\e) =f.\end{equation}
with $f \in H^{-1}(M)$. 
Then,
the sequence $u_\epsilon$ converges weakly to $u$ in $H_0^1(\Omega)$ and the sequence $\grad u_\epsilon$ two-scale converges to 
$(\grad_q u)^\uparrow+ \grad^v u_1$ where $u, u_1$ is the unique solution in  the set $H^1_0(\Omega) \times L^2(\Omega,H^1_{per}(\TM))$ of the following two-scale homogenized system:

\begin{align}
&\div^v( A[p,v] [(\g_p u)^\uparrow+ \g_v u_1])=0 \textrm{ in } \pi^{-1}(\Omega),  \label{2shomogenizedproblemv} \\
&\div_p \left(\vol (\T M_p)^{-1}\int_{\TM_p}  A[p,v]( (\g_p u)^\uparrow+ \g_v u_1)\, dv\right)=f,\label{2shomogenizedproblemp} \\
&u=0 \textrm{ on } \partial \Omega,
\end{align}
where $\pi:\TM\to M$ is the standard projection.

\end{Theorem}

\begin{proof}
By the uniform ellipticity of $A_\epsilon$, $u_\epsilon \in H^1$ uniformly. Therefore by Theorem~\ref{differential}, and up to a subsequence which we do not
relabel,  the sequence $\g u_\e(p)$  two-scale converges to $(\g_p u)^\uparrow+ \grad^v u_1(p, v)$.  On the other hand, by the weak formulation of \eqref{eq:div_eqn} 
 it holds that 
\begin{equation}
\label{eq:weakform}
\int A_\epsilon[ \grad u_\epsilon, \grad \varphi]\, dq=\int \varphi f
\, dq
\end{equation}
for every $\varphi \in C_0^\infty(\Omega)$.  
Notice that, as mentioned earlier, we are alternating at our convenience the writing of $A_\e$ as a $(1,1)$ or as a $(2,0)$-tensor, since from the context and the notation, this should not create confusion.

Let $\varphi$ be an arbitrary test function in $\Omega$ and let $\varphi_1 \in C^\infty(\TM)$. 
By  Proposition~\ref{prop:gradient_epsilon_and_epsilon_gradient}, 

\[ \lim_{\epsilon \to 0} \int_{M}   A_\epsilon[ \grad u_\epsilon, \grad \varphi+ \epsilon \grad_p \varphi_1^\epsilon]\,  dq= \lim_{\epsilon \to 0} 
\int_{M}A_\epsilon [\grad u_\epsilon, \grad \varphi+ (\grad_v \varphi_1 )^\epsilon]\, dq \]

Now by Lemma~\ref{lem:constanttensors},  $\grad \varphi \S2s (\grad \varphi)^\up$, and by Lemma~\ref{lem:convergenceofT}, $(\grad_v \varphi_1)^\epsilon \S2s \grad^v \varphi_1$. Therefore,
we can apply compensated compactness for tensors as in Lemma~\ref{lem:cctensors}  to pass to the limit and conclude that 
\[ 
\lim_{\epsilon \to 0} \int_{M}   A_\epsilon[ \grad u_\epsilon, \grad \varphi+ \epsilon \grad_p (\varphi_1)^\epsilon]\,  dp
=
\oint_{\mathbb{T}M} A[\grad u^\up+\grad_v u_1, \grad \varphi^\up+\grad_v \varphi_1] \,dp\,dv 
\]
which deals with the limit of the left-hand side of \eqref{eq:weakform}.  Since the right-hand side is constant, we arrive to 
\begin{equation}\label{eq:weakhomogenized}
\int_{\Omega}\vol(\T M_q)^{-1}\int_{\TM_q}\, \langle A\,(\g u^\up+\g_v u_1), \g\varphi^\up +\grad_v \varphi_1 \rangle \, dv\,dq= \int_\Omega\, f\cdot\varphi\, dq,
\end{equation}
which is the weak formulation of the homogenized system.

To uncouple the system, take first $\varphi_1=0$. Since $\g \varphi^\up$ is constant along each fiber $\TM_q$, thus we can write the left-hand side of \eqref{eq:weakhomogenized} as

\begin{equation}
\int_{\Omega}\left\langle \vol(\T M_q)^{-1}\int_{\TM_q}\,  A\,( \g u^\up+\g _v u_1)\, dv \, ,\,  \g\varphi \right\rangle \,dq= \int_\Omega\, f\cdot\varphi\, dq,
\end{equation}
which is the weak formulation of \eqref{2shomogenizedproblemp}, 
\[
\div_q\left(\vol(\T M_q)^{-1}\int_{\TM_q}\, A\,(\g u^\up+\g _v u_1)\, dv\right)=f,
\]
as desired.

To obtain \eqref{2shomogenizedproblemv}, we switch to the case $\varphi=0$ and $\varphi_1:\TM\to\R$ arbitrary to obtain 
%\textcolor{blue}{the strong two scale convergence of $(A[d\phi^1])^\e$
%and Remark~\ref{rem:testfunctions} to obtain the equation}
\begin{equation}
\int_{\Omega}\vol(\T M_p)^{-1}\int_{\TM_q}\, \left( A\,(\g u^\up+\g_v u_1), d^v\varphi_1 \right) \, dv\,dq= 0.
\end{equation}
Then, we can integrate by parts (taking product test functions as in the end of the proof of Theorem \ref{differential}) to arrive to the desired 
\[
\div^v(A\,(\g u^\up+\g _v u_1)=0
\]
on $\pi^{-1}(\Omega)$.
\end{proof}
Theorem~\ref{2stheorem} truly generalizes \cite[Theorem 1.14] {Allaire92} to the manifold setting.

Let us finish the section with the explicit
examples arising from Section~\ref{sec:2stensors}. For $A$ as in Definition \ref{def:elliptic_self_adjoint_vertical_endomorphism}, we define
\begin{equation}
\label{eq:Aepsilon} 
A^\epsilon =  
\sum_j\, \psi_j(q) \,(H_{\e,j})_*(A[p_j,v]), \qquad  2A^\epsilon_{\sym}=  A^\epsilon+(A^\epsilon)^{\ad}\end{equation}
and 
\begin{equation}
\label{eq:barAepsilon}
\bar{A}^\epsilon= \sum_{i,\ell}^n   \sum \psi_j   A^{i}_{\ell} [p_j,\exp_{p_j}^{-1}(q)/\e] e_i \otimes \theta^\ell, \qquad 2{\bar A}^\epsilon_{\sym}=  {\bar A} ^\epsilon+({\bar A}^\epsilon)^{\ad}  
\end{equation}

Observe that a routine computation using Definitions \ref{def:Tensorcoordinates} and \ref{def:osctensorpull} shows that 
$A^\epsilon_{\sym}$ and ${\bar A}^\epsilon_{\sym}$ are uniformly elliptic with a constant $K'$ dependent of $\e$ but converging to $K$ as $\e$ tends to $0.$

Let us emphasize that for any $(1,1)$-vertical tensors, Lemma~\ref{lem:convergenceofT}  implies that 
$A^\epsilon  \S2s A$, and  Proposition~\ref{prop:strongbarT} implies that $ \bar{A}^\epsilon \S2s A$ as well. However, we can not apply directly Theorem~\ref{2stheorem}
to them because they are not symmetric. In spite of this, we have the following lemma:

\begin{lem}\label{lem:offadjoints}
Let $A \in \TV11 \cap T^{1,1}_K({\mathcal{V}})
$ and $A^\epsilon, {\bar A}^\epsilon$ be defined by \eqref{eq:Aepsilon}, \eqref{eq:barAepsilon}. Then 
\[ \lim_{\epsilon \to \infty} \int_{M}  |A^\epsilon-(A^{\epsilon})^{\ad}|\, dp+ \int_{M}  |{\bar A}^\epsilon-     (  {\bar A}^{\epsilon}    )^{\ad}|\, dp=0 \]
\end{lem}

\begin{proof}
We start with $\bar{A}^\epsilon$. Recall that $(\bar{A}^\epsilon)^{\ad}(q)$ is obtained from $\bar {A}^\epsilon(q)$ through the metric $g$ at $q$ by
\eqref{eq:adjoint}. Since $A \in T^{1,1}_K({\mathcal{V}})$  \eqref{eq:selfadjoint} implies that 

\[ A^{i}_{\ell} [p_j,\exp_{p_j}^{-1}(q)   ]g_{i k}(p_j)= A^{i}_{k} [p_j,\exp_{p_j}^{-1}(q)   ]g_{i \ell } (p_j).
\]

Thus, by the continuity of the metric and the bounded overlapping property it holds that 

\[  |{\bar A}^\epsilon-     (  {\bar A}^{\epsilon}    )^{\ad}| \le C \epsilon ^\beta\]
pointwise. 

In order to deal with $A^\epsilon$ recall that

\[\lim_{\epsilon \to \infty} \int_{M} |A^\epsilon-\bar{A}^\epsilon| \, dq=0 \]

Thus, it follows that 
\[ \lim_{\epsilon \to \infty} |A^\epsilon-(A^{\epsilon})^{\ad}|\, dq=0 \]
as well. 

\end{proof}

\begin{Cor} Let $\Omega$ be a domain in $M$. Let $A_\epsilon=A_{sym}^\epsilon$ or $A_\epsilon={\bar A}^\epsilon_{sym}$ be  strongly two  scale convergent to $A \in T^{1,1}_K(\V_\Omega)$. Let $u_\epsilon$, the solution of equation \eqref{eq:div_eqn}.  Then,
the sequence $u_\epsilon$ converges weakly to $u$ in $H_0^1(\Omega)$ and the sequence $\grad u_\epsilon$ two-scale converges to 
$(\grad_q u)^\uparrow+ \grad_v u_1$ where $u, u_1$ is the unique solution to the system  \eqref{2shomogenizedproblemv} \eqref{2shomogenizedproblemp} 
\end{Cor}
\begin{proof}

Lemma~\ref{lem:convergenceofT}  implies that for  $A^\epsilon  \S2s A$, and thus Lemma~\ref{lem:offadjoints} implies that  $A_{\sym}^\epsilon  \S2s A$ as well. Similarly,
 Proposition~\ref{prop:strongbarT} and Lemma~\ref{lem:offadjoints} guarantee that  $ \bar{A}^\epsilon_{\sym} \S2s A$. The corollary then follows from theorem~\ref{2stheorem}
\end{proof}

\section{Proof of Theorem \ref{thm:hom}}\label{sec:homogenization}
With the  previous  section  at hand, we are in position  to obtain the  homogenization result stated in the introduction. Furthermore,  we  prove  a more general  statement  without  assuming  orthonormality of the frame   with respect to  the metric.

Our strategy is to show that the  solution of the homogenized problem agrees with the unique solution to the two-scale homogenized problem by incorporating a $u_1$ built from the $w_i$'s solutions to the cell problem.  We start by fixing the terminology to avoid later confusion. Weak formulation and existence and uniqueness of solutions is explained in appendix
\ref{appendix:homogenizedproblems}

\begin{Def}  [The two scale homogenized problem] \label{def:2shomogenizedproblem}
Let $A \in T^{1,1}_K(\V_\Omega)$, and $f \in H^{-1}(\Omega)$,
We say that the function pair $u,u_1 \in H_0^{1}(\Omega) \times L^2(\Omega, \Hper(\TM_p))$ is the unique solution to
the two-scale homogenized problem  if
\begin{align}
&\div_v( A[p,v] [(\g_p u)^\uparrow+ \g_v u_1])=0 \textrm{ in } \pi^{-1}(\Omega),  \label{2shomogenizedproblemv2} \\
&\div_p \left(\int_{\TM_p}  A[p,v]( (\g_p u)^\uparrow+ \g_v u_1)\, dv\right)=f,\label{2shomogenizedproblemp2} \\
&u=0 \textrm{ on } \partial \Omega,
\end{align}
\end{Def}

\begin{Def}[The homogenized problem]\label{def:homogenizedproblem}
Let $w_i$ , $1\leq i\leq n$, be solutions to the following cell equations, in the functions space ${L^2(M,\Hper(\TM_p)}$ introduced in Section~\ref{subsec:functionspaces}. Here $\{o_i\}$ is an orthonormal basis of $T_pM$,
 \begin{equation}\label{eq:w_i-equation}
-\div^v (A_{[p,v]}[\g_v w_i+o_i ^\uparrow ])=0.
\end{equation}

We define the homogenized tensor $A^* \in T^{1,1}(M)$ by the symmetric (2,0)-tensor
\begin{equation} \label {homatrix}
A^*[p][X,Y]=\langle\avint_{\T M_p} (I+ (D_vw)^{\ad})A[p,v](I+ (D_vw))\ X^\up \,dv,Y^\up\rangle,
\end{equation}
  where  $(D_v w)^{\ad}= G D_vw^t G^{-1}$      for the matrix $D_vw$ with $i$-column $\grad_v w_i$.

\end{Def}
\begin{Def}We say that $u \in H^1_0(\Omega)$ is the solution of the homogenized system if 

\begin{equation}\label{eq:homogenized_system}
-\div A^*_p(\g u)=f\\
\end{equation}

\end{Def}

The following lemma states that given any function $u$ in $M$ we can lift it to a solution of   equation \eqref{2shomogenizedproblemv2} in the manifold through the solutions to the cell equations \eqref{eq:w_i-equation}.   It yields a way to relate $u$ and $u_1$. 
 \begin{Lemma}\label{lem:u1}
Given  $X\in \mathfrak{X}(M)$ and  $\{w_i\}_{i=1}^n$ the solution of the cell equation \eqref{eq:w_i-equation}. Let us denote  $u_1\in  L^2(\Omega,\Hper(T_p))$ as
$u_1[p,v]= \langle X^{\uparrow},w\rangle$ where,
\begin{equation}\label{eq:vhomogenized}
w=\sum_i\,w_i[p,v]\cdot o_i ^\uparrow.
\end{equation}
Then it holds that for almost every $p$
\begin{equation}
\div_v A(X^{\uparrow}(p) +\grad _v u_1)=0.
\end{equation}
In particular for $u\in H^1(M)$
choosing  $X=\grad u$  and \begin{equation}
u_1[p,v]=\langle\g u^\uparrow,w\rangle
\end{equation}
gives a solution of
\begin{equation}
\label{eq:vert_div_A_vanishes}
\div_v A(\grad u^{\uparrow}(p) +\grad _v u_1)=0.
\end{equation}
which satisfies
\begin{equation}\label {ecuacion}
\g_vu_1=D_vw(\grad u)^\up.
\end{equation}

\end{Lemma}

\begin{proof}
The condition $u_1\in  L^2(\Omega,\Hper(T_p))$ follows from the same condition of the solutions of the cell equation.

From the orthogonality  of  $\{o_i\}$ it holds   $X^{\uparrow}=\sum_i \langle X^{\uparrow},o_i ^\uparrow\rangle o_i ^\uparrow$.
Observe that, with the vertical inner product in the fibers of $\TM$, since each entry $\langle X^{\uparrow},o_i^\uparrow\rangle$ is constant in the fiber, we have
\[
\grad _v u_1 = \sum_i \langle X^{\uparrow},o_i^\uparrow\rangle  \grad_v w_i = D_vw X ^{\uparrow},
\]
where $D_vw$ denotes the matrix with columns $\grad_vw_i$.
 
   We obtain
\[
X^{\uparrow}+\grad_v u_1= (I+ D_v
w) X^\uparrow=\sum_i \langle X^{\uparrow},o_i ^\uparrow\rangle \cdot (o_i ^\uparrow+\grad_v w_i),\]
and  from the cell equations  \eqref{eq:w_i-equation}, we have that each term in the following sum vanishes:
\[\div^v A(X^{\uparrow}+\grad _v u_1)=\sum_i \langle X^{\uparrow},o_i ^\uparrow\rangle \cdot\div^v(A((o_i ^\uparrow)+\g_v w_i)) =0
\]
\end{proof}

\begin{Theorem}\label{thm:finalalternativo}
Assume that $u$ solves the homogenized problem \eqref{eq:homogenized_system} and  define  $u_1[p,v]=\langle\g u^\uparrow,w\rangle$ an in lemma \ref{lem:u1}, then $(u,u_1)$ is the solution of the two scale homogenized problem. Reciprocally assume that $(u,u_1)$ is  the solution  of the two scale homogenized problem, then $u$ solves the homogenized equation   and $u_1= \langle\g u^\uparrow,w\rangle$ .
\end{Theorem}
\begin{proof}
Before addressing the proof, we  start by 
 recalling the weak version of both systems (see \ref{appendix:homogenizedproblems} for further details) and some preliminary observations. 

For the decoupled homogenized system  we have  the cell equations, 
\begin{equation}\label{eq:weak_i-equation}
-\int_{\T M_p}\langle A_{[p,v]}[\g_v w_i(p,v)+o_i ^\uparrow ], \grad _v \phi(v)\rangle\, dv=0,
\end{equation}
and the homogenized equation
\begin{equation}\label{eq:weak_h-equation}
\int_M \langle A^*_p\g u,  \grad \varphi(p)\rangle \,dp=\int_\Omega f \varphi .
\end{equation}
For the two-scale homogenized system we have for a test pair $(\varphi, \varphi^1) \in H^1(M)\times H^1(\T M)$
\begin{equation}\label{weakallaire1}
B[(u, u_1),(\varphi, \varphi^1)]:= \oint_{\pi^{-1}(\Omega)} \langle\, A[p,v](\g u^\up+\g_v u_1), \g \varphi^\up+\g_v \varphi^1 \,\rangle\, dv\, dp=\int_M f \varphi \end{equation}
We can split the above in two equations corresponding to weak versions  of \eqref{2shomogenizedproblemv2} and \eqref{2shomogenizedproblemp2}.

Assume first that $u$ is the solution of the homogenized system in the  weak sense, i.e. for any $\varphi \in H^1(M)$
\begin{equation}
\int_M \langle A^*_p\g u,  \grad \varphi(p)\rangle \,dp=\int_\Omega f \varphi .
\end{equation}
Inserting  definition \eqref{homatrix} we have 
\begin{equation} \label{weakh1}
\int_M\avint_{\TM_p}\langle (I+ (D_vw)^{ad})A[p,v](I+D_vw)(\grad u)^\up, (\grad \varphi)^\up\rangle \,dv\,dp=\int_M f\varphi
\end{equation}
Taking $u_1= \langle \grad u ^\up,w\rangle$ and using   \refeq{ecuacion}  we have on one hand
\begin{equation} \label{weakh2}
=\int_M\avint_{\TM_p}\langle A[p,v]((\grad u)^\up+ \grad_vu_1 ), (I+D_vw)(\grad \varphi)^\up\rangle \,dv\,dp=\int_M f\varphi. 
\end{equation}
 On the other hand  by the  weak version of \eqref{eq:vert_div_A_vanishes} in   lemma \eqref{lem:u1} we have for  any $\varphi^1 \in H^1(\TM)$
 \begin{equation} \label{weakh3}
\oint_{\TM}\langle A[p,v]((\grad u)^\up+ \grad_vu_1 ), \grad_v \varphi^1\rangle \,dv\,dp=0.
\end{equation}
Declare 
 $\tilde{\varphi}:\TM\to\R$ as 
 \begin{equation}\label{omega}
\tilde{\varphi}[p,v]=\langle \g \varphi^\up, w\rangle
\end{equation}
 Take  $\varphi^1= \tilde{\varphi}[p,v]$ in \eqref {weakh3}, this  gives   
$$ \oint_{\TM }\langle A[p,v]((\grad u)^\up+ \grad_vu_1 ), D_vw\g \varphi^\up\rangle \,dpdv=0,$$
 where  we used  $\grad_v \tilde{\varphi}= D_vw \g\varphi^\up$ as in the Lemma.
We can then write  \eqref{weakh2}  as
\begin{equation} \label{weakh4}
= \oint_{\TM}\langle A[p,v]((\grad u)^\up+ \grad_vu_1 ),  (\grad \varphi)^\up\rangle \,dpdv=\int_M f\varphi. 
\end{equation}
which is the  weak version of \eqref{2shomogenizedproblemp2} in the  two scale homogenized system.
To finish, just  notice that \eqref {weakh3} is just the weak version of \eqref{2shomogenizedproblemv2} in the  two scale homogenized system.

Assume now that $(u,u_1)\in H^1(M)\times L^2(\Omega,H^1_{per}(\T_p M))$ is  the solution  of the two scale homogenized problem.
 By choosing  $\varphi=0$ in \eqref{weakallaire1}, we obtain that $u_1$ satisfies \eqref{weakh3} for any $\varphi^1:TM\to \mathbb R$. Then by uniqueness in 
 $L^2(\Omega,H^1_{per}(\T_p M))$ , $u_1=\langle \g u^\up, w\rangle$   given  in the lemma.
 
 Now take $\varphi^1=0$ in \eqref{weakallaire1}, then 
 \begin{equation}\label{weakallaire2}
 \oint_{\pi^{-1}(\Omega)} \langle\, A[p,v](\g u^
 \up+\g_v u_1), \g \varphi^\up \,\rangle\, dv\, dp=\int_M f \varphi .
  \end{equation}
Take $\varphi^1= \tilde \varphi$  in \eqref{weakh3}  as in \eqref{omega}. Then 
$$ \oint_{\TM }\langle A[p,v]((\grad u)^\up+ \grad_vu_1 ), D_vw\g \varphi^\up\rangle \,dpdv=0.$$
Adding this identity  to \eqref{weakallaire2}and using  \eqref{ecuacion}, we obtain that $u$ satisfies the homogenized  equation.

\end{proof}

\appendix
%\part[Appendix]{Appendix}
\section{The tangent bundle}
\label{appendix:tangent}
Let $(M,g)$ be a Riemannian $n$-manifold. 
 The \emph{tangent bundle of $M$}, $TM$, is the set  
\[
TM:=\{\,(p,v)\,:\,p\in M, v\in T_pM\,\}
\]
with the natural differentiable structure induced from  $M$.

Whenever $v\in T_pM$, we use $(p,v)$ to denote the corresponding point in the tangent bundle of $M$. 
We will use $\pi:TM\to M$ to denote the bundle projection 
$\pi(p,v)=p$. Given $(U, \varphi) \in M$ with $\varphi=(x^1,x^2,...,x^n)$ we define a chart $\tilde \varphi: \pi^{-1}U\to  \mathbb{R}^n \times  \mathbb{R}^n$ 
adapted to the frame  $\Gamma(p)=(e_1, \ldots,  e_n)$
by 
\begin{equation}
\label{eq:chart_in_TM}
\tilde{\varphi}(p,v)=(x^1(p), \ldots, x^n(p), v^1(p,v), \ldots, v^n(p,v))
\end{equation}
where for $v\in T_pM$,
\[ v=\sum_{j=1}^n v^i (p,v) e_i. \]

There is an exponential map
\[
\exp:TM\to M, \qquad \exp(p,v)=\exp_p(v),
\]
where $\exp_p(v)$ is $\gamma_v(1)$, and $\gamma_v$ is the geodesic in $M$ with initial condition $\gamma'_v(0)=v$. Another map that is often useful is 
\[
\Exp:TM\to M\times M, \quad \Exp(p,v)=(p,\exp_p(v)).
\]
By the Inverse Function Theorem, $\Exp$ is a diffeomorphism from a neighbourhood of the zero section in $TM$ to a neighbourhood of the diagonal $\Delta_M\subset M\times M$.

The differential of all these maps require us to consider the manifold $TTM$, \emph{the tangent bundle of the tangent bundle}. 
For each $(p,v)\in TM$, there are $n$-dimensional subspaces $\mathcalH_{(p,v)}$, $\mathcalV_{(p,v)}$ in $T_{(p,v)}TM$ called the \emph{horizontal} and \emph{vertical} subspaces, and defined as follows:
\begin{itemize}
\item $\mathcalV_{(p,v)}$ is obtained from vectors tangent to curves of the form $\alpha_w(t)=(p,v+tw)$, where $w$ ranges over vectors in $T_pM$. In our terminology, this would be the vector $w^\uparrow(p,v)$. When $W$ is a vector field in $M$, we get an induced vector field $W^\uparrow$ in TM. Notice that 
$\textrm{ker}(\pi_*)= \mathcalV_{(p,v)}$, where $\pi:TM\to M$ is the bundle projection.
\item $\mathcalH_{(p,v)}$  is obtained as follows: for each vector $w\in T_pM$, we consider the geodesic $\gamma_w$ tangent to $w$, and take the parallel transport $P_t(v)$ of $v$ along $\gamma_w(t)$. This defines a smooth curve $\beta_w(t)$ in $TM$ with $\beta_w(0)=v$; thus, its tangent vector $\beta'_w(0)$ belongs to 
$T_{(p,v)}TM$, and it is easy to see that the set of $ \beta'_w(0)$ obtained when $w\in T_pM$ form an $n$-dimensional subspace of $T_{(p,v)}TM$ that we denote as $\mathcalH_{(p,v)}$. 
\end{itemize}

\begin{figure}
  \begin{tikzpicture}[scale=0.7]
  \draw [ultra thick, fill=lightgray] (0,0) to [out=77,in=283] (0,6) to [out=15,in=165]  (6,6)  to [out=283,in=77] (6,0) to [out=165,in=15] (0,0);
  \node [below] at (2.4,2.8) {$[p,v]$};
  \node  at (5.4,3) {$\mathcal{H}$};
   \node  at (3,5.4) {$\mathcal{V}$};
  \draw[fill] (3,3) circle [radius=1pt];
  \draw[thin, ->] (2.7,3) -- (5,3);
   \draw[thin, ->] (3,2.7) -- (3,5);
   \draw[thick,->] (3,-0.5) -- (3,-2); 
    \draw[ultra thick] (0,-3.0) to [out=15,in=165] (6,-3);
     \node  at (3.3,-1.5) {$\pi$};
     \node  at (5.4,-2.5) {$M$};
      \node  at (5.4,-0.5) {$\mathbb{T}M$};
  \end{tikzpicture}
\caption{The vertizontal descomposition of $\TM$}
    \label{fig:vertizontal descomposition}
\end{figure}
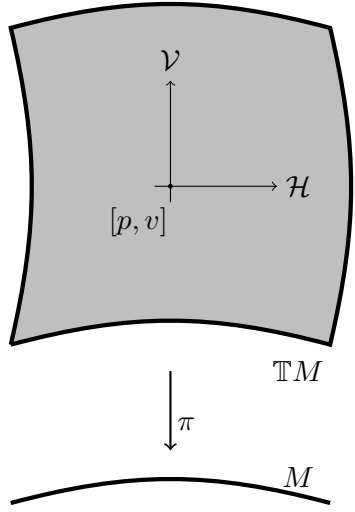

It is well known that 
\[
T_{(p,v)}TM = \mathcalH_{(p,v)}\oplus \mathcalV_{(p,v)} \simeq T_pM \times T_pM,
\]
where the last identification just follows from the construction outlined previously.  For a vector $\xi \in  T_{(p,v)}TM$
we will write

\[ 
\xi=\xi_h+\xi_v, \qquad \xi_h \in \mathcalH_{(p,v)}, \quad \xi_v \in\mathcalV_{(p,v)} 
\]

Similarly, for covectors in $TM$,
\[T^*_{(p,w)}TM=\mathcalH^*_{(p,v)}\oplus \mathcalV^*_{(p,v)}.\] 

That is, for a one form $\omega \in \Lambda^1(TM)$  we can write $\omega=\omega_h+\omega_v $
 where for $\xi \in T_{(p,v)}TM$
 \[
 (\omega_h,\xi)= (\omega,\xi_h),
\qquad 
  (\omega_v,\xi)=(\omega,\xi_v).
 \] 
 We say that  $\omega \in \Lambda^1(TM)$ is vertical if
and only if $\omega_h=0$. 

 In general, an $(m,n)$-tensor $T$ is vertical if 
\[T(\xi_1, \ldots, \xi_m, \omega_1,\ldots, \omega_n)=T((\xi_1)_v, \ldots,  (\xi_m)_v, (\omega_1)_v, \ldots,(\omega_n)_v). \]

Finally, we notice that in the coordinate chart $(\pi^{-1}U,\tilde{\phi})$, a vertical tensor has the expression

\[ T_{(p,v)}= \sum_{I,J} t^I_J(p,v) \frac{\p}{\p v^I} \otimes  dv^J,  \]
where $I,J$ are suitable multi-indexes.

In this paper, we usually consider a metric tensor defined only on vertical vector fields. But whenever we consider a metric on the whole of $\TM$, we consider the \emph{Sasaki metric}. An excellent introduction to it appears in \cite{Paternain}, but we describe it here for completeness and to facilitate the reading:
given vectors $\xi$, $\eta$ in $T_  {(p,v)}TM$,the Sasaki metric is defined as
\[
g_{(p,v)}(\xi,\eta):= g_p(\xi^h,\eta^h)+g_p(\xi^v, \eta^v), 
\]
where $\xi=\xi^h+\xi^v$, $\eta=\eta^h+\eta^v$ in the horizontal-vertical splitting.
It is clear that with this metric, the horizontal and vertical subspaces are orthogonal complements of each other. Moreover, the projection map $\pi:TM\to M$ becomes a Riemannian submersion, although we will not use this fact in the rest of the paper. More important for us, is that the Sasaki metric, although originally defined in $TM$, descends to a metric on the torus bundle $\TM$ with a similar decomposition in terms of the vertical and horizontal components. With this metric (that we will keep calling the Sasaki metric), the vertical torus fibers are orthogonal to the horizontal distribution.

\section{Weak formulation of the homogenized system }\label{appendix:homogenizedproblems}
In this section, we explore carefully  the two scale homogenized system within the tangent bundle formalism.  In particular, 
we state is weak formulation, and prove the existence of unique solutions. For completeness, we also recall the weak formulation of 
the homogenized problem.

\subsection{The two-scale homogenized problem}
\begin{Def}  [The two-scale homogenized system]
Let $A \in \Sigma_K^\mathcal V$, and $f \in H^{-1}(\Omega)$,
We say that the function pair $u,u_1 \in H_0^{1}(\Omega) \times L^2(\Omega, \Hper(\TM_p))$ is the unique solution to
the two-scale homogenized problem  if
\begin{align}
&\div_v( A[p,v] [(\g_p u)^\uparrow+ \g_v u_1])=0 \textrm{ in } \pi^{-1}(\Omega),  \label{2shomogenizedproblemv2b} \\
&\div_p \left(\int_{\TM_p}  A[p,v]( (\g_p u)^\uparrow+ \g_v u_1)\, dv\right)=f,\label{2shomogenizedproblemp2b} \\
&u=0 \textrm{ on } \partial \Omega,
\end{align}
\end{Def}

As we are dealing with weak solutions, the pointwise definition does not make sense, and therefore we interpret it through the variational formulation:

\begin{Def}
We say that $u,u_1$ solves the above system weakly if for every $(\varphi, \varphi_1) \in H_0^{1}(\Omega) \times L^2(\Omega,  H_0^{1}(\Omega) \times L^2(\Omega, \Hper(\TM_p))$
it holds that,

\begin{equation}\label{weakallaireappendix}
\int_{\pi^{-1}(\Omega)} \langle\, A[p,v](\g u^\up+\g_v u_1), \g \varphi^\up+\g_v \varphi_1 \,\rangle\, dv\, dp=\int f \varphi\, dp
\end{equation}
\end{Def}

The bilinear form 
 at the left-hand side o equation \eqref{weakallaireappendix} will be denoted by $B[(u,u_1), (\varphi,\varphi_1)]$.

\begin{Prop}
For smooth coefficients, the definitions of strong and weak solution agree.
\end{Prop}

\begin{proof}
In what follows, we use some of the properties of vertical divergence that appear in Section \ref{sec:vertical_divergence}.
 We integrate by parts in both variables to obtain that
%\begin{equation}
%\begin{aligned}
\begin{multline}\label{B}
\int_{\pi^{-1}(\Omega)} \langle\, A[p,v](\g u^\up+\g_v u_1), \g \varphi^\up+\g_v \varphi_1 \,\rangle\, dv\, dp =\\
-\int_{\pi^{-1}(\Omega)}\,\varphi^\up\cdot \div A[p,v](\g u_\up+\g_v u_1)\, dv\, dp  \\
-\int_\Omega\int_{\TM_p}\,\varphi_1\cdot\div^v A[p,v](\g u^\up+\g_v u_1)\, dv\, dp 
\end{multline}
Since $\phi^\up$ is constant in each $\TM_p$ fiber, we have that the first term is equal to
\begin{equation}
-\int_\Omega\,\varphi\cdot \int_{\TM_p}\, \div A[p,v](\g u^\up+\g_v u_1)\, dv\, dp
\end{equation} 
and here we use that, since $A[p,v](\g u^\up+\g_v u_1)$ is a vertical field, the integral along a fiber of its vertical divergence vanishes, thus
\begin{multline}
\int_{\TM_p}\,\div A[p,v](\g u^\up+\g_v u_1)\, dv =
\int_{\TM_p}\,\div^H A[p,v](\g u^\up+\g_v u_1)\, dv =
\\
\div\int_{\TM_p}\,A[p,v](\g u^\up+\g_v u_1)\, dv = f
\end{multline}
by the properties of divergence. 
On the other hand, the integrand in the second term in \eqref{B} vanishes, and the result is obvious. 
%=\\-\int_{\Omega} \varphi \textrm{div}_x( \int_Y  \langle A(x,y) [\g_x u+\g_y u_1])- \int_Y \varphi_1 \textrm{div}_Y (\g_x \varphi+\g_y \varphi_1) dy dx=\\ 
%\int f \varphi(x)
%\end{multline}
%\end{aligned}
%\end{equation}
%

\end{proof}

\begin{Prop}
There is a unique solution to the two-scaled  homogenized problem 
\end{Prop}

\begin{proof}
Uniqueness and existence follows from the classical Lax-Milgram theorem applied to the bilinear form $B$ in the space  $H_0^{1}(\Omega) \times L^2(\Omega,\Hper(\TM_p))$. Notice that
\[ \int_{\pi^{-1}\Omega} \, \langle \g \varphi, \g \varphi_1 \rangle \, dv\, dp=0 \]
 
 Next, by the ellipticity of $A$ it holds that for $\Phi=(\varphi,\varphi_1)$

\[ B(\Phi,\Phi) \ge c \int_{\pi^{-1}\Omega}  |\g \varphi+\g \varphi_1|^2 \,dv\,dp \ge c\cdot V \int_{\Omega} | \g \varphi|^2 + c\int_{\pi^{-1}\Omega}|\g \varphi_1|^2 ,
\]
where $V=\max_p{\vol(\TM_p)}$.
\end{proof}

\subsection{The homogenized system}
We recall the definition of the homogenized problem for arbitrary  metric $g$ and frame $\Gamma$.
\begin{Def}[The homogenized problem]
We define a symmetric tensor $A^* \in \mathcalT^{1,1}(M)$ 
\label{appendix:homoweakapp}
\begin{equation} \label {homatrixapp}
A^*[p][X,Y]=\avint_{\T M_p} \langle(I+ (D_vw)^{ad})A[p,v](I+ (D_vw))X^\up dv,Y^\up\rangle
\end{equation}
where  $(D_v w)^{ad}= G D_vw ^tG^{-1}$     for the matrix $D_vw$ with $i$-column    $\g_vw_i$.

\end{Def}

\begin{Def}
We say that $u \in H^1_0(\Omega)$ is a solution of the homogenized system if 

\begin{equation}\label{eq:homogeneized_system}
-\div A^*_p(\g u)=f
\end{equation}

\end{Def}

For later uses we notice that equation \eqref{eq:w_i-equation} implies that for every $\varphi_1 \in \Hper(\TM_p)$ it holds that

\begin{equation}\label{weakyappendix}
\int_{\TM_p}\, \langle A[p,v] (o_k^\up+\g_v w_k), \g_v \varphi_1\,\rangle \,dv=0
\end{equation}
and similarly, from equation \ref{eq:homogeneized_system}, for every $\varphi \in H_0^1(\Omega)$ it holds that
\begin{equation}\label{weakxapendix}
\int_{\Omega} \,\langle A^* \g u\,,\,\g \varphi\, \rangle\, dp=\int_\Omega f\cdot\varphi\, dp
\end{equation}
 
Existence and uniqueness of the $w_i$'s and $u$ for both the cell problem and the homogenized system, follows from the Lax-Milgram theorem applied to each of the corresponding 
bilinear forms.

\end{document}